\documentclass[11pt,xcolor=pdftex,a4paper,oneside,centertags,reqno]{amsart}

\usepackage{latexsym}
\usepackage{amsmath}
\usepackage{amsfonts}
\usepackage{amssymb}
\usepackage{graphpap}
\usepackage{epsfig}
\usepackage{amscd}
\usepackage{amsbsy}
\usepackage{amsthm}
\usepackage{enumerate}
\usepackage{eucal}
\usepackage[usenames,dvipsnames]{color}
\usepackage{tikz}
\usepackage{pict2e,color}
\numberwithin{equation}{section} 

\usepackage[francais]{[babel}
\usepackage[deutsch]{[babel}
\usepackage[ansinew]{inputenc}
\usepackage{graphics}
\usepackage{graphicx}

\usepackage{hyperref}
\usepackage{pdfsync}
\usepackage{color}
\usepackage{endnotes}
\usepackage{mathrsfs}

\font\got=eufm10 at 11pt

\font\posebni=msam10

\font\sevcyi=wncyi7
\font\eigcyi=wncyi8
\font\tencyr=wncyr10
\font\tencyb=wncyb10
\font\tencyi=wncyi10
\font\twlcyr=wncyr10 at 11pt
\font\twlcyb=wncyb10 at 11pt
\font\twlcyi=wncyi10 at 11pt

\providecommand{\Img}{\mathop{\rm Im}\nolimits}
\providecommand{\Ker}{\mathop{\rm Ker}\nolimits}
\providecommand{\Real}{\mathop{\rm Re}\nolimits}
\providecommand{\supp}{\mathop{\rm supp}\nolimits}

\newcommand{\C}[0]{{\mathbb C}}

\newcommand{\N}[0]{{\mathbb N}}

\newcommand{\R}[0]{\mathbb{R}}

\newcommand{\Z}[0]{{\mathbb Z}}
\newcommand{\EP}[0]{\mathbb{E}\mathcal{P}}

\newcommand\kA{\mathbf A}

\newcommand\kT{\mathbf T}

\newcommand{\cB}[0]{{\mathcal B}}
\newcommand{\cC}[0]{{\mathcal C}}
\newcommand{\cD}[0]{{\mathcal D}}

\newcommand{\cF}[0]{{\mathcal F}}

\newcommand{\cH}[0]{{\mathcal H}}

\newcommand{\cK}[0]{{\mathcal K}}
\newcommand{\cL}[0]{{\mathcal L}}
\newcommand{\cM}[0]{{\mathcal M}}
\newcommand{\cN}[0]{{\mathcal N}}
\newcommand{\cO}[0]{{\mathcal O}}
\newcommand{\cP}[0]{{\mathcal P}}
\newcommand{\cR}[0]{{\mathcal R}}
\newcommand{\cS}[0]{{\mathcal S}}

\newcommand{\cZ}[0]{{\mathcal Z}}

\newcommand{\leqsim}[0]{\,\text{\posebni \char46}\,}
\newcommand{\geqsim}[0]{\,\text{\posebni \char38}\,}

\newcommand{\e}[0]{\varepsilon}
\newcommand{\f}[0]{\varphi}

\newcommand{\wh}[0]{\widehat}
\newcommand{\pd}[0]{\partial}

\newcommand{\nor}[1]{\| #1 \|}

\newcommand{\Nor}[1]{\left |\hskip -1.5pt\left | #1 \right |\hskip -1.5pt\right |}

\newcommand{\av}[1]{\langle #1 \rangle}
\newcommand{\avg}[1]{\langle #1 \rangle}
\newcommand{\bavg}[1]{\left\langle #1 \right\rangle}

\newcommand{\mn}[2]{\{ #1\, ;\, #2 \}}
\newcommand{\Mn}[2]{\left\{ #1 ; #2 \right\}}
\newcommand{\sk}[2]{\langle #1 , #2\rangle}
\newcommand{\Sk}[2]{\left\langle #1 , #2\right\rangle}

\newcommand{\Do}{\noindent {\bf Proof.}\ }

\newtheorem{theorem}{Theorem}[section]
\newtheorem{exercise}[theorem]{\textcolor[rgb]{0.87,0.1,0.1}{Exercise}}
\newtheorem{proposition}[theorem]{Proposition}

\newtheorem{lemma}[theorem]{Lemma}
\newtheorem{corollary}[theorem]{Corollary}
\newtheorem{conjecture}[theorem]{Conjecture}
\renewcommand\leq[0]{\leqslant}
\renewcommand\geq[0]{\geqslant}
\renewcommand\epsilon[0]{\varepsilon}
\renewcommand\theta[0]{\vartheta}

\newtheorem{predefinition}[theorem]{Definition}
\newenvironment{definition}%
{\begin{predefinition}\rm}{\end{predefinition}}

\newcounter{magnolia blues}

\newtheorem{prequestion}[magnolia blues]{Question}  \newenvironment{question}%
{\begin{prequestion}\rm}{\end{prequestion}}

\newtheorem{preremark}[theorem]{Remark}  \newenvironment{remark}%
{\begin{preremark}\rm}{\end{preremark}}
\newtheorem{preexample}[theorem]{Example}  \newenvironment{example}%
{\begin{preexample}\rm}{\end{preexample}}

\begin{document}

\title[Ahlfors-Beurling transform]
{Analysis of the Ahlfors-Beurling transform\\ {\it Lecture notes for the summer school at the University of Seville, September 9-13, 2013}}

\author[Dragi\v{c}evi\'c]{Oliver Dragi\v{c}evi\'c}


\address{Oliver Dragi\v{c}evi\'c \\Department of Mathematics\\ Faculty of Mathematics and Physics\\ University of Ljubljana\\ Jadranska 21, SI-1000 Ljubljana\\ Slovenia}
\email{oliver.dragicevic@fmf.uni-lj.si}

\maketitle


\tableofcontents

\section*{Prologue}

These notes are based on the ten lectures their author held in September of 2013 at the University of Seville. The purpose of the course was to introduce basic concepts about Ahlfors-Beurling operator $T$ and explain a few of the recent (already published or known) results and techniques focused around it. 
The intended level of the course was primarily that for beginner graduate students in analysis.
In accordance with that, in this text we tried to give proofs of most of the statements {\it  characteristic}
of $T$, paying attention to details and leaving out as exercises a few of the less demanding proofs. Facts that hold for much larger classes of operators have for the most part only been cited or again left as exercises. 
Whereever possible, we gave references for statements that were not proven.

This text was by no means designed as a comprehensive survey of all the important results regarding the subject. 
In particular, the Bellman-function-heat-flow method that was the basis of the works \cite{PV,NV,DV2} is not presented here. 
Furthermore, the stochastic approach that has been utilized by Ba\~nuelos et al. is not discussed here either. It can be found, for example, in \cite{B} where many other interesting issues related to the Ahlfors-Beurling operator are treated. A brief recount of the probabilistic approach to the Ahlfors-Beurling operator can also be found in \cite[Section 4]{V2}. For a very interesting paper which combines both techniques see \cite{BJV}. Finally, there has recently been a revival of interest in finding sharp weighted estimates for general Calder\'on-Zygmund operators. It culminated in the work by Hyt\"onen \cite{Hy1} where the complete solution to the problem was described for the first time. While many results preceeding and following the resolution of the conjecture are beautiful and interesting for their own sake, they would exceed the scope of this note and are not included either.

\medskip\noindent{\bf Organization.} 
The bulk of this text is devoted to estimates of $T$ and its powers $T^n$ on the usual $L^p$ as well as $L^p(w)$ with $w$ from the so-called Muckenhoupt class $A_p$. We will also briefly discuss spectral theory for $T$. 
In Section \ref{namp} we introduce the principal objects we work with and state some of the most basic facts about them.
Section \ref{motivati} provides background for the results on $T$ we consider in this text. In particular, we give a very brief introduction to the theory of quasiconformal maps and explain the connection with weighted and ``unweighted'' estimates for $T$ and its powers. For the most recent and very thorough treatise of quasiconformal theory on $\C$ the reader is advised to consider the monograph by Astala, Iwaniec and Martin \cite{AIM}. 
Sections \ref{Evtushenko} and \ref{nas na babu promenyal} are mainly devoted to proofs of results announced in Section \ref{motivati}. Several different techniques will be treated in the process: most prominently (i) the Bellman function technique (though without direct association with heat flows) 
and (ii) the averaging method. Finally, in Section \ref{george szell} we address some spectral properties of $T$.

\subsection*{Acknowledgements}

I am most grateful to Carlos P\'erez for the generous opportunity to present one of my areas of research in the form of a 10-hour course. Further gratitude goes to the students and researchers who attended the course -- for their interest -- and to the University of Seville for its kind hospitality.

\vskip 15pt
\ \hfill Novo mesto, March 2014

\newpage

\section{Notation and main protagonists}
\label{namp}

Given two quantities $A$ and $B$, we adopt the convention whereby $A\leqsim B$ means that there exists an absolute constant $C>0$ such that $A\leqslant CB$. If both $A\leqsim B$ and $B\leqsim A$, then we write $A\sim B$. If $\{\lambda_1,\dots,\lambda_n\}$ is a set of parameters, $C(\lambda_1,\dots,\lambda_n)$ denotes a constant depending only on 
$\lambda_1,\dots,\lambda_n$. When $A\leq C(\lambda_1,\dots,\lambda_n)B$, we will often write $A\leqsim_{\lambda_1,\dots,\lambda_n} B$.

For $R>0$ and $z\in\C$ we will write $B(z,R)=\mn{w\in\C}{|z-w|<R}$.
The open unit disc $B(0,1)$ will also be denoted by $\Delta$.

For any $x=(x_1,x_2)\in\R^2$ and $p>0$ denote
$$
|x|_p=(|x_1|^p+|x_2|^p)^{1/p}
\hskip 40pt  
\text{and}
\hskip 40pt
|x|_\infty=\max\{|x_1|,|x_2|\}.
$$

Given $\psi\in\R$, introduce the rotation matrix on $\R^2$,
$$
\cO_\psi=\left[ {\begin{array}{rr}
  \cos\psi&-\sin\psi\\
  \sin\psi&\cos\psi
 \end{array} } \right].
$$
For a function $f$ on $\R^2$ we will introduce three basic families of transformations ($x\in\R^2$): 
\begin{equation}
\label{prokofjev}
\begin{array}{rllc}
\text{translations}:\ & {\displaystyle(\tau_t f)(x)=f(x-t)} & ; & t\in\R^2;\\
\text{dilations}:\ & {\displaystyle(\delta_a f)(x)=f(x/a)} & ; & a>0\,;\\
\text{rotations}:\ & {\displaystyle(U_\psi f)(x)=f(\cO_{-\psi} x)} & ; & \psi\in\R.
\end{array}
\end{equation}

 Function $\Omega:\C\backslash\{0\}\rightarrow\C$ which for any $r>0$ and $\zeta\in\pd\Delta$ satisfies 
 \begin{itemize}
 \item
 $\Omega(r\zeta)=\Omega(\zeta)$ \hskip 5.5pt
is called  {\it homogeneous of degree zero};

\item
 $\Omega(r\zeta)=\Omega(r)$ \hskip 5.7pt is called {\it radial};

\item
 $\Omega(r\zeta)=\Omega(r)\zeta$ is called a {\it (radial) stretch function} if, in addition, $\Omega|_{(0,\infty)}$ is  sctrictly increasing, continuous and extends continuously to zero \cite[p. 28]{AIM}.
\end{itemize}

By $m(E)$ or $|E|$ we shall denote the Lebesgue measure of a Borel set $E\subset\C$, while $dx$ in the integrals means just $dm(x)$. Sometimes even $dx$ will be omitted. We will use 
the standard pairing
\begin{equation}
\label{si mislyat}
\sk\f\psi=\int_{\R^n}\f\bar\psi,
\end{equation}
where 
$\f,\psi$ are complex functions on $\R^n$ such that the above integral makes sense.

For complex $C^1$ functions of two real variables $x,y$ denote
$$
\pd =\frac12(\pd_x-i\pd_y)\,\hskip 30pt {\rm and}\hskip 30pt \bar\pd =\frac12(\pd_x+i\pd_y)
\,.
$$
We may instead of $\pd f,\bar\pd f$ write $\pd f/\pd z, \pd f/\pd \bar z$ or $\pd_z f,\pd_{\bar z}f$ or simply $f_z,f_{\bar z}$.

Let ${\mathcal S}=\cS(\R^d)$ be the Schwartz class on $\R^d$. Recall that its topology is given by a family of seminorms
$$
\rho_{\alpha,\beta}(f)=\sup_{x\in\R^d}|x^\alpha\pd^\beta f(x)|\,,
$$
with $\alpha,\beta$ ranging over all $\N_0^d$; here $\N_0:=\N\cup\{0\}$. Here 
$$
(x_1,\hdots,x_d)^{(\alpha_1,\hdots,\alpha_d)}:=x_1^{\alpha_1},\hdots,x_d^{\alpha_d}
\,\hskip 30pt {\rm and}\hskip 30pt
\pd^{(\beta_1,\hdots,\beta_d)}:=\pd_{x_1}^{\beta_1}\hdots\pd_{x_d}^{\beta_d}\,.
$$

For $\f\in{\mathcal S}(\R^d)$ and $\xi\in\R^d$ define the {\it Fourier transform} $\widehat\f$ of $\f$ by
$$
\widehat\f(\xi)=\int_{\R^d}\f(x)e^{-2\pi i\sk{x}{\xi}}dx.
$$

We will throughout this note work with the obvious identification of $\C$ with $\R^2$ via $x+iy\equiv (x,y)$. Accordingly, a function on $\C$ will also be thought of as a function on $\R^2$ and vice versa.
An important notation to bear in mind will be 
$$
p^*=\max\Big\{p,\frac p{p-1}\Big\}\,.
$$

Among the standard tools we shall frequently need is the {\it 
Minkowski's integral inequality}, e.g. Stein \cite[\S A.1]{S}, Grafakos \cite[Exercise 1.1.6]{G1}, Duoandikoetxea \cite[p. xviii]{Du}: if $1\leq p<\infty$ and $F$ is measurable on $X\times Y$, where $(X,\mu)$ and $(Y,\nu)$ are $\sigma$-finite measure spaces, then
$$
\left(\int_Y\bigg|\int_XF(x,y)\,d\mu(x)\bigg|^p\,d\nu(y)\right)^{1/p}
\leq
\int_X\left(\int_Y|F(x,y)|^p\,d\nu(y)\right)^{1/p}\,d\mu(x)\,.
$$

In the context of operators, the asterisk $^*$ will typically denote their adjoints. 
Thus if $A\in B(L^p)$ then $A^*\in B(L^q)$, where $1/p+1/q=1$, and $\sk{Af}{g}=\sk{f}{A^*g}$ for all $f\in L^p$, $g\in L^q$.

Finally, for a linear operator $A$ on a vector space $X$ we will define its {\it null space} (or {\it kernel}) $\cN(A)$ and {\it range} (or {\it image}) $\cR(A)$ by
$$
\aligned
\cN(A) &:=\mn{x\in X}{Ax=0};\\
\cR(A) &:=\mn{Ax}{x\in X}\,.
\endaligned
$$

Recall the following basic facts.

\begin{exercise}
\label{alban}
Let $X$ be a Banach space and $A\in B(X)$. Suppose that both $\nor{Af}\,\geqsim\,\nor{f}$ for all $f\in X$ and $\cR(A)$ is dense in $X$. Then $A$ is surjective, i.e. $\cR(A)=X$. 
\end{exercise}

\begin{exercise}
\label{berg}
Let $X$ be a Banach space and $A\in B(X)$. Then $\cR(A)$ is dense if and only if $\cN(A^*)=\{0\}$.
\end{exercise}

\begin{exercise}
\label{filter}
Suppose $X$ is a $\sigma$-finite measure space and 
$$
A\in\bigcap_{p>1}L^p(X)\,.
$$
Then for every $p>1$ we have $\log\nor{A}_p=\f(1/p)$, where $\f:(0,1)\rightarrow\R_+$ is a convex function.
Consequently, function $p\mapsto\nor{A}_p$ is continuous on $(1,\infty)$.
\end{exercise}

\begin{exercise}
\label{blgaria}
The set of all $C_c^\infty$ functions on $\R^n$ whose average (i.e. integral over $\R^n$) is zero is dense in $L^p(\R^n)$ for any $1<p<\infty$. What about $p=1$?
\end{exercise}

\begin{exercise}
\label{ging heut morgen uebers feld}
Suppose $X$ is a locally compact Hausdorff space and $\mu$ a regular measure on $X$. Take $1<p<\infty$ and a set $\Psi:=\mn{\psi_\alpha}{0<\alpha<1}\subset L^2(\mu)\cap L^p(\mu)$. If $\Psi$ is bounded in $L^p$ and $\psi_\alpha\rightarrow 0$ in $L^2$ as $\alpha\rightarrow 0$, then also $\psi_\alpha\rightarrow 0$ weakly in $L^p$ as $\alpha\rightarrow 0$.
\end{exercise}

\begin{exercise}
\label{Thomas Quasthoff}
Suppose $K$ is a closed convex subset of a Hilbert space $\cH$. Prove that for any $h\in\cH$ there exists a unique $k\in K$ such that $d(h,K)=\nor{h-k}$. Is the map $\cH\rightarrow K$, defined by $h\mapsto k$, continuous, uniformly continuous, a contraction?
\end{exercise}

\subsection{Weak derivatives and Sobolev spaces}

There are plenty of sources on this most fundamental class of function spaces. Let us mention only
Gilbarg--Trudinger \cite[Chapter 7]{GT}, Evans \cite[Chapter 5]{E}, Stein \cite[V.\S 2]{S}, H\"ormander \cite[Section 7.9]{Hor} and Grafakos \cite[Section 6.2]{G2}. 

Suppose $ U\subset\R^n$ is an open set. 
Let $u\in L^1_{loc}( U)$. Function $v\in L^1_{loc}( U)$ is called the $\alpha^{\text{th}}$ {\it weak} or {\it distributional derivative} of $u$ provided that
$$
\int_ U u\,D^\alpha\f
=(-1)^{|\alpha|}\int_ U v\,\f
\hskip 40pt
\forall \f\in C_c^{|\alpha|}( U)\,.
$$
This notion is well-defined, i.e. if the $\alpha^{\text{th}}$  weak derivative of a function exists, it is uniquely determined up to a set of measure zero. Its {\it order} is by definition $|\alpha|$.  
We write $v=D^\alpha u$. 

We say that $u$ is
\begin{itemize}
\item
{\it weakly differentiable}, if all of its 
first-order weak derivatives exist;
\item
{\it $k$-times weakly differentiable}, if all of its derivatives of orders $1,\hdots,k$ exist.
\end{itemize}
The vector space of all $k$-times weakly differentiable functions on $ U$ is denoted by $W^k( U)$.
Clearly, $C^k( U)\subset W^k( U)$.

\begin{exercise}
\label{fop dop}
Suppose $D$ is a homogeneous first-order partial differential operator on $\R^n$, i.e.
$$
D=\sum_{j=1}^na_j\frac{\pd}{\pd x_j}
$$
for some $a_1,\hdots,a_n\in\C$. If $h\in W^1( U)$ and $g\in C^\infty( U)$ then $h\cdot g\in W^1( U)$ and 
$$
D(h\cdot g)=Dh\cdot g +h\cdot Dg\,.
$$
\end{exercise}

Next result is known as (the holomorphic version of) {\it Weyl's lemma}. See, for example, \cite[Lemma A.6.10]{AIM},  \cite[Theorem 4.1.6]{Hu} or \cite[p. 45]{A1}.

\begin{theorem}
\label{Weyl}
Suppose $ U\subset\C$ is open and $g\in L^1_{loc}( U)\cap W^1( U)$. If $g_{\bar z}=0$ weakly, then also $g_{\bar z}=0$ strongly, in the sense that $g$ coincides p.p. $ U$ with a holomorphic function.
\end{theorem}

\begin{exercise}
\label{useful}
Suppose $p\geq1$ and $h\in L^p(\C)$ is holomorphic (entire). Then $h\equiv 0$.
\end{exercise}

\begin{remark}
Since $L^p\subset L^1_{loc}$, Weyl's lemma implies that Exercise \ref{useful} holds even if we only assume that $h\in L^p(\C)$ is weakly holomorphic, the latter meaning that $h_{\bar z}=0$ weakly.
\end{remark}

\begin{definition}
For $k\in\N_0$ and $p\geq1$ the corresponding {\it Sobolev space} is defined as
$$
W^{k,p}( U):=\mn{u\in W^k( U)}{D^\alpha u\in L^p( U) \text{ for all }|\alpha|\leq k}\,.
$$
\end{definition}

These are clearly vector spaces. On $W^{k,p}( U)$ we define the norm
$$
\nor{u}_{W^{k,p}( U)}:=\left(\int_ U\sum_{|\alpha|\leq k}|D^\alpha u|^p\right)^{1/p}\,,
$$
which makes it a Banach space.
Because of the equivalence of the euclidean norms we see that
$$
\nor{u}_{W^{k,p}( U)}\,\sim\,\sum_{|\alpha|\leq k}\nor{D^\alpha u}_p\,,
$$
where the implied constants depend on $k,p$.

The special case $p=2$ merits its own notation, namely $H^k( U):=W^{k,2}( U)$. Thus in particular $H^1( U)=W^{1,2}( U)$ and
\begin{equation}
\label{courage}
\nor{u}_{H^1( U)}^2=\nor{u}_2^2+\nor{\nabla u}_2^2\,.
\end{equation}
By $H_0^1( U)$ we denote the closure of  $C_c^\infty( U)$ in $H^1( U)$.

\medskip
The following is a special case of an essential result known as the {\it Sobolev embedding theorem}, see Stein \cite[V.2.2]{S}. It immediately follows from another fundamental theorem, the so-called {\it Gagliardo-Nirenberg-Sobolev inequality}, see 
Evans \cite[5.6.1]{E}, for example.

\begin{theorem}
\label{Johhny Shines}
If $1<p<n$ then $W^{1,p}(\R^n)\subset L^{p'}(\R^n)$, where
$$
\frac1{p'}=\frac1{p}-\frac1n\,.
$$
\end{theorem}

Let us also introduce {\it local Sobolev spaces}, by setting \cite[p. 646]{AIM}
$$
W^{k,p}_{loc}(U):=\bigcap_{U'}W^{k,p}(U')\,,
$$
where $U\subset\C$ and $U'$ runs over all relatively compact subsets of $U$.

\begin{definition}
Suppose $f\in W^1( U)$ for some open set $ U\in\C$. If $f=u+iv$, $u,v$ real, then the {\it Jacobian (determinant)} of $f$ is defined as
$$
J_f:=
\bigg|
\begin{array}{ll}
u_x & u_y\\
v_x & v_y
\end{array}
\bigg|\,.
$$
\end{definition}
It is immediate that
\begin{equation}
\label{jacobi}
J_f=|f_z|^2-|f_{\bar z}|^2\,.
\end{equation}

\begin{exercise}
\label{left foot}
If $\f\in C^\infty(\C)$ with $\Re\f$ or $\Im\f$ compactly supported, and $\e\geq0$, then
$$
\int_{\{\f\ne0\}} |\f|^\e J_\f=0\,.
$$
Does the statement hold even for $\e>-1$?
\end{exercise}

We shall be only interested in functions $f$ with $J_f>0$, i.e. those that preserve the orientation. For a given $f$ of this kind define its {\it complex dilatation} or the {\it Beltrami coefficient} $\mu=\mu_f$ by
$$
\mu_f=\frac{f_{\bar z}}{f_z}\,.
$$
Clearly $|\mu_f|<1$ on $U$.

\subsection{Riesz transforms}
\label{riesz}
\cite[Section 4.2]{G1}
Let $\Omega$ be an integrable function on the unit sphere $S^{d-1}$ in $\R^d$ with zero average. 
The functional $W_\Omega$, defined by 
$$
W_\Omega(\f)=\,{\rm p.v.}\int_{\R^d}\frac{\Omega(y/|y|)}{|y|^d}\,\f(y)\,dy
\hskip 40pt 
\text{ for }\f\in{\mathcal S}(\R^d)\,,
$$
is a tempered distribution. We denote by $T_\Omega$ the {\it singular integral} that acts on ${\mathcal S}(\R^d)$ as a convolution with the distribution $W_\Omega$:
\begin{equation}
\label{hancock}
T_\Omega\f=\f*W_\Omega
\hskip 40pt 
\text{ for }\f\in{\mathcal S}(\R^d).
\end{equation}
References for general singular integrals are, for example, \cite{S, G1, G2, Du, MS}. 
We will be almost exclusively concerned with the case $d=2$ only.
The (first-order) {\it Riesz transforms} $R_1,R_2$ on $\R^2$ are the operators of the form \eqref{hancock} in the case of projections
$$
\Omega(\zeta_1,\zeta_2)=\frac1{2\pi}\,\zeta_j
$$
for $j=1,2$, respectively.
Explicitly,
$$
(R_j\f)(x)=\frac1{2\pi}\,{\rm p.v.}\int_{\R^2}\frac{y_j}{|y|^{3}}\,\f(x-y)\,dy\,.
$$
It is well known, see \cite[III. \S 1, eq. (8)]{S}, \cite[Proposition 4.1.14]{G1} or \cite[eq. (4.8)]{Du}, that $R_j$ is a Fourier multiplier with the symbol $-i\xi_j/|\xi|$:
for $j=1,2$ and any $\f\in{\mathcal S}(\R^2)$ we have
\begin{equation}
\label{carevac}
\widehat{R_j\f}(\xi)=-i\,\frac{\xi_j}{|\xi|}\widehat \f(\xi).
\end{equation}

As a special case of the central result on the $L^p$ boundedness of the Calder\'on-Zygmund singular integral operators, all operators $R_j$ are bounded on $L^p(\R^2)$ for $1<p<\infty$.
More precisely, for $p>1$ every $R_j$ admits an extension from $\cS$ to a bounded operator on (entire) $L^p$. This extension is denoted by the same symbol - $R_j$.

\medskip
We will much need the {\it complex Riesz transform} $R$ \cite[Section 4.2]{AIM}, sometimes also called  the ``complex Hilbert transform'' and denoted ${\bf H}_\C$ \cite{IM}, which is defined by
$$
R=R_2+iR_1,
$$
and its integer powers
\begin{equation}
\label{elegy}
R^k=(R_2+iR_1)^k.
\end{equation}
By \cite[Section 4.2]{AIM}, the convolution kernel of $R^k$ is given by
\begin{equation}
\label{coppy}
{\rm p.v.}\ \frac{\Omega_k(z/|z|)}{|z|^2}\,,
\end{equation}
where for $\zeta\in\pd\Delta$,
\begin{equation}
\label{chains and things}
\Omega_k(\zeta)=\frac{i^{|k|}|k|}{2\pi }\zeta^{-k}.
\end{equation}

On the Fourier side, by \eqref{carevac} we have
\begin{equation}
\label{dihat za ovratnik}
\widehat{R\f}(\xi)=
\frac{\overline\xi}{|\xi|}\widehat \f(\xi).
\end{equation}

\begin{exercise}
\label{tommy lee jones robert de niro}
For $\psi\in\R$ consider the rotation operator $U_\psi$ as in \eqref{prokofjev}. Prove that 
$$
R=e^{i\psi}U_{-\psi}RU_\psi.
$$

\end{exercise}

\subsection{Ahlfors-Beurling transform}
\label{to ne veter vetku klonit}
The central object of this note, the {\it Ahlfors-Beurling transform} $T$, is obtained by taking $k=2$ in \eqref{elegy}, i.e.
$$
T={\bf H}_\C^2=R^2=(R_2+iR_1)^2.
$$ 
Explicitly, for test functions $f$ we have
$$
Tf(z)=-\frac{1}{\pi}\,{\rm p.v.}
\int_{\C}\frac{f(\zeta)}{(\zeta-z)^2}\ dm(\zeta)\,.
$$
Alternatively, it can be introduced in terms of the Fourier transforms:
\begin{equation}
\label{infisa}
\widehat{Tf}(\xi)=\frac{\bar \xi}{\xi}\,\hat{f}(\xi)\,.
\end{equation}
From here and the Plancherel identity it immediately follows that $\nor{Tf}_2=\nor{f}_2$ for all $f\in L^2$.

For any $p>1$ and $f\in L^p$ define
$$
Sf:=\overline{T\overline f}\,.
$$
Since $R_1,R_2\in B(L^p)$, the same holds for $T$ and thus also for $S$.

\begin{exercise}
\label{lullaby}
Take $p\in(1,\infty)$ and let $q$ be its conjugate exponent. 
The operator $T$ is invertible on $L^p$. By a small abuse of notation (since $S,T^{-1}$ act on $L^p$, while $T^*$ acts on $L^q$) we have $S=T^{-1}=T^*$, in the sense that
\begin{enumerate}[i)]
\item
$TSf=STf=f$ for all $f\in L^p(\R^2)$;
\item
\label{A319}
$\sk{Tf}{g}=\sk{f}{Sg}$ for all $f\in L^p(\R^2)$ and $g\in L^q(\R^2)$.
\end{enumerate}
\end{exercise}

As an immediate consequence it follows that for all $p>1$, 
\begin{equation}
\label{simeon stilit}
\nor{Tf}_p\sim_p\nor{f}_p 
\hskip 20pt\text{and}\hskip 20pt  
\nor{Sf}_p\sim_p\nor{f}_p\,.
\end{equation}

\subsection{Muckenhoupt weights}

For any locally integrable function $f$ on $\C$ and any bounded set $Q\subset\C$ with positive Lebesgue measure $|Q|$, denote by $\avg f _Q$ the average of $f$ on $Q$,
$$
\avg{f}_Q=\frac{1}{|Q|}\int_Q f(x)\, dx
\,.
$$
If $w$ is a positive locally integrable function on $\C$, introduce
$$
[w]_p:=\sup_{Q\subset\R^2}\,
\avg{w}_Q\bavg{w^{-\frac{1}{p-1}}}^{p-1}_Q\,,
$$
where the supremum is taken over all squares in $\C$ regardless of their orientation.

Let $L^p(w)$ be the space of all functions $p$-integrable with respect to the weight $w$, i.e., $\int |f|^pw<\infty$. When $w\equiv 1$ we will simply write $L^p$.

From now on we assume that $w$ belongs to the Muckenhoupt class $A_p$, defined as
$$
w\in A_p\
\stackrel{def}{\Longleftrightarrow} [w]_p<\infty\,.
$$
This class is systematically discussed in the monographs by Garc\'ia-Cuerva and Rubio de Francia \cite{GC-RdF}, Stein \cite[Chapter V]{S1}, Duoandikoetxea \cite[Chapter 7]{Du}, Grafakos \cite[Chapter 9]{G2} and Torchinsky \cite[Chapter IX]{T}.

Both ${\mathcal S}$ and $C_c^\infty$ are dense in $L^p(w)$, for any $w\in A_p(\R^d)$, see \cite[Lemma 2.1]{M}.

\begin{exercise}
\label{violeta urmana}
Is the set $\mn{f\in \cS(\R^d)}{f(0)=0}$ dense 
\begin{itemize}
\item
in $L^p(w)$?
\item
in $\cS$?
\end{itemize}
\end{exercise}

The following sharp version of the Rubio-de-Francia extrapolation theorem was proven in \cite{DGPP}:

\begin{theorem}
\label{extrapol}
If an operator $T$ satisfies the estimates
$$
\nor{T}_{\cB(L^2(w))}\leq C[w]_{A_2}^\tau
$$
for some $C,\tau>0$ and all $w\in A_2$, then it also satisfies the estimates
$$
\nor{T}_{\cB(L^p(w))}\leq c_pC[w]_{A_p}^{(p^*/p)\tau}
$$
for some $c_p>0$ and all $p\in(1,\infty)$ and $w\in A_p$.
\end{theorem}

We conclude this introductory section with a terminological remark. Most of this text is dedicated to 
``weighted estimates" (of the operator $A$) or the ``unweighted estimates of $A$". Let us explain what we mean by this.
In the former case 
we study $A$ as an operator on $L^p(w)$ and express its norm estimates in terms of $[w]_{A_p}$, while in the latter we
study $A$ on $L^p$ (the case $w\equiv 1$), but typically ask for more precise information on the dependence of the norms on $p$.

\section{Motivation}
\label{motivati}

An important question in mathematical analysis is how mappings transform shape and size of objects. To make this question precise let us introduce two notions of distortion. Take an open $\Omega\subset\C$ and suppose $f:\Omega\rightarrow\C$ is an injective continuous function. For any $z\in\Omega$ 
define the {\it infinitesimal distortion} of $f$ at $z$ as \cite[Section 2.4]{AIM}
$$
H_f(z):=\limsup_{r\rightarrow 0}\frac{\max_{|h|=r}|f(z+h)-f(z)|}{\,\min_{|h|=r}|f(z+h)-f(z)|}\,.
$$
We are interested in functions of ``bounded distortion'', i.e. such $f$ for which $H_f<\infty$ uniformly on $\Omega$.
If $f$ is differentiable at $z$ with $J_f(z)>0$ then $H_f(z)$ can be simply expressed in terms of the Beltrami coefficient \cite[(2.27)]{AIM}. It turns out that finding functions of ``bounded distortion'' corresponds to finding functions with prescribed Beltrami coefficients. We will show how this problem naturally leads to the appearance of the Ahlfors-Beurling operator $T$.

Another concept that underlines much of what is discussed here is {\it area distortion} \cite[Section 13.1]{AIM}, by which we somewhat loosely mean how to control $|f(E)|$ in terms of $|E|$ for all Borel measurable $E$ with finite Lebesgue measure. 

\medskip
As a historical prelude let us consider a special case in which $H_f\equiv1$.

\subsection{Some estimates concerning conformal mappings}
\label{leeves}

The main source for this subsection is  \cite{CG}.
For the sake of the reader's convenience, we nevertheless chose to present complete proofs of most of the results in this subsection.

\begin{definition}
Suppose $\Omega\subset\C$ is an open set. We say that a function $f:\Omega\rightarrow\C$ is {\it conformal} on $\Omega$ if it is holomorphic and injective. Such a function is also called {\it univalent}. Univalent functions $f$ on $\Delta$ which are normalized by $f(0)=0$ and $f'(0)=1$ are called {\it schlicht}.
\end{definition}

\begin{remark}
Often, e.g. Rudin \cite[Definition 14.1]{R1}, conformality of $f$ is defined by requiring that $f$ preserve angles (both in terms of size and orientation). This is equivalent to $f$ being holomorphic and that $f'\ne 0$ everywhere on $\Omega$ \cite[Theorem 14.2]{R1}, which is in turn equivalent to $f$ being holomorphic and {\it locally} injective \cite[Theorems 10.30 and 10.33]{R1}. This however still does not suffice for {\it global} injectivity that is required in our definition, the 
example being the exponential function. 
\end{remark}

Our first result on univalent functions is due to Gronwall. The proofs can be found in \cite[Theorem I.1.1]{CG} or \cite[Theorem 14.13]{R1}.

\begin{theorem}[area theorem]
Suppose $g$ is univalent on $\Omega:=\Delta\backslash\{0\}$ and has there the Laurent series expansion
$$
g(z)=\frac1z+\sum_{n=0}^\infty b_nz^n\,.
$$
Then
$$
\sum_{n=1}^\infty n|b_n|^2\leq 1\,.
$$
\end{theorem}

\begin{proof}
Fix $0<r<1$ and define $\Delta_r=\mn{z\in\Delta}{0<|z|<r}$ and $D_r:=\C\backslash g(\Delta_r)$. 
By the open mapping theorem \cite[Chapter 10]{R1}, the set $D_r$ is closed. 
Take $w\in\Omega\backslash\Delta_r$ and write  $\alpha=g(w)$. Since $g$ was assumed to be injective, the function $z\mapsto g(z)-\alpha$ has no zeros on $\Delta_r$, therefore we may define $G:\Delta_r\rightarrow \C$ by
$$
G(z):=\frac1{g(z)-\alpha}=\frac z{1+zh(z)}\,,
$$
where
$$
h(z)=b_0-\alpha+\sum_{n=1}^\infty b_nz^n\,.
$$
Function $G$ thus has a removable singularity at zero. We still denote its holomorphic extension to $\Delta_r\cup\{0\}$ by $G$. By the open mapping theorem and since $G(0)=0$, there exists $\delta>0$ such that $B(0,\delta)\subset G(\Delta_r\cup\{0\})$. 
Suppose $u\in\C$ satisfies $|u|>1/\delta$. Then $0<|1/u|<\delta$, therefore $1/u\in G(\Delta_r)$, which means that $u\in g(\Delta_r)-\alpha$. Consequently, $D_r\subset\overline B(a,1/\delta)$, so $D_r$ is also bounded.

It can now be shown (e.g. by using the Green's formula) 
that
\begin{equation}
\label{ludi svit}
\text{area}(D_r)=\frac1{2i}\int_{\pd D_r}\bar z\,dz\,.
\end{equation}
Let us prove that, as sets, $\pd D_r$ and $g(r S^1)$ coincide. 
We have $\pd D_r=\pd\big(\C\backslash g(\Delta_r)\big)=
\pd g(\Delta_r)$. 
We know that $g(\Delta_r)$ is an open set, therefore 
$\pd g(\Delta_r)=\overline{g(\Delta_r)}\backslash g(\Delta_r)$.
Since $g$ is continuous, $g(\overline{\Delta_r })\subset \overline{g(\Delta_r)}$. 
Function $g$ is injective, therefore it has an inverse, $g^{-1}$. By a basic 
theorem \cite[Theorem 10.30]{R1}, $g^{-1}$ is again holomorphic, thus continuous. Consequently, $g^{-1}\big(\overline{g(\Delta_r)}\big)\subset \overline{ g^{-1}(g({\Delta_r }))}=\overline{ \Delta_r }$, which means 
$\overline{g(\Delta_r)}\subset g\big(\overline{ \Delta_r }\big)$. Therefore, in a set-theoretical sense,
$\pd g(\Delta_r)=g\big(\overline{ \Delta_r }\big)\backslash g(\Delta_r)=g(r S^1)$. 

Since $g$ is conformal, it preserves the orientation of {\it  angles}. 
As for
{\it curves},
it reverts the orientation of $rS^1$, in the following sense: if $\omega_r$ is the curve $rS^1$ oriented counterclockwise and $a\in D_r$, then $\text{Ind}_{g\circ\omega_r}(a)=-1$ ({\bf exercise}).  

To summarize, one admissible parametrization for the integral in \eqref{ludi svit} is $z=g(re^{-it})$, $t\in[0,2\pi)$. 
We emphasize that the negative value of the winding number above amounts to taking a minus in the exponent. 
With this choice (and with replacing $t$ by $2\pi-t$ after the first parametrization) one eventually gets
$$
\text{area}(D_r)=-\frac r2\int_0^{2\pi}\overline{g(re^{it})}\,g'(re^{it})\,e^{it}\,dt\,.
$$

Clearly
$$
g'(z)=-\frac1{z^2}+\sum_{n=1}^\infty b_nz^{n-1}\,.
$$
Expand the integrand above into a double infinite sum of trigonometric monomials, integrate term-by-term (we may do that since the double sum is absolutely integrable) and use that
$$
\int_0^{2\pi}e^{int}\,dt=\left\{
\begin{array}{lcl}
2\pi & ; & n=0\\
0 & ; & n\ne0\,.
\end{array}
\right.
$$
The outcome is
$$
\text{area}(D_r)=\pi\bigg(\frac1{r^2}-\sum_{k=1}^\infty k|b_k|^2r^{2k}\bigg)\,.
$$
Since the area is nonnegative, we proved
$$
\sum_{k=1}^\infty k|b_k|^2r^{2k}\leq\frac1{r^2}
$$
for all $0<r<1$. Finally send $r\nearrow 1$. 
\end{proof}

\begin{theorem}[Bieberbach]
\label{Bieberbach}
Suppose 
$$
f(z)=z+\sum_{n=2}^\infty a_nz^n
$$
is a Taylor series representation of a {\it schlicht} function. Then $|a_2|\leq2$.
\end{theorem}

\begin{proof}
Define
$$
h(z)=\frac{f(z)}z=1+\sum_{n=1}^\infty a_{n+1}z^n\,.
$$
Since $f$ is by assumption injective, $0$ is its only zero, therefore $h$ has no zeros on $\Delta$. Consequently, there exists $\f\in\cH(\Delta)$ such that $h=\f^2$ and $\f(0)=1$, e.g. \cite[Theorem 13.11.(j)]{R1}. 
Define also
 $$
 \psi(z):=z\f(z^2)\,.
 $$
Then 
 \begin{equation}
 \label{fen}
 \psi^2(z)=z^2\f^2\big(z^2\big)=z^2 h\big(z^2\big)=f\big(z^2\big)\,,
 \end{equation}
 i.e. 
$\psi:\Delta\rightarrow\C$ is a holomorphic square root of the function $z\mapsto f(z^2)$. 

We have, for $w\in\Delta$, 
$$
\frac1{\f(w)}=1+\sum_{k=1}^\infty b_{k}w^k\,,
$$
where
$$
b_1=\bigg(\frac1\f\bigg)'(0)=-\frac{\f'(0)}{\f(0)^2}=-\f'(0)\,.
$$
From $h=\f^2$ it follows that $h'=2\f\f'$, so $\f'(0)=h'(0)/2=a_2/2$, therefore $b_1=-a_2/2$. 
This means that we can define and expand function $g$ as
\begin{equation}
\label{Toshiba hard disc}
g(z):= \frac1{\psi(z)}=
\frac1{z\f(z^2)}=\frac1z-\frac{a_2}2\,z+\sum_{k=2}^\infty b_{k}z^{2k-1}\,.
\end{equation}

Let us show that $g$ is injective. If $g(z_1)=g(z_2)$ for some $z_1,z_2\in\Delta$, then $\psi(z_1)=\psi(z_2)$, thus $\psi^2(z_1)=\psi^2(z_2)$, i.e. $f(z_1^2)=f(z_2^2)$, by \eqref{fen}. Since $f$ was injective, $z_1^2=z_2^2$, which gives options $z_2=\pm z_1$. But $g$ is an odd function, hence $g(-z_1)=-g(z_1)$, which equals $g(z_1)$ only if $g(z_1)=0$. However $g$ has no zeros on $\Delta$, which forces $z_2=z_1$. So $g$ is indeeed injective and thus univalent. 

Now we may apply Gronwall's area theorem and conclude from \eqref{Toshiba hard disc} that 
$$
\Big|\frac{a_2}2\Big|^2+\sum_{k=2}^\infty (2k-1)|b_{k}|^{2k}\leq 1\,.
$$
In particular, $|a_2|\leq2$.
\end{proof}

Bieberbach conjectured in 1916 that, in fact, under the assumptions of Theorem \ref{Bieberbach}, $|a_n|\leq n$ for all $n\in\N$. This was confirmed many years later in a renowned work by de Branges \cite{dB}. It is easy to see that all of these estimates are optimal, the extremals being the so-called {\it Koebe functions}, cf. Exercise \ref{Koebe} below, and its rotations. 
Proposition \ref{segovia - hotel de condes de castilla} below shows that the converse of de Branges' theorem is false, in the sense that if $\f(z)=z+\sum_2^\infty b_nz^n$ is holomorphic on $\Delta$ and $|b_n|\leq n$, then $\f$ is not necessarily injective, even if all $b_n$ are strictly positive. 

For much more information about this problem see Gong \cite{Gong}.

\begin{theorem}[Koebe 1/4 theorem]
Suppose $f$ is {\it schlicht}. Then $B(0,1/4)\subset f(\Delta)$. 
\end{theorem}

\begin{proof}
Take $c\in\C$ such that $f(z)\ne c$ for any $z\in\Delta$. Define
$$
g(z):=\frac{cf(z)}{c-f(z)}=\frac{c^2}{c-f(z)}-c\,.
$$
We calculate
$$
\aligned
g'(z)&=\frac{c^2f'(z)}{\big(c-f(z)\big)^2}\\
g''(z)&=c^2\,\frac{f''(z)\big(c-f(z)\big)+2f'(z)^2}{\big(c-f(z)\big)^3}\,.
\endaligned
$$
Therefore $g(0)=0$, $g'(0)=1$ and $g''(0)=a_2+1/c$, where
$$
f(z)=z+\sum_{k=2}^\infty a_kz^k
\,.
$$

Function $g$ is clearly injective. From the Bieberbach's estimate (Theorem \ref{Bieberbach}) applied to both $f$ and $g$ we get $|a_2|\leq 2$ and $|a_2+1/c|\leq 2$, respectively. This implies $|1/c|\leq 4$, i.e. $|c|\geq 1/4$. 

We proved that $c\not\in f(\Delta)$ implies $|c|\geq 1/4$, which is of course equivalent to $|c|<1/4$ implying $c\in f(\Delta)$.
\end{proof}

The estimate from the Koebe theorem is sharp, meaning that $1/4$ appearing in the formulation cannot be replaced by a smaller number:

\begin{exercise}
\label{Koebe}
On $\Delta$ define the {\it Koebe function}
$$
K(z):=\frac{z}{(1-z)^2}=\sum_{n=1}^\infty nz^n\,.
$$
Then $K(\Delta)=\C\backslash(-\infty,-1/4]$.
\end{exercise}

\begin{corollary}
For every {\it schlicht} function $f$ the following estimate is valid:
$$
\frac14\leq d\big(0,\pd f(\Delta)\big)\leq1\,.
$$
\end{corollary}

\begin{proof}
The lower estimate follows immediately from the Koebe theorem. 

The upper estimate will follow once we prove that $\overline\Delta\not\subset f(\Delta)$. Suppose the contrary, i.e. that $\overline\Delta\subset f(\Delta)$. Then
\begin{equation}
\label{countdown}
f^{-1}(\Delta)\subset f^{-1}(\overline\Delta)\subset f^{-1}(f(\Delta))=\Delta\,.
\end{equation}
The Schwartz lemma implies $f(z)=\lambda z$ for some $|\lambda|=1$. Since $f$ is univalent, we have $f'(0)=1$, therefore $\lambda=1$, i.e. $f(z)=z$. But this contradicts $\overline\Delta\subset f(\Delta)$. 
\end{proof}

Observe that the assumption $\overline\Delta\subset f(\Delta)$ was only used at the very end, i.e. \eqref{countdown} clearly holds already if $\Delta\subset f(\Delta)$, which can actually happen (e.g. with $f=\text{id}$).

\begin{corollary}
If $f$ is univalent on $B(z,\delta)$, then
$$
d\big(f(z),\pd f(B(z,\delta))\big)\geq\frac\delta4|f'(z)|\,.
$$
\end{corollary}

\begin{proof}
Define $g:\Delta\rightarrow\C$ by 
$$
g(w)=\frac{f(\delta w+z)-f(z)}{\delta f'(z)}\,.
$$
Then $g$ is {\it schlicht}, therefore $d\big(0,\pd f(\Delta)\big)\geq 1/4$, i.e.
\[
\frac{d\big(f(z),\pd f(\delta\Delta +z)\big)}{\delta |f'(z)|}\geq\frac14\,.
\qedhere
\]
\end{proof}

\begin{theorem}
Suppose $f$ is univalent on the region $D$ and $z\in D$. Then
$$
\frac14\leq\frac1{|f'(z)|}\cdot\frac{d\big(f(z),\pd f(D)\big)}{d(z,\pd D)}\leq4\,.
$$
\end{theorem}

\begin{proof}
Take $\delta:=d(z,\pd D)$. Then $B(z,\delta)\subset D$, therefore $f\big(B(z,\delta)\big)\subset f(D)$ and consequently
$$
\aligned
d\big(f(z),\pd f(D)\big) &= d\big(f(z),f(D)^c\big) \geq d\big(f(z),f(B(z,\delta))^c\big)=d\Big(f(z),\pd\big( f(B(z,\delta))\big)\Big)\\
&\geq\frac{|f'(z)|}{4}\,d(z,\pd D)\,.
\endaligned
$$

This confirms the lower estimate. In order to get the upper one, apply the lower estimate for
$$
\aligned
z & & \leftrightsquigarrow &\ \ f(z)\ \\
f & & \leftrightsquigarrow &\ \ f^{-1}\\
D & & \leftrightsquigarrow &\ \ f(D)\,. 
\endaligned
$$
Since 
$$
1=(f^{-1}\circ f)'(z)=\big(f^{-1}\big)'(f(z))\cdot f'(z)\,,
$$
it follows that
$$
d(z,\pd D)=d\big(f^{-1}(f(z)),\pd D\big)\geq\frac14\cdot\frac1{|f'(z)|}\cdot d\big(f(z),\pd f(D)\big)\,.
$$
This finishes the proof.
\end{proof}

\begin{theorem}
If $f$ is {\it schlicht} then
$$
d\big(f(z),\pd f(\Delta)\big)>\frac1{16}\,\big(1-|z|^2\big)
$$
for all $z\in\Delta$.
\end{theorem}

\begin{proof}
See Carleson, Gamelin \cite[Theorem I.1.7]{CG}.
\end{proof}

\subsection{Beyond conformality: quasiconformal maps and the Beltrami equation}

Let us now start discussing the issues of area distortion for {\it nonholomorphic} functions. 
The simplest nonholomorphic function is probably the linear one, i.e. of the following form. Take $a,b\in\C\backslash\{0\}$ and define 
$A:\C\rightarrow\C$ by $A(u)=au+b\bar u$.

\begin{exercise}
\label{DSCH-Oktyabr'}
For $A$ as above, the following properties hold:

\begin{itemize}

\item
If $A$ is understood in a canonical way as a real $2\times2$ matrix, then 
$$\det A=|a|^2-|b|^2\,.$$

\item
When $|a|\ne|b|$, we have
$$
A^{-1}u=\frac{\bar au-b\bar u}{|a|^2-|b|^2}\,.
$$

\item
$A(\pd\Delta)$ is an ellipse whose semi-axes have lengths $|a|+|b|$ and $\big||a|-|b|\big|$. Inclination of the longer axis with respect to the positive half of the real axis is $(\arg a+\arg b)/2$.

\item
If we define $\nor{A}:=\max\mn{|Au|}{|u|=1}$ then $\nor{A}=|a|+|b|$.
\end{itemize}
\end{exercise}

We proceed with the general case.

\begin{definition}
\label{asdfasdfsadfasdf}
Let $U,V\subset\C$ be open sets. Take $K\geq1$ and denote
\begin{equation}
\label{B-A-B-G}
k:=\frac{K-1}{K+1}\in[0,1)\,.
\end{equation}
We say that a map $f:U\rightarrow V$ is {\it $K$-quasiconformal}, if
\begin{itemize}
\item
it is a homeomorphism;
\item
belongs to $W^{1,2}_{loc}(U)$;
\item
its distributional derivatives satisfy $|f_{\bar z}|\leq k|f_z|$ almost everywhere on $U$.
\end{itemize}
\end{definition}
\begin{remark}
Let us list a few quick observations:
\begin{enumerate}[*]
\item
The last condition is equivalent to $f$ solving the {\it Beltrami equation}
\begin{equation}
\label{cyrenaica}
\overline\pd f=\mu\cdot\pd f\,,
\end{equation}
for some $\mu\in L^\infty(\C)$ such that $\nor{\mu}_\infty\leq k$. 

\item
When $\mu=0$ (i.e. $K=1$) the Beltrami equation becomes the well-known Cauchy-Riemann system. 
Weyl's lemma (Theorem \ref{Weyl}) implies that any $1$-quasiconformal mapping is in fact conformal \cite[Corollary 4.1.7]{Hu}, \cite[p. 27]{AIM}.

\end{enumerate}
\end{remark}

\begin{definition}
A map $f$ is {\it quasiconformal} if it is $K$-quasiconformal for some $K\geq1$. 
The smallest $K$ with this property is called the {\it quasiconformal constant} for $f$ and is denoted by $K(f)$.
\end{definition}

It was shown by Gehring and Lehto \cite{GL} that quasiconformality of a homeomorphism is equivalent to its being of bounded distortion, in the sense of $H_f<\infty$ holding uniformly. 

\medskip
Suppose now $f$ is $K$-quasiconformal and also $f\in C^1(U)$. Then $Df(z_0)$ exists in the ordinary sense for all $z_0\in U$ and we have
$$
\big[Df(z_0)\big]u=\frac{\pd f}{\pd z}(z_0)\,u+\frac{\pd f}{\pd \bar z}(z_0)\,\bar u\,.
$$
Write $a=f_{z}(z_0)$ and $b=f_{\bar z}(z_0)$. Then the $K$-quasiconformality gives 
$|b|\leq k|a|<|a|$. 
Thus, by \eqref{jacobi} and Exercise \ref{DSCH-Oktyabr'}, $Df(z_0)$ is nondegenerate and $Jf=|a|^2-|b|^2>0$, meaning that $f$ preserves the orientation. Moreover, the eccentricity of $Df(z_0)$ is equal to
$$
\frac{|a|+|b|}{|a|-|b|}=\frac{1+|b|/|a|}{1-|b|/|a|}\leq\frac{1+k}{1-k}=K\,.
$$
Again by Exercise \ref{DSCH-Oktyabr'}, $\nor{Df}=|f_z|+|f_{\bar z}|$, therefore the above inequality could alternatively be expressed as
$$ 
\frac{\nor{Df}^2}{Jf}\leq K\,.
$$

\subsubsection{The first objective} 
We want to solve the Beltrami equation. A very prominent r\^ole in this expedition will be played by the Ahlfors-Beurling operator $T$.

\subsection{Cauchy transform}

We will often use the simple fact that 
\begin{equation}
\label{bereza stayala}
\frac1\zeta\in L^r_{loc}(\C)
\hskip 30pt
\text{for all }r<2\,.
\end{equation}

For $h\in C_c(\C)$ define its (planar) {\it Cauchy transform} by 
\begin{equation}
\label{sibir'}
(\cC h)(z):=-\frac1\pi\int_\C \frac{h(\zeta)}{\zeta-z}\,dA(\zeta)\,.
\end{equation}
By \eqref{bereza stayala} the above integral converges absolutely for $h\in C_c(\C)$. Indeed, 
$$
\int_\C \bigg|\frac{h(\zeta)}{\zeta-z}\bigg|\,dA(\zeta)\leq\nor{h}_\infty\int_{K}\frac{dA(\eta)}{|\eta|}<\infty\,.
$$
Here $K:=\supp h-z$ is a compact set.
Observe that the integral also converges absolutely under a weaker assumption of $h\in L^p$ having compact support.

\begin{exercise}
\label{mishku vypustili}
If $h\in C_c^1(\C)$ then $\cC h\in C^1(\C)$ and
$$
\aligned
\overline\pd(\cC h)&=\cC (\overline\pd h)\\
\pd(\cC h)&=\cC (\pd h)\,.
\endaligned
$$
In brief, on $C_c^1(\C)$ the operator $\cC$ commutes with both $\overline\pd$ and $\pd$.
\end{exercise}

This statement has an immediate corollary:

\begin{corollary}
\label{sviridov - svyatyj bozhe}
We have $\cC:C_c^\infty(\C)\rightarrow C^\infty(\C)$ and $\cC D=D\cC$ on $C_c^\infty(\C)$ for any linear partial differential operator $D$ on $C^\infty(\C)$ with constant coefficients.
\end{corollary}

Recall that $T$ was the Ahlfors-Beurling operator, defined in Section \ref{to ne veter vetku klonit}.

\begin{lemma}
\label{lyubo, bratcy, lyubo}
Suppose $h\in C_c^1(\C)$. Then $\cC h\in C^1$ and
$$
\aligned
(\cC h)_{\bar z}&=h\\
(\cC h)_{z}&=Th\,.
\endaligned
$$
\end{lemma}

\begin{proof}

Write $\zeta=\alpha+i\beta$ and $z=x+iy$. Take $h\in C_c^1(\C)$. In view of Exercise \ref{mishku vypustili} it suffices to evaluate integrals 
$$
\iint_\C\frac{\pd h}{\pd\bar\zeta}(\zeta)\,\frac1{\zeta-z}\,d\alpha\,d\beta
\hskip 20pt\text{ and }\hskip 20pt
\iint_\C\frac{\pd h}{\pd\zeta}(\zeta)\,\frac1{\zeta-z}\,d\alpha\,d\beta\,.
$$
We will first do it with $\pd_\alpha$ and $\pd_\beta$ in place of $\pd_{\bar\zeta}=(1/2)(\pd_\alpha+i\pd_\beta)$ and $\pd_\zeta=(1/2)(\pd_\alpha-i\pd_\beta)$. We start with
$$
\iint_\C\frac{\pd h}{\pd\alpha}(\zeta)\,\frac1{\zeta-z}\,d\alpha\,d\beta
=\lim_{\e\rightarrow 0}
\underbrace{\iint_{A_\e(z)}\frac{\pd h}{\pd\alpha}(\zeta)\,\frac1{\zeta-z}\,d\alpha\,d\beta}_{=:I_\e(z)}\,,
$$
where
$A_\e(z):=z+A(\e,1/\e)=\mn{\zeta\in\C}{\e<|\zeta-z|<1/\e}$.
\begin{figure}
\label{asdasd}
\begin{center}
\setlength{\unitlength}{1mm}
\begin{picture}(80,80)
  \put(40,40){{\color{SpringGreen}\circle*{80}}}
  \put(40,40){{\color{white}\circle*{30}}}
  \put(40,40){\circle{80}}
  \put(40,40){\circle{30}}
  \put(40,40){\line(1,2){17.9}}
  \put(40,40){\line(3,1){14.2}}
  \put(40,40){\circle*{1}}
  \put(38,37){$z$}
  \put(18,70){$A_\e(z)$}
  \put(54,64){$1/\e$}
  \put(49,45){$\e$}
\thicklines  \put(26.6,46.7){{\color{OrangeRed}\vector(2,-1){22.4}}}
  \put(26.6,46.7){\circle*{1}}
  \put(24,49){$\zeta$}
  \put(47,32.5){{\color{OrangeRed}$\nu$}}
 \end{picture}
\end{center}
\caption{}
{\protect{\label{fig0}}}
\end{figure}
We integrate by parts \cite[p.712, Theorem 2]{E}, using that for small $\e>0$ we have $h\big|_{S(z,1/\e)}\equiv 0$, since $h$ is of compact support. For $\zeta\in\pd A_\e(z)=S(z,\e)\cup S(z,1/\e)$ let $\nu=\nu(\zeta)=\nu^\alpha+i\nu^\beta$ denote the outer unit normal vector. 
We get
$$
I_\e(z)
=-\iint_{A_\e(z)}h(\zeta)\,\frac{\pd }{\pd\alpha}\Big(\frac1{\zeta-z}\Big)\,d\alpha\,d\beta
+\int_{S(z,\e)}h(\zeta)\cdot\frac1{\zeta-z}\cdot\nu^\alpha\,dS\,.
$$
An analogous formula holds for $\pd_\beta$. Thus for $\pd_{\bar\zeta}=(1/2)(\pd_\alpha+i\pd_\beta)$ we get, by taking into account that $\pd_{\bar\zeta}\big[1/(\zeta-z)\big]=0$,
$$
\iint_{A_\e(z)}\frac{\pd h}{\pd\bar\zeta}(\zeta)\,\frac1{\zeta-z}\,dA(\zeta)
=
\int_{S(z,\e)}h(\zeta)\cdot\frac1{\zeta-z}\cdot\frac{\nu^\alpha+i\nu^\beta}2\,dS\,.
$$
We see (e.g. by 
Figure \ref{fig0}) that on $S(z,\e)$ the outer unit normal $\nu=\nu^\alpha+i\nu^\beta$ can be evaluated as
$$
\nu=-\frac{\zeta-z}\e\,,
$$
therefore 
$$
-\frac1\pi\iint_{A_\e(z)}\frac{\pd h}{\pd\bar\zeta}(\zeta)\,\frac1{\zeta-z}\,dA(\zeta)
=\frac1{2\pi\e}
\int_{S(z,\e)}h\,dS=\av{h}_{S(z,\e)}\,.
$$
Thus the continuity of $h$ gives
$$
\lim_{\e\searrow 0}
-\frac1\pi\iint_{A_\e(z)}\frac{\pd h}{\pd\bar\zeta}(\zeta)\,\frac1{\zeta-z}\,dA(\zeta)
=h(z)\,,
$$
 i.e., for $h\in C_c^1(\C)$ we indeed have
 $$
\cC(\pd_{\bar z}h)=h\,.
 $$
 
\bigskip
Now let us address $\cC(\pd_{z}h)$. Similarly as before we get
$$
\aligned
\iint_{A_\e(z)}& \frac{\pd h}{\pd\zeta}(\zeta)\,\frac1{\zeta-z}\,dA(\zeta)\\
& =-\iint_{A_\e(z)}h(\zeta)\,\frac{\pd }{\pd\zeta}\Big(\frac1{\zeta-z}\Big)\,dA(\zeta)
+\int_{S(z,\e)}h(\zeta)\cdot\frac1{\zeta-z}\cdot\frac{\nu^\alpha-i\nu^\beta}2\,dS\,.
\endaligned
$$
Obviously 
$$
\frac{\pd }{\pd\zeta}\Big(\frac1{\zeta-z}\Big)=-\frac1{(\zeta-z)^2}
\hskip 20pt\text{and}\hskip 20pt
\nu^\alpha-i\nu^\beta=-\overline{\frac{\zeta-z}\e}\,,
$$
therefore
$$
\aligned
-\frac1\pi\iint_{A_\e(z)}&\frac{\pd h}{\pd\zeta}(\zeta)\,\frac1{\zeta-z}\,dA(\zeta)\\
&=-\frac1\pi\iint_{A_\e(z)}\frac{h(\zeta)}{(\zeta-z)^2}\,dA(\zeta)+
\underbrace{\frac1{2\pi\e}\int_{S(z,\e)}h(\zeta)\frac{\overline{\zeta-z}}{\zeta-z}\,dS(\zeta)}_{=:H_\e(z)}\,.
\endaligned
$$
We have
$$
H_\e(z)=\frac1{2\pi\e}\int_{S(0,\e)}h(\eta+z)\frac{\overline{\eta}}{\eta}\,dS(\eta)\,.
$$
Since
$$
\int_{S(0,\e)}\frac{\overline{\eta}}{\eta}\,dS(\eta)=\int_0^{2\pi}e^{-2it}\,dt=0\,,
$$
by introducing $g(\eta)=h(\eta+z)$ we obtain
$$
H_\e(z)=\frac1{2\pi\e}\int_{S(0,\e)}g(\eta)\frac{\overline{\eta}}{\eta}\,dS(\eta)
=\frac1{2\pi\e}\int_{S(0,\e)}\frac{g(\eta)-g(0)}\eta\cdot\bar\eta\,dS(\eta)\,.
$$
The fraction in the integrand above is uniformly bounded in a neighbourhood of 0, since $h\in C^1$. Therefore
$$
|H_\e(z)|\,\leqsim\,\av{|\bar\eta|}_{S(0,\e)}=\e\,.
$$
It follows that
\[
-\frac1\pi\iint_{\C}\frac{\pd h}{\pd\zeta}(\zeta)\,\frac1{\zeta-z}\,dA(\zeta)=-\frac1\pi\lim_{\e\rightarrow0}\iint_{A_\e(z)}\frac{h(\zeta)}{(\zeta-z)^2}\,dA(\zeta)=Th(z)\,.
\qedhere
\]
\end{proof}

Corollary \ref{sviridov - svyatyj bozhe} now implies its own analogue for $T$:

\begin{corollary}
\label{et in hora mortis nostrem}
We have $T:C_c^\infty(\C)\rightarrow C^\infty(\C)$ and $TD=DT$  on $C_c^\infty(\C)$ for any linear partial differential operator $D$ on $C^\infty(\C)$ with constant coefficients.
\end{corollary}

What follows is arguably the most important property of $T$ from the point of view of complex analysis.

\begin{exercise}
\label{hor sretenskogo monostirya}
Take $p>1$, For $h\in W^{1,p}(\C)$ we have
$$
T:\ \frac{\pd h}{\pd\bar z}\ \longmapsto\ \frac{\pd h}{\pd z}\,.
$$
Hint: first consider smooth $h$ and then approximate.
\end{exercise}

\begin{exercise}
\label{tetris}
The set $\Mn{\f_{\bar z}}{\f\in C_c^\infty(\C)}$ is dense in $L^2(\C)$.
\end{exercise}
\noindent
In view of these two exercises the isometry of $T$ on $L^2$ can be reformulated as 
Exercise \ref{left foot}, or vice versa.

\subsubsection{Cauchy transform on $L^p$}
We want to have the Cauchy operator available on $L^p$, $p>2$, 
without the extra assumption of compact support. But we may have integrability problems at infinity if in \eqref{sibir'} we only assume that $h\in L^p$. Therefore, following Ahlfors \cite[Chapter V.A]{A}, as a replacement for $\cC$ we introduce, for $p>2$ and $h\in L^p(\C)$, 
\begin{equation}
\label{posvjecenje proljeca}
(Ph)(z):=-\frac1\pi\int_\C h(\zeta)\left(\frac1{\zeta-z}-\frac1\zeta\right)\,dA(\zeta)\,.
\end{equation}
\begin{exercise}
\label{ruptura partialis}
The integral above converges absolutely when $h\in L^p$, i.e. $P$ is well defined.
\end{exercise}

Observe that for compactly supported $h\in L^p$ we have 
\begin{equation}
\label{grigorij aleksandrovich}
Ph(z)=\cC h(z)-\cC h(0)\,.
\end{equation}

\begin{lemma}
\label{rumyancev}
If $h\in L^p$ then $Ph$ is a continuous function. In addition, it satisfies the uniform H\"older condition with exponent $1-2/p$, i.e. 
$$
|Ph(z)-Ph(w)|\,\leqsim_p\,\nor{h}_p|z-w|^{1-2/p}
\hskip 30pt\forall\,z,w\in\C\,.
$$
\end{lemma}

\begin{proof}
By the H\"older's inequality we have, for $z\ne 0$,
$$
|Ph(z)|\leq\frac{|z|}\pi\,\nor{h}_p\Nor{\frac1{\zeta(\zeta-z)}}_q\,.
$$
By the change of variable $\zeta=z\eta$ we get, for $z\ne 0$,
$$
\aligned
\int_\C\frac{dA(\zeta)}{|\zeta(\zeta-z)|^q}
&=\int_\C\frac{J_{\zeta(\eta)}}{|z\eta(z\eta-z)|^q}\,dA(\eta)
=\int_\C\frac{|z|^2}{|z\eta(z\eta-z)|^q}\,dA(\eta)\\
&=|z|^{2-2q}\underbrace{\int_\C\frac{dA(\eta)}{|\eta(\eta-1)|^q}}_{c_p}\,,
\endaligned
$$
therefore 
\begin{equation}
\label{mirko rondovic}
|Ph(z)|\leq\frac{|z|}\pi\,\nor{h}_p\Big(c_p|z|^{2-2q}\Big)^{1/q}
=\underbrace{\frac1\pi\,\Nor{\frac1{\eta(\eta-1)}}_{L^q(d\eta)}}_{C_p}\nor{h}_p|z|^{1-2/p}.
\end{equation}

Fix arbitrary $w\in \C$. The function $\widetilde h(z):=h(z+w)$ of course again belongs to $L^p(\C)$. We have
$$
\aligned
P\widetilde h(z-w) 
& =-\frac1\pi\int_\C h(\zeta+w)\left(\frac1{\zeta+w-z}-\frac1\zeta\right)\,dA(\zeta)\\
& =-\frac1\pi\int_\C h(\zeta)\left(\frac1{\zeta-z}-\frac1{\zeta-w}\right)\,dA(\zeta)\\
&= Ph(z)-Ph(w)\,.
\endaligned
$$
Together with \eqref{mirko rondovic} this gives 
$$
|Ph(z)-Ph(w)|=\left|P\widetilde h(z-w)\right|\leq C_p\nor{h}_p|z-w|^{1-2/p}\,,
$$
as desired.
\end{proof}

This lemma has an important corollary, which can be perceived as a generalization of Lemma \ref{lyubo, bratcy, lyubo}:

\begin{exercise}
\label{preobrazhenskij}
For every $h\in L^p(\C)$, $p>2$, we have
$$
\aligned
(P h)_{\bar z}&=h\\
(P h)_{z}&=Th\,
\endaligned
$$
in the distributional sense. 

(Recall that this means that 
$$
\aligned
\int_\C Ph\cdot\phi_{\bar z} &=-\int_\C\phi\cdot h\\
\int_\C Ph\cdot\phi_{z} &=-\int_\C\phi\cdot Th
\endaligned
$$
for all $\phi\in C_c^\infty(\C)$.)
\end{exercise}

\subsection{Solving the Beltrami equation}

Exercise \ref{hor sretenskogo monostirya} suggests that $T$ might have an important r\^ole in solving the Beltrami equation \eqref{cyrenaica}. Indeed this is the case, as we explain in this section, following Ahlfors \cite[Chapter V.B]{A}. We will treat in detail the case of compactly supported $\mu$.

\begin{theorem}
\label{simon bolivar dudamel}
Suppose $\mu\in L^\infty(\C)$ is of compact support and $k:=\nor{\mu}_\infty<1$. Let $p>2$ be such that $k\nor{T}_p<1$. Then there exists a unique $f\in C(\C)\cap W^1(\C)$ that satisfies the conditions
\begin{equation}
\label{orvieto}
\aligned
\overline\pd f & =\mu\cdot\pd f\,;\\
f(0) & =0\,;\\
f_z-1&\in L^p(\C)\,.
\endaligned
\end{equation}

\end{theorem}

\begin{remark}
Since $T$ is an isometry on $L^2$, its norm equals 1, of course. Therefore, by the continuity of $L^p$ norms (Exercise \ref{filter}), the condition $k\nor{T}_p<1$ is indeed fulfilled for $p$ sufficiently close to 2.

Since $\mu$ is compactly supported, the Beltrami equation and the Weyl lemma imply that $f$ is holomorphic outside a large disc.
\end{remark}

\begin{proof}
First let us prove uniquenness.

If $f$ satisfies \eqref{orvieto}, then 
$f_{\bar z} =\mu\cdot f_z=\mu(f_z-1)+\mu\in L^p(\C)$, so that $P(f_{\bar z})$ is well defined and hence so is
$F:=f-P(f_{\bar z})$. From Exercise \ref{preobrazhenskij} we conclude that 
$F_{\bar z}=f_{\bar z}-\big(P(f_{\bar z})\big)_{\bar z}=0$ distributionally. By Lemma \ref{rumyancev}, $F$ is continuous, therefore also $F\in L_{loc}^1(\C)$. Owing to the Weyl's lemma (Theorem \ref{Weyl}), under such conditions $F$ must be holomorphic. We have, in a weak sense,
\begin{equation}
\label{sergej sergejevich}
F_z-1=(f_z-1)-\big(P(f_{\bar z})\big)_z=(f_z-1)-T(f_{\bar z})\in L^p(\C)\,.
\end{equation}
Certainly, for smooth functions their weak derivative coincides with the strong one. Now Exercise \ref{useful} implies that $F_z-1\equiv 0$. 
Consequently \eqref{sergej sergejevich} yields
$$
f_z=T(f_{\bar z})+1=T(\mu f_z)+1\,.
$$

Suppose now $g$ is another solution of \eqref{orvieto}. Then $g_z=T(\mu g_{z})+1$ as before and thus
$$
f_z-g_z=T\big(\mu(f_{z}-g_{z})\big)\,,
$$
therefore $\nor{f_z-g_z}_p\leq k\nor{T}_p\nor{f_z-g_z}_p$. But $k\nor{T}_p<1$, hence $f_z-g_z=0$ p.p. $\C$. Beltrami's equation implies $f_{\bar z}-g_{\bar z}=0$ p.p. $\C$. By using the Weyl's lemma again, together with the fact that $\overline\f_{\bar z}=\overline{\f_z}$, we conclude that $f-g$ and $\overline{f-g}$ are both holomorphic functions. This is only possible if $f-g\equiv c$ for some $c\in\C$, whereupon the normalization gives $f\equiv g$.

\medskip
Let us now address the existence. First we will find a natural {\sl candidate} for the solution and then prove that this candidate is indeed the right one. 

{\sl Suppose} $f$ is a solution to \eqref{orvieto}. Note that $f\not\in W^{1,p}(\C)$, because $f_z\in 1+L^p$. Hence for $g:=f-z$ we have $g_z=f_z-1\in L^p$ and $g_{\bar z}=f_{\bar z}\in L^p$, by the Beltrami equation. In other words, $g\in W^{1,p}(\C)$. It follows from Exercise \ref{hor sretenskogo monostirya} that 
$g_{\bar z}=\mu(g_z+1)=\mu(Tg_{\bar z})+\mu$ and so $\mu=(I-\mu T)g_{\bar z}$. The assumed relation between $\mu$ and $p$ means that $\mu T$ is invertible on $L^p$, by the Neumann series, therefore $g_{\bar z}=(I-\mu T)^{-1}\mu$. Exercise \ref{preobrazhenskij} offered a solution to the (distributional) $\bar\pd$-equation with data from $L^p$, namely by means of the operator $P$. That is, $g=P\big((I-\mu T)^{-1}\mu\big)$. Therefore we have our candidate for a solution of \eqref{orvieto}:
\begin{equation}
\label{alcazar}
f=z+P\big((I-\mu T)^{-1}\mu\big)\,.
\end{equation}
It remains to verify that this function indeed has all the required properties. 

\begin{itemize}
\item
We saw before that $(I-\mu T)^{-1}$ is invertible on $L^p$, therefore $(I-\mu T)^{-1}\mu\in L^p$, for $\mu\in L^\infty_c\subset L^p$. By Lemma \ref{rumyancev}, $P\big((I-\mu T)^{-1}\mu\big)$ is a well-defined continuous function on $\C$, therefore so is our $f$. 

\item
Since $Ph(0)=0$ for any $h\in L^p$, we get $f(0)=0$.

\item
Exercise \ref{preobrazhenskij} gives, in a distributional sense, 
\begin{eqnarray}
\label{ferrara}
\aligned
f_{\bar z}&=(I-\mu T)^{-1}\mu\\
f_{z}&=1+T\big((I-\mu T)^{-1}\mu\big)\,.
\endaligned
\end{eqnarray}
Since $(I-\mu T)^{-1}\mu\in L^p$ and $T\in B(L^p)$, we conclude that $f_{\bar z},f_{z}-1\in L^p$. 

\item
Finally,
$$
\aligned
\mu f_z & =\mu+\mu T\big((I-\mu T)^{-1}\mu\big)=(I-\mu T)(I-\mu T)^{-1}\mu+\mu T\big((I-\mu T)^{-1}\mu\big)\\
 & =\big[(I-\mu T)+\mu T\big]\big((I-\mu T)^{-1}\mu\big) =(I-\mu T)^{-1}\mu\\
 & =f_{\bar z}\,. 
\endaligned
 $$
\end{itemize}
Thus all the properties of \eqref{orvieto} hold, as claimed.
\end{proof}

Function $f$ defined in \eqref{alcazar} is called the {\it normal solution} to the Beltrami equation \eqref{cyrenaica}.
It is known \cite[Theorem V.B.2]{A} that normal solutions are not only continuous, but in fact homeomorphisms.

\begin{remark}
\label{lecho}
The formula \eqref{alcazar} is essentially due to Bojarski \cite{Bo1}.

From 
\eqref{alcazar} and \eqref{ferrara} we see that $f=z+P(f_{\bar z})$. 

As our final -- yet very imporant -- remark in this section let us mention that in Theorem \ref{lyubo, bratcy, lyubo} one can remove that assumption about $\mu$ having compact support \cite[Theorem V.B.3]{A}.
\end{remark}

\subsection{The Iwaniec conjecture}
\label{Ravel piano G}

We also see from \eqref{ferrara} that $f-z\in W^{1,p}$, provided that $\nor{\mu T}_{p}<1$. Hence the Sobolev integrability of quasiconformal maps {\it self-improves}, meaning that the a priori assumption $f\in W^{1,2}_{loc}(\C)$ when coupled with the Beltrami equation improves to $f\in W^{1,p}_{loc}(\C)$ for some $p>2$. 
This fact is known as 
Bojarski's theorem, see \cite{Bo} or \cite[Theorem 5.4.2]{AIM}. 
For a long time a major question in the area was to determine the best (i.e. largest) $p$ that can be attained if one knows $\nor{\mu}_\infty=:k$. The question was solved by K. Astala \cite{Ast} who showed that the answer is $1+1/k-$. One immediately sees that the condition $\nor{\mu T}_{p}<1$ would be satisfied for all $p\in[2,1+1/k)$ if we knew that $\nor{T}_p=p-1$ for $p\geq 2$.
This is however a notoriously difficult problem, open at least since 1982 when it was formulated by T. Iwaniec \cite{I}:
\begin{conjecture}[Iwaniec]
\label{peter hammers}
For every $p\in (1,\infty)$,
\begin{equation}
\label{roll over, roll over}
\nor{T}_p=p^*-1\,.
\end{equation}
\end{conjecture}
It is not very difficult to prove that $\nor{T}_p\geq p^*-1$, see \cite{L} or \cite{BM-S}.
As to the upper estimates, following a series of results by different authors, the best to-date result, due to Ba\~nuelos and Janakiraman, is $\sqrt{2p(p-1)}$ for $p\geq 2$, see \cite{BJ} and the references there. In the asymptotical sense, the strongest estimate known today is due to 
Borichev, Janakiraman, Volberg \cite{BJV} who obtained $\nor{T}_p\leq 1^\cdot 3922(p^*-1)$ for $p^*\rightarrow\infty$. 

\medskip
Apart from immediately implying the optimal Sobolev integrability of quasiconformal maps, the confirmation of the Iwaniec conjecture would have many other deep consequences. See e.g. \cite{BM-S} for a related discussion. In particular, the validity of \eqref{roll over, roll over} is closely related to the calculus of variations and the Morrey's problem \cite{BM-S, AIPS}.

Another implication that \eqref{roll over, roll over} would bring regards the issue we are well familiar with by now -- distortion of area under the action of quasiconformal mappings.

\subsubsection{Area distortion by quasiconformal mappings}

The proof of Theorem \ref{simon bolivar dudamel} gives the following important consequence
\cite[Theorem I.7.2]{CG}.

\begin{theorem}
\label{drine vode hladne}
Take $R>0$ and $\mu\in L^\infty$ so that $\supp\mu\subset R\Delta$. Write $k:=\nor{\mu}_\infty<1$. 
Suppose $f$ is the normal solution to the Beltrami equation \eqref{orvieto}. If $p>2$ is such that $k\nor{T}_p<1$, then for every 
$|E|\subset R\Delta$ we have
$$
|f(E)|\,\leqsim_{R,k,p}\,|E|^{1-2/p}\,.
$$
\end{theorem}

\begin{proof}
We start by recalling \eqref{jacobi} and calculating
\begin{equation}
\label{bronenosec}
|f(E)|=\int_{f(E)}1=\int_E J_f\leq\int_E|f_z|^2\,.
\end{equation}
By using \eqref{ferrara} and writing $\f:=T\big((I-\mu T)^{-1}\mu\big)$, so that $f_z=1+\f$, we obtain
\begin{equation}
\label{knjaz}
\aligned
\int_E|f_z|^2 & =\nor{f_z}_{L^2(E)}^2\\
& \,\leqsim\, \nor{1}^2_{L^2(E)} + \nor{\f}^2_{L^2(E)}\\
& = |E|+\Nor{\f^2}_{L^1(E)}\,.
\endaligned
\end{equation}
The last term can be estimated by applying the H\"older's inequality to the pair $p/2$, $(p/2)'=p/(p-2)$ to get
$$
\Nor{\f^2}_{L^1(E)}\leq
\Nor{\f^2}_{L^{p/2}(E)}\nor{1}_{L^{(p/2)'}(E)}
=\Nor{\f}_{L^{p}(E)}^2|E|^{1-2/p}\,.
$$
But
$$
\Nor{\f}_{L^{p}(E)}\leq\Nor{\f}_p\leq\nor{T}_p\cdot\frac{1}{1-k\nor{T}_p}\cdot\nor{\mu}_p
\leq\frac{k\nor{T}_p}{1-k\nor{T}_p}\cdot(\pi R^2)^{1/p}<\infty\,,
$$
therefore
\begin{equation}
\label{potjomkin}
\Nor{\f^2}_{L^1(E)}\,\leqsim_{R,k,p}\,|E|^{1-2/p}\,.
\end{equation}

On the other hand, $|E|/R\leq 1$ implies 
$|E|/R\leq (|E|/R)^{1-2/p}$, 
i.e. 
\begin{equation}
\label{tavricheskij}
|E|\leq R^{2/p}|E|^{1-2/p}\,.
\end{equation}
Finally merging \eqref{bronenosec}, \eqref{knjaz}, \eqref{potjomkin} and \eqref{tavricheskij} finishes the proof.
\end{proof}

It was conjectured by Gehring and Reich \cite{GR}, see also \cite{BM-S}, that the best integrability exponent in the theorem above should be $1/K$, where $K$ is inferred from $k$ through \eqref{B-A-B-G}. This exponent is attained in the case of the radial stretch function $z\mapsto z|z|^{1/K-1}$. 
The Gehring-Reich conjecture 
has attracted a lot of attention as the central problem in the area of planar quasiconformal mappings. It was eventually solved by Astala \cite{Ast}. Er\"emenko and Hamilton \cite{EH} gave a simplified proof of Astala's theorem.

One sees that the extremal case $1-2/p=1/K$ holds precisely when 
$$
\frac1k=\frac{K+1}{K-1}=\frac{1+1/K}{1-1/K}=\frac{2-2/p}{2/p}=p-1\,.
$$
In view of Theorem \ref{drine vode hladne} 
this small observation shows that one would immediately recover Astala's theorem if -- again -- one knew 
\eqref{roll over, roll over}.
At this point one may guess that the Gehring-Reich conjecture
is essentially equivalent to the optimal Sobolev integrability of solutions to the Beltrami equation, see \cite[Chapter 13]{AIM}.

\subsection{Weak quasiregularity vs. quasiregularity: weighted estimates of $T$}
\label{liberation sarajevo}

In Section \ref{Ravel piano G} we mentioned the question of the optimal Sobolev integrability that an a priori $W^{1,2}_{loc}$ function is pushed to by satisfying the Beltrami equation. One can pose a ``dual" version of this question:\\

\noindent
What is the smallest $p_-\in(1,2)$ such that $f\in W^{1,p_-}_{loc}$ which satisfies the Beltrami equation with $\nor{\mu}=k$, automatically belongs to $W^{1,2}_{loc}$ (and therefore, by Astala's theorem, to $W_{loc}^{1,1+1/k-}$)?\\

The aim of this section is to 
answer this question. We will show that it reduces to proving sharp estimates of the Ahlfors-Beurling operator on weighted $L^p$ spaces with weights from the Muckenhoupt class, as specified in the following fundamental theorem. It was first proven by Petermichl and Volberg \cite{PV}. In this note we will give a proof as in \cite{DV1}. We will devote Section \ref{Evtushenko} to this purpose.

\begin{theorem}[\cite{PV,DV1}]
\label{novljan}
For any $p>1$ and any $w\in A_p$ we have
$$
\nor{T}_{\cB(L^p(w))}\,\leqsim_p\, [w]_{A_p}^{p^*/p}.
$$
\end{theorem}

Recently there have been remarkable efforts by many mathematicians to extend this result to all Calder\'on-Zygmund operators. The final step towards the proof of that extension was done by T. Hyt\"onen \cite{Hy1}.

\medskip
Assuming Theorem \ref{novljan} until the end of this section, let us proceed to its consequences. It will be directly used in the proof of Theorem \ref{Formicidae} which will in turn imply the answer to the question posed above.

\medskip
Denote, for $k\in(0,1)$,  
$$
I_k:=(1+k,1+1/k)\,.
$$

\begin{lemma}
\label{pat i mat}
Let $\mu\in L^\infty(\C)$ be such that $k:=\nor{\mu}_\infty<1$. 
Suppose that for all $p\in I_k$ and $g\in L^p$ we have $\nor{(I-\mu T)g}_p\,\geqsim_{p,k}\,\nor{g}_p$. Then $I-\mu T$ is invertible on $L^p$ for all $p\in I_k$. 
\end{lemma}

\begin{proof}
By Exercises \ref{alban}, \ref{berg} it suffices to show that $(I-\mu T)^*$ is injective. Recall also Exercise \ref{lullaby}. Then we have on $L^q$, compare with \cite[p. 371]{AIM},
$$
(I-\mu T)^*=I-S\overline{\mu}=S(I-\overline\mu S)T\,.
$$
Clearly 
$
(I-\overline\mu S)h=\overline{(I-\mu T)\overline h}.
$
Consequently, we get from \eqref{simeon stilit} and the assumption on $I-\mu T$
\[
\nor{(I-\mu T)^*g}_q
\,\geqsim\,\nor{(I-\overline\mu S)Tg}_q
=\nor{(I-\mu T)\overline{Tg}}_q
\,\geqsim\,\nor{Tg}_q
\,\geqsim\,\nor{g}_q\,.
\qedhere
\]
\end{proof}

The next result is stated in \cite[p. 4]{V2}.

\begin{exercise}
\label{larynx}
Let $\mu\in L^\infty(\C)$ be such that $k:=\nor{\mu}_\infty\in(0,1)$. Take $p\in I_k$. 
If the Iwaniec conjecture \eqref{roll over, roll over} holds, then $I-\mu T$ is invertible on $L^p$ and 
$$
\Nor{(I-\mu T)^{-1}}_p\leq\frac{1/k}{d(p,I_k^c)}\,.
$$
\end{exercise}

Actually, the 
invertibility of $I-\mu T$ is in fact true, regardless of the validity of \eqref{roll over, roll over}. More precisely, we have this result \cite[Theorem 3]{AIS}, \cite[Theorem 1.4]{V2}:

\begin{theorem}
\label{avtopralnica}
Suppose $\nor{\mu}_\infty=:k<1$. Take $p\in I_k$. Then $I-T\mu$ and $I-\mu T$ are invertible in $L^p$.
\end{theorem}

Before proceeding to the proof of this theorem let us quote another beautiful result of Astala, Iwaniec and Saksman \cite[Theorem 12]{AIS} which is crucial for establishing connection between quasiconformal maps and the theory of $A_p$ weights. The formulation presented here is 
not the most general one, but instead one which just suffices for our purposes. For a more general statement 
see \cite[Theorem 13.4.2 and estimate (13.58)]{AIM}. 

\begin{theorem}
\label{vozvraschenie}
Choose $K\geq 1$ and let $k$ be as in \eqref{B-A-B-G}. Let $p\in[2,1+1/k)$. Suppose $f:\C\rightarrow\C$ is a $K$-quasiconformal map. Define $w:=J_f^{1-p/2}$. Then $w\in A_2$ and 
$$
[w]_{A_2}\,\leqsim_K\,\frac1{1+1/k-p}\,.
$$
\end{theorem}

\begin{proof}[{\bf Proof of Theorem \ref{avtopralnica}}]
We have $I-T\mu=T(I-\mu T)S$ and since $T,S$ are both invertible in any $L^p$, cf. Exercise \ref{lullaby}, it is enough to verify the statement for $I-\mu T$. Lemma \ref{pat i mat} reduces the theorem to proving
\begin{equation}
\label{boogie woogie}
\nor{(I-\mu T)g}_p\,\geqsim_{p,k}\,\nor{g}_p
\hskip 30pt \text{for all }p\in I_k \text{ and }g\in L^p\,.
\end{equation}
We follow \cite[proof of Lemma 14]{AIS}. By Exercise \ref{blgaria} we may take $g\in C_c^\infty(\C)$ with zero average, i.e. $\int_\C g=0$. 

Define $\f:=\cC g$, with $\cC$ being the planar Cauchy transform. Lemma \ref{lyubo, bratcy, lyubo} gives $\f_{\bar z}=g$. Let $h=(I-\mu T)g$. Therefore our goal is to prove
\begin{equation}
\nor{\f_{\bar z}}_p\,\leqsim_{k,p}\,\nor{h}_p\,.
\end{equation}
We know that there is a $K-$qasiconformal homeomorphism $f$ such that 
$f_{\bar z}-\mu f_{z}=0$, cf. Remark \ref{lecho}. Set $u=\f\circ f^{-1}$, that is, $\f=u\circ f$. The chain rule reads
\begin{equation}
\label{14:53}
\aligned
\f_{\bar z} & =(u_z\circ f)\cdot\pd_{\bar z}f+(u_{\bar z}\circ f)\cdot\pd_{\bar z}\bar f\\
\f_{ z} & =(u_z\circ f)\cdot\pd_{ z}f+(u_{\bar z}\circ f)\cdot\pd_{ z}\bar f\,.
\endaligned
\end{equation}
We will start with 
\begin{equation}
\label{arch stanton}
|\f_{\bar z}|^p \,\leqsim_p\,|u_z\circ f|^p|\pd_{\bar z}f|^p+|u_{\bar z}\circ f|^p|\pd_{\bar z}\bar f|^p\,.
\end{equation}
From \eqref{14:53} it follows, through Exercise \ref{hor sretenskogo monostirya}, that
$$
h=\f_{\bar z}-\mu\f_{z}=0+(u_{\bar z}\circ f)\cdot\left(\pd_{\bar z}\bar f-\mu\cdot\pd_{z}\bar f\right)\,.
$$
Since $\pd_{\bar z}\bar f=\overline{\pd_{z} f}$ and 
$\pd_{z}\bar f=\overline{\pd_{\bar z}f}
=\overline{\mu\cdot\pd_{z}f}=\overline\mu\cdot\overline{\pd_z f}$,
we obtain
$$
h=\left(1-|\mu|^2\right)\cdot (u_{\bar z}\circ f)\cdot\overline{\pd_z f}\,,
$$
which together with $J_f=|f_z|^2-|f_{\bar z}|^2\leq|f_z|^2$ implies
$$
|u_{\bar z}\circ f|^p\cdot J_f^{p/2-1}\cdot J_f
\leq 
|u_{\bar z}\circ f|^p\cdot|f_z|^p
= 
\frac{|h|^p}{\big(1-|\mu|^2\big)^p}\,
\leqsim_k\,
|h|^p\,.
$$
Integrate, change the variable into $z'=f(z)$ and use $J_f\circ f^{-1}=(J_{f^{-1}})^{-1}$ to get
\begin{equation}
\label{ame agaru}
\int|u_{\bar z}|^p\cdot J_{f^{-1}}^{1-p/2}\,\leqsim_k\,\int|h|^p\,.
\end{equation}

\medskip
On the other hand, the definition of $f$ gives $|f_{\bar z}|\leq k|f_z|$, therefore 
$J_f\geq(1/k^2-1)|f_{\bar z}|^2$ and so
$$
|u_{z}\circ f|^p|\pd_{\bar z}f|^p
\leq\frac{k^2}{1-k^2}
|u_{z}\circ f|^p\cdot J_f^{p/2-1}\cdot J_f\,,
$$
therefore
$$
\int|u_{z}\circ f|^p|\pd_{\bar z}f|^p
\,\leqsim_k\,\int|u_{z}|^p\cdot J_{f^{-1}}^{1-{p/2}}
=\int|Tu_{\bar z}|^p\cdot J_{f^{-1}}^{1-{p/2}}.
$$
Write
$$
w=J_{f^{-1}}^{1-{p/2}}\,.
$$
Let us summarize our findings so far:
$$
\aligned
&\bullet
\int|u_{z}\circ f|^p|\pd_{\bar z}f|^p
\leq\frac{k^2}{1-k^2}
\int|Tu_{\bar z}|^pw\\
&\bullet
\int|u_{\bar z}|^pw
\leq
\int|u_{\bar z}\circ f|^p\cdot|f_z|^p
\leq\frac1{(1-k^2)^p}\int|h|^p
\,.
\endaligned
$$
In fact, the top chain of inequalities can be (up to a multiplicative constant) continued by the bottom one:

It is a well-known fact that every Calder\'on-Zygmund operator (such as $T$) is bounded on $L^r(w)$ for any $A_r$ weight $w$, e.g. \cite[Theorem 7.11]{Du}.
Since $f$ is $K$-quasiconformal, so is $f^{-1}$ \cite[Theorem 3.7.7]{AIM}. Theorem \ref{vozvraschenie} implies that $w\in A_2$.

If $p\geq2$ then also $w\in A_{p}$ 
\cite[Proposition 7.2]{Du}.
Now suppose $1<p<2$. Then $w\in A_p$ if and only if $w^{1-q}\in A_q$, where $q=p/(p-1)\in[2,1+1/k)$. Now 
$$
w^{1-q}=J_{f^{-1}}^{(1-{p/2})(1-q)}=J_{f^{-1}}^{1-{q/2}}\,, 
$$
which by Theorem \ref{vozvraschenie} belongs to $A_2$ and thus also $A_q$.

This shows that, after integration, both of the summands in \eqref{arch stanton} are controlled by $\int|h|^p$, which finishes the proof.
\end{proof}

\begin{remark}
\label{plocnik}
The proof above shows that if
$$
\int|Tu_{\bar z}|^pw\,\leqsim\,F\big([w]_{A_2}\big)\int|u_{\bar z}|^pw
$$
for some increasing function $F$, then 
$$
\Nor{(I-\mu T)^{-1}}_{p\rightarrow p}\,\leqsim_K\,1+F\big([w]_{A_2}\big)\,,
$$
where, as before, $w=J_{f^{-1}}^{1-{p/2}}$ and $f$ is $K$-quasiconformal.
\end{remark}

We are nearing the main objective of this section. The next result can be found in \cite[Theorem 14.4.1]{AIM}. 

\begin{theorem}
\label{Formicidae}
Suppose that $\mu\in L^\infty(\C)$ with $\nor{\mu}_\infty=:k<1$. Let $q:=1+k$. Then $I-\mu T$ and $I-T \mu$ are injective on $L^{1+k}(\C)$.
\end{theorem}

\begin{proof}
By Exercise \ref{berg} it suffices to prove that $I-\mu T$ and $I-T \mu$ have dense ranges on $L^{1+1/k}(\C)$. Now Exercise \ref{lullaby} and the identity $I-T\mu=T(I-\mu T)S$ imply that it is enough to prove this for $I-\mu T$. Since $\cR(I-\mu T)$ is a convex set, it suffices to prove that it is {\sl weakly} dense in $L^{1+1/k}(\C)$, see Rudin \cite[Theorem 3.12]{R}.

Choose $h\in C_c^\infty(\C)$ and $0<\e<1$. Denote by $p$ the conjugate exponent of $q$, i.e. $p=1+1/k$. 
Since $p\in I_{(1-\e)k}$, Theorem \ref{avtopralnica} says that 
$$
\phi_\e:=\big[I-(1-\e)\mu T\big]^{-1}h\in L^p(\C)\,.
$$
Let us write this formula as 
$$
(I-\mu T)\phi_\e=h-\mu T(\e\phi_\e)\,.
$$
Now invoking again Exercise \ref{lullaby} {\it ii)} reduces our problem to proving that $\e\phi_\e\rightarrow 0$ weakly in $L^q$ as $\e\rightarrow 0$. 

At this point we use our weighted estimate for $T$ -- Theorem \ref{jednamladost}. For together with Remark \ref{plocnik} and Theorem \ref{vozvraschenie} it gives the estimate
$$
\Nor{\big[I-(1-\e)\mu T\big]^{-1}}_{\cB(L^p(\C))}\,\leqsim_K\,[w]_{A_2}\leq \frac1\e\,.
$$
This immediately implies $\nor{\e\phi_\e}_p\,\leqsim_K\nor{h}_p$. On the other hand, by applying in $L^2$ the Neumann series to the definition of $\phi_\e$ we see that
$$
\nor{\phi_\e}_2\leq\frac1{1-(1-\e)k}\nor{h}_2\,.
$$
Hence $\e\phi_\e\rightarrow 0$ in $L^2$ as $\e\rightarrow 0$. 
Now Exercise \ref{ging heut morgen uebers feld} implies that $\e\phi_\e\rightarrow 0$ weakly in $L^p$ as $\e\rightarrow 0$, which had to be proven.
\end{proof}

The next result summarizes our efforts in this section. It appeared in \cite[Lemma 19 and Proposition 19]{AIS}, \cite[Corollary 14.4.3]{AIM}, \cite[Theorem 1.5]{V2}.

\begin{theorem}
Suppose $\mu\in L^\infty(\C)$ and $\nor{\mu}_\infty=:k<1$. 
If a solution of $F_{\bar z}-\mu F_z=0$ lies in $W^{1,1+k}_{loc}(\C)$, then it automatically also lies in $W^{1,2}_{loc}(\C)$. Therefore 
it even lies in $W^{1,1+1/k-}_{loc}(\C)$.
\end{theorem}

\begin{proof}
The case $k=0$ follows from Weyl's lemma. Now assume $k\in(0,1)$. 
Write $q=1+k$. Suppose $F\in W^{1,q}_{loc}(\C)$ satisfies $F_{\bar z}-\mu F_z=0$. Take $\phi\in C_c^\infty(\C)$ and define $G=\phi F$. Then $G\in W^{1,q}(\C)$. Let $\omega\in C_c^\infty(\C)$ be such that $\omega\big|_{\supp \phi}\equiv 1$. 
Then clearly
$$
G_{\bar z}-\mu G_z=(\phi_{\bar z}-\mu \phi_z)F=(\phi_{\bar z}-\mu \phi_z)(\omega F)=:\widetilde F\,.
$$
Since $F\in W^{1,q}_{loc}(\C)$ it follows that $\omega F\in W^{1,q}(\C)$, therefore, by the Sobolev embedding (Theorem \ref{Johhny Shines}), 
$\omega F\in L^{2K}$, where $K> 1$ is as in \eqref{B-A-B-G}. Since $\widetilde F$ contains $\phi_{\bar z}-\mu \phi_z$ as a factor, it is compactly supported, therefore $\widetilde F\in L^2\cap L^{2K}$.

Let us now try to solve the equation 
\begin{equation}
\label{my girl}
H_{\bar z}-\mu H_z=\widetilde F
\end{equation}
but for $H$ whose derivatives belong to $L^2(\C)$ instead of $L^q(\C)$. By Exercise \ref{hor sretenskogo monostirya} this is the same as 
$$
(I-\mu T)H_{\bar z}=\widetilde F\,.
$$
Since $\widetilde F\in L^{r}$ with $r\leq 2K$, $2K>2$ and $I-\mu T$ is invertible on $L^{2+\e}$ for small $\e>0$ (Neumann series), this is the same as
$$
H_{\bar z}=(I-\mu T)^{-1}\widetilde F\in L^{2+\e}\,.
$$
By Exercise \ref{preobrazhenskij} the solution of this $\bar\pd$-problem is given by the Cauchy transform of $(I-\mu T)^{-1}\widetilde F$. But obviously also $(I-\mu T)^{-1}\widetilde F\in L^{2}$, therefore $H_{\bar z}, H_{ z}=TH_{\bar z}\in L^2$. 
We may assume that $\mu$ has a compact support, otherwise we may consider $\lambda:=\mu\chi_{\supp\phi}$ or $\lambda=\omega\mu$ instead. Therefore 
$H_{\bar z}=\widetilde F +\mu H_z$
has compact support and hence belongs to $L^q$, since $\widetilde F$ does. Then also 
$H_{ z}=TH_{\bar z}\in L^q$.

\medskip
So we got {\it two} solutions to \eqref{my girl} on $L^q$, namely $G$ and $H$. Thus $G_{\bar z}-\mu G_z=H_{\bar z}-\mu H_z$. 
Since $G_{ z}=TG_{\bar z}$ and $H_{ z}=TH_{\bar z}$ we obtained $(I-\mu T)(G_{\bar z}-H_{\bar z})=0$. We saw that $G_{\bar z},H_{\bar z}\in L^q$. But Theorem \ref{Formicidae} says that $I-\mu T$ is injective on $L^q$, therefore $G_{\bar z}=H_{\bar z}\in L^2$ and consequently also $G_{z}=H_{z}\in L^2$. Thus $G-H$ is a constant. Since $H$ is of the form $Ph$ for some $h\in L^{2+\e}$, it is continuous, by Lemma \ref{rumyancev}. Therefore $H\in L^2_{loc}$ and hence also $G\in L^2_{loc}$. Since $G$ has compact support, it follows that $G\in L^2$. Thus for $G=\phi F$ we got $G,G_{\bar z},G_z\in L^2$. Since $\phi\in C_c^\infty(\C)$ was arbitrary this means that $F\in W^{1,2}_{loc}$.
\end{proof}

\subsection{Estimates of $T^n$ on $L^p$}
\label{saso mange}

In the article \cite{IM} Iwaniec and Martin study singular integrals that appear in regularity theory of nonlinear PDE in arbitrary dimensions. For example, they compute $L^p$ norms of scalar Riesz transforms on $\R^n$, thus extending a well-known result of Pichorides \cite{P}. One of the key features of their work is that they reduce the estimates of vector-valued operators on $\R^n$ (such as combinations of Riesz transforms, complex Riesz transforms, certain differential operators, etc.) to those of scalar-valued operators on $\C$. 
They succeed in doing that by developing a so-called {\it complex method of rotations} \cite[Section 6]{IM}. We summarize this technique in Section \ref{rotmet} below. As it emerges from their approach, the crucial r\^{o}le is played by the Ahlfors-Beurling operator $T$, its ``square root'' ${\bf H}_\C$ (recall it was defined in Section \ref{riesz}) and their powers ${\bf H}_\C^k$. 
Throughout their article they extensively work with them; most of their estimates 
are expressed in terms of the norm of ${\bf H}_\C^k$ on $L^p(\C)$, which they denote by $H_p(k)$. However, no estimate on $H_p(k)$ itself is given. 
In a subsequent paper by Iwaniec and Sbordone \cite{IS} it was noticed that for  
odd $k$ one can resort to the method of rotations, developed in the 1956 paper \cite{CZ} of Calder\'on and Zygmund, see also \cite[Section 4.3]{Du} or \cite[Chapter II]{S}, which yields 
\begin{equation}
\label{minakarina}
H_p(2n-1)\leqslant \frac\pi 2(2n-1)\cot\frac{\pi}{2p^*}\leqsim\, 
np^*,\hskip 15pt \forall n\in\N\,.
\end{equation}
On the other hand, the same approach applied in the case of even $k$ does not give such a linear estimate. Recall that 
$H_p(2n)=\nor{T^n}_p$. 
Obviously
\begin{equation}
\label{tavisuplebis}
H_p(2n)\leqslant H_p(2n-1)H_p(1)\leqsim\,n{p^*}^2\,.
\end{equation}
The slight difference lies in the fact that the kernel of ${\bf H}_\C^{2n}=T^n$ is even. Calder\'on and Zygmund \cite{CZ} derive a method for operators with even kernels as well, but that method yields the same quadratic estimate in $p$ as \eqref{tavisuplebis}, namely
$$
H_p(2n)\leqsim\,n{p^*}^2\,.
$$

In \cite{DPV} yet another method of rotation was presented, which works very well exactly for even kernels. Applying it to $T^n={\bf H}_\C^{2n}$ returned an estimate $H_p(2n)\leqsim\,n{p^*}$, which is the best possible estimate {\it of the type $\phi(n)\psi(p)$}, but still not the overall optimal estimate of $H_p(2n)$ simultaneously in $n$ and $p$. The optimal result is stated in Theorem \ref{izvorska}.

Finally, the estimates of $H_p(k)$ for any $k\in\Z$ are also considered in \cite[Section 4.5]{AIM}. The estimate in Corollary 4.5.1 there is
\begin{equation}
\label{harcho}
H_p(k)\leq C(p)(1+k^2)
\hskip 30pt\forall\ k\in\Z\,.
\end{equation}  
This is improved in our Theorem \ref{goin' to louisiana} below.

\medskip
Section \ref{nas na babu promenyal} is devoted to the estimates of $H_p(k)$. In Theorems \ref{goin' to louisiana} and \ref{izvorska} we improve \eqref{minakarina}, \eqref{tavisuplebis} and \eqref{harcho}, and in the case of $H_p(2n)$ our estimate is asymptotically sharp in terms of both $n$ and $p$. We also conjecture the exact value of $H_p(2n)$. 
See Section \ref{Mitropoulos} for precise formulations of our results.

\subsubsection{Iwaniec and Martin's complex method of rotations} 
\label{rotmet}
Here we summarize \cite[Section 6]{IM}.

Let $E$ and $F$ be finite-dimensional complex vector spaces, each of them endowed with its own Hermitian inner product. Denote by $\cL(E,F)$ the vector space of all linear mappings from $E$ to $F$. The norm on $\cL(E,F)$ is introduced by
$$
\nor{\Lambda}:=\sup_{h\in E\backslash\{0\}}\frac{\nor{\Lambda h}_F}{\nor{h}_E}\,.
$$
We additionally assume that $m\in\Z\backslash\{0\}$ and $p\in(1,\infty)$ are given.

\begin{theorem}
Suppose $\Omega:\C^n\rightarrow\cL(E,F)$ is a measurable function satisfying
\begin{equation}
\label{vrat}
\Omega(\lambda\zeta)=\left(\frac{\lambda}{|\lambda|}\right)^{-m}\frac{\Omega(\zeta)}{|\lambda|^{2n}}\,
\end{equation}
for all $(\lambda,\zeta)\in\C\times S^{2n-1}$ and suppose that the restriction $\Omega\big|_{S^{2n-1}}$ is integrable.
Then the singular integral operator $\kT:L^p(\C^n,E)\rightarrow L^p(\C^n,F)$, defined by 
$$
(\kT f)(z)=\int_{\C^n}\Omega(z-w)f(w)\,dw\,,
$$
is bounded and admits the following norm estimate:
$$
\nor{\kT}_p\leq\frac{H_p(m)}{|m|}\int_{S^{2n-1}}\nor{\Omega(\zeta)}\,d\zeta\,.
$$
\end{theorem}

\begin{proof}[{\bf Sketch of the proof.}]
For each complex ``direction'' $\zeta\in S^{2n-1}$ we introduce the operator
$\kA_\zeta:L^p(\C^n,E)\rightarrow L^p(\C^n,E)$ by
$$
(\kA_\zeta f)(z)=\frac{|m|}{2\pi}\int_\C 
\left(\frac{\lambda}{|\lambda|}\right)^{-m}\frac{f(z-\lambda\zeta)}{|\lambda|^{2}}\,d\lambda\,.
$$
Iwaniec and Martin call it ``directional complex Hilbert transform''; strictly speaking though, it looks more like  
``directional {\it $m$-th power} of the complex Hilbert transform''. What matters more is that they prove the following:
\begin{equation}
\label{tavagna}
\nor{\kA_\zeta}_p=H_p(m),
\end{equation}
which is not so surprising, in view of \eqref{coppy} and \eqref{chains and things}.

The second key calculation in their proof, and the one we present here, shows how to pass from $\kA_\zeta$ to $\kT$. Compare with the classical method of rotation by Calder\'on and Zygmund. 

Fix $z\in\C^n$ and consider a test function $\Phi=\Phi_z:S^{2n-1}\rightarrow F$, defined by
$$
\Phi(\zeta)=\Omega(\zeta)\kA_\zeta f(z)\,.
$$
By the definition of $\kA_\zeta$ we get
$$
\Phi(\zeta)=\frac{|m|}{2\pi}\,\Omega(\zeta)\int_\C 
\left(\frac{\lambda}{|\lambda|}\right)^{-m}\frac{f(z-\lambda\zeta)}{|\lambda|^{2}}\,d\lambda\,.
$$
It is easy to verify that for $\Lambda\in\cL(E,F)$ and $g:\C\rightarrow E$ we have $\Lambda\int_\C g=\int_\C\Lambda\circ g$.
Apply this simple fact to $\Lambda=\Omega(\zeta)$ and then use property \eqref{vrat}. The outcome is:
\begin{equation}
\label{punjabi techno}
\Phi(\zeta)=\frac{|m|}{2\pi}\int_\C 
|\lambda|^{2n-2}\Omega(\lambda\zeta)f(z-\lambda\zeta)\,d\lambda\,.
\end{equation}

At this point we need the following identity: for even $m\in\N$ we have
\begin{equation}
\label{pletnev}
2\pi\int_{\R^m}F(x)\,dx=\int_{S^{m-1}}\int_\C F(\lambda\sigma) |\lambda|^{m-2}\,d\lambda\,d\sigma.
\end{equation}
To prove this start with the well-known formula \cite[Appendix C.3, Theorem 4]{E}
$$
\int_{\R^m}F(x)\,dx=\int_0^\infty\int_{rS^{m-1}}F\,d\sigma\,dr.
$$
By using spherical coordinates we get
$$
\int_{\R^m}F(x)\,dx=\int_{S^{m-1}}\int_0^\infty F(r\sigma)\,r^{m-1}\,dr\,d\sigma.
$$
We integrate as $\int_0^{2\pi}\,d\f$ and obtain
$$
2\pi\int_{\R^m}F(x)\,dx=\int_0^\infty\int_0^{2\pi}\int_{S^{m-1}} F(r\sigma)\,r^{m-1}\,d\sigma\,d\f\,dr.
$$
Since $m$ is even, we may think of $\sigma\in\R^m$ as $\sigma\in\C^n$, where of course $m=2n$. Then it makes sense, for any fixed $\f$, to introduce the new variable $\sigma'=e^{-i\f}\sigma$, of course understood as a tensor product of rotations applied to a vector in $\R^m$. We get
$$
\aligned
2\pi\int_{\R^m}F(x)\,dx
&=\int_0^\infty\int_0^{2\pi}\int_{S^{m-1}} F(re^{i\f}\sigma')\,d\sigma'\,d\f\,r^{m-1}\,dr\\
&=\int_{S^{m-1}} \int_0^{2\pi}\int_0^\infty F(re^{i\f}\sigma)r^{m-2}r\,dr\,d\f\,d\sigma\,,
\endaligned
$$
which we recognize as \eqref{pletnev}.

By recalling that $|\zeta|=1$, we may apply \eqref{pletnev} to \eqref{punjabi techno} and arrive at
$$
\int_{S^{2n-1}}\Phi_z(\zeta)\,d\zeta
=|m|\int_{\C^n}\Omega(\zeta)f(z-\zeta)\,d\zeta
=|m|(\kT f)(z)\,.
$$
Finally we apply Minkowski's integral inequality:
$$
\aligned
\nor{\kT f}_p & \leq\frac1{|m|}\int_{S^{2n-1}}\nor{\Phi_\cdot(\zeta)}_p\,d\zeta\\
&\leq\frac1{|m|}\int_{S^{2n-1}}\nor{\Omega(\zeta)}\nor{\kA_\zeta f}_p\,d\zeta\\
&\leq\frac{H_p(m)}{|m|}\left(\int_{S^{2n-1}}\nor{\Omega(\zeta)}\,d\zeta\right)\nor{f}_p\,.
\endaligned
$$ 
For the last inequality we applied \eqref{tavagna}. 
\end{proof}

\subsubsection{Our results}
\label{Mitropoulos}

The following estimates give partial answers to the questions discussed above. They will be proven in Section \ref{nas na babu promenyal}.

\begin{theorem}[\cite{DPV}]
\label{goin' to louisiana}
We have $H_p(k)\,\leqsim\, |k|^{1-2/p^*}(p^*-1)$, where $k\in\Z\backslash\{0\}$ and $p>1$.
\end{theorem}

\medskip
When $k$ is {\it even} the one-sided inequality from the above theorem 
is mirrored by the same estimate from below, thus establishing the correct asymptotic behaviour of $\nor{T^n}_p$ simultaneously in $n$ and $p$:

\begin{theorem}[\cite{DPV, D1}]
\label{izvorska}
We have $\nor{T^n}_p\sim n^{1-2/p^*}(p^*-1)$, for 
$n\in\N$ and $p>1$.
\end{theorem}

Our proving Theorem \ref{izvorska} in Section \ref{nas na babu promenyal} will eventually lead us to conjecture that 
$$
\nor{T^n}_p=\frac{(1/q)_n}{(1/p)_n}
$$
for all $p\geq 2$, $q=p/(p-1)$ and $n\in\N$. See Conjecture \ref{gryanyem bratci} below. Here
\begin{equation}
\label{psy}
(a)_n=\frac{\Gamma(a+n)}{\Gamma(a)}=a\,(a+1)\cdot\hdots\cdot(a+n-1)
\end{equation}
is the {\it Pochhammer symbol}. 

\medskip
Let us explain why the case of odd $k$ is more difficult. Recall that, by definition, $H_p(k)=\nor{R^k}_{\cB(L^p(\C))}$. 

If $k$ is {\it even} we have $R^{2l}=T^l$. The operator $T$ is characterized by $\pd_{\bar z}f\mapsto\pd_zf$ (Exercise \ref{hor sretenskogo monostirya}). This property is used in proving lower estimate of $H_p(k)$ for even $k$, see Proposition \ref{hm}.

If $k$ is {\it odd} we have $R^{2l+1}=RT^l$. Now, $R$ has a characterization in terms of $(-\Delta)^{1/2}f\mapsto 2\pd_zf$, however $(-\Delta)^{1/2}$ is not a differential operator, which makes the calculations less explicit. See \cite[Example 3.7.5]{D} for details regarding $(-\Delta)^{1/2}$.

\medskip
Still, we believe that $H_p(k)\,\sim\, |k|^{1-2/p^*}(p^*-1)$ might hold for {\it all} $k\in\Z\backslash\{0\}$, not only the even ones which Theorem \ref{izvorska} addresses. At the moment however this question is still open. 
We have trivial inequalities $H_p(2l+2)\leq H_p(2l+1)H_p(1)$ and $H_p(4l+2)\leq H_p(2l+1)^2$. It follows from \cite[Theorem 1.1]{IM} that $H_p(1)\leqsim p^*$. Therefore we obtain, as a simple consequence of Theorems \ref{goin' to louisiana} and \ref{izvorska}, the following estimate of $H_p(k)$ from below in the case when $k$ is {\it odd}:

\begin{corollary}
$H_p(2l+1)\geqsim \max\left\{
l^{1-2/p^*},l^{1/2-1/p^*}\sqrt{p^*-1}\right\}$.
\end{corollary}

\medskip
We conclude by a weak-type (1,1) result analogous to Theorem \ref{izvorska}.

\begin{theorem}
\label{v-o-d-a}
We have $\nor{T^n}_{1\rightarrow1,\infty}\,\sim\, n$, for 
$n\in\N$.
\end{theorem}

\subsection{Weighted estimates of $T^n$}
\label{american pie}

Based on the results announced in this section (estimates of $T^n$ on $L^p$ and of $T$ on $L^p(w)$) 
we find it natural to consider estimating powers of $T$ on weighted spaces. It turns out an adequate modification of the techniques from \cite{DV1} and \cite{DPV} yields the following generalization of Theorem \ref{novljan}:

\begin{theorem}[\cite{D2}]
\label{jednamladost}
For every $p>1$ there is $C(p)>0$ such that for every $n\in\Z\backslash\{0\}$ and $w\in A_p$ we have
\begin{equation}
\label{prituri_se}
\nor{T^n}_{B(L^p(w))}\leqslant C(p)\, |n|^3\,
[w]_p^{p^*\!/p}\,.
\end{equation}
\end{theorem}
The proof will not be presented in this note. The theorem above has an improvement due to Hyt\"onen \cite[Corollary 7.5]{Hy} who proved, for every $n\in\N$, $p>1$ and $\e>0$, the estimate
\begin{equation}
\label{38}
\nor{T^n}_{B(L^p(w))}\leqslant C(p,\e)\, |n|^{1+\e}\,
[w]_p^{p^*\!/p}\,.
\end{equation}
More recently \cite{Hy2} he observed that one may improve this to 
\begin{equation*}
\label{39}
\nor{T^n}_{B(L^p(w))}\leqslant C(p)\, |n|(1+\log |n|)^2\,
[w]_p^{p^*\!/p}\,.
\end{equation*}
Even more recently, Hyt\"onen, Roncal and Tapiola \cite[Corollary 4.2]{HRT} improved this estimate to
\begin{equation}
\label{40}
\nor{T^n}_{B(L^p(w))}\leqslant C(p)\, |n|(1+\log |n|)\,
[w]_p^{p^*\!/p}\,.
\end{equation}

On the other hand, Hyt\"onen also noticed (personal communication, 2011) that, for any $\e>0$, the estimate 
$$
\nor{T^n}_{B(L^2(w))}\leqslant C(\e)\, n^{1-\e}\,
[w]_2\,
$$
is {\it not} valid for all $n\in\N$ and $w\in A_2$. Indeed, suppose the contrary. Then, by Theorem \ref{extrapol}, we would get
$$
\nor{T^n}_{B(L^p(w))}\leqslant C(p,\e)\, n^{1-\e}\,
$$
for any $n\in\N$, $p>2$ and all $w\in A_p$. In particular, in the unweighted case ($w\equiv 1$) we would have
$$
\nor{T^n}_{B(L^p(\C))}\leqslant C(p,\e)\, n^{1-\e},
$$
for all $n\in\N$ and $p>2$. However, Theorem \ref{izvorska} gives
$$
\nor{T^n}_{B(L^p(\C))}\geqslant C(p)\, n^{1-2/p},
$$
therefore the combination of the last two inequalities would lead to
$$
n^{1-2/p}\,\leqslant C(p,\e)\, n^{1-\e}
$$
for all $n\in\N$ and $p>2$, which is clearly impossible. 

\medskip
It remains an open question whether one can get rid of $\e$ in \eqref{38} or, equivalently, of the logarithmic term in \eqref{40}. See \cite[Section 6.2]{Hy2} and \cite[Conjecture 4.6]{HRT}.

\medskip
Hyt\"onen's observation in \cite[Corollary 7.5]{Hy} arrives from a much stronger result -- the $A_2$ theorem for general Calder\'on-Zygmund operators -- whose original proof, also due to Hyt\"onen \cite{Hy1}, is way more involved than our proof of Theorem \ref{jednamladost}, while on the other hand it also uses a (far subtler) version of the averaging technique.

\section{Weighted estimate for $T$: proof of Theorem \ref{novljan}.}
\label{Evtushenko}

In view of the above-cited extrapolation result (Theorem \ref{extrapol}), in order to prove Theorem \ref{novljan} it suffices to consider the case $p=2$ 
(recall that
 $$
 [w]_2=\sup_{Q\subset\R^2}\,
 \avg{w}_Q\bavg{w^{-1}}_Q\,,
 $$
where the supremum is taken over all squares in $\C$ regardless of their orientation):

\begin{theorem}[\cite{PV,DV1}]
\label{zujovic-crni}
For any $w\in A_2$ we have
$$
\nor{T}_{\cB(L^2(w))}\,\leqsim\, [w]_{A_2}.
$$
This estimate is sharp.
\end{theorem}

This theorem was first proven by Petermichl and Volberg \cite{PV} by means of the Bellman function technique. Here we present another approach, essentially from \cite{DV1}. The key result, and one of the inspirations for the proof in \cite{DV1}, is an inequality by Wittwer \cite{W} (see Theorem \ref{sve prolazi} 
below), saying, roughly,  that martingale transforms admit on $L^2(w)$ linear estimates in terms of $[w]_{A_2}$, i.e. just the kind of estimates we want for $T$. Its proof is an example of the Bellman function technique and is modelled on the paper by Nazarov, Treil and Volberg \cite{NTV}. We present it in Section \ref{tamo dolje niz mahalu}.

Another source of inspiration was a paper by Petermichl, Treil and Volberg \cite{PTV}, where (first-order) Riesz transforms were represented as averages of so-called dyadic shifts. That result, in its turn, was an upgrade of the theorem by Petermichl \cite{P}, where the Hilbert transform was represented as an average of dyadic shifts. While first-order Riesz transforms are singular operators with odd kernels, the Ahlfors-Beurling operator $T$ has an even kernel and thus dyadic shifts cannot be turned into $T$ by the averaging method. It turns out that an appropriate replacement are precisely the martingale transforms.

\subsection{Haar functions and martingale transforms}
\label{flash gordon}

For each interval $I\subset\R$ let
$I_-, I_+$ be its left and its right half, respectively. Denote by
$\chi_I$  its characteristic function and by $h_I$ its {\it Haar function}, which is to say 
$h_I=|I|^{-1/2}(\chi_{I_+}-\chi_{I_-})$.
An interval in $\R$ is called {\it dyadic} if it is of the form $[k2^m,(k+1)2^m)$ for some integers $k,m$.
The collection $\mn{h_I}{I\subset[0,1]\text{ dyadic}}$ is called the (standard) {\it Haar system} on $[0,1]$. Together with the constant function $\mathbf 1$ it forms a complete orthogonal system in $L^2([0,1])$; see Muscalu-Schlag \cite[Section 8.4]{MS} or Grafakos \cite[Section 5.4]{G1}.

Typically Haar functions are defined on intervals in $\R$. By working in $\R^2$ one is led to define them on squares \cite{DV1}. 
Let ${\mathcal L}$ be the standard dyadic lattice in $\R^2$, i.e., the collection of squares $Q=I\times J$, where both $I$ and $J$ are dyadic intervals of the same length.
The corresponding {\it Haar functions} are then
$$
h_Q^0=\frac1{\sqrt{|I|}}\, \chi_I\otimes h_J\,,\hskip 15pt
h_Q^+=\sqrt{\frac2{|J|}}\, h_I\otimes \chi_{J_+}\,,\hskip 15pt
h_Q^-=\sqrt{\frac2{|J|}}\, h_I\otimes \chi_{J_-}\,.
$$
The constants in the front are chosen so that the functions are normalized in (the usual) $L^2$.

We define {\it martingale transforms} $M_\sigma$ to be the operators
\begin{equation}
\label{sopci}
M_\sigma f(\zeta)=\sum_{Q\in{\mathcal L}\atop *\in\{0,+,-\}}\sigma_Q^*\sk{f}{h_Q^*}h_Q^*(\zeta)\,,
\end{equation}
where $\sigma_Q^*\in \overline\Delta$
and $h_Q^*$ are Haar functions.

Wittwer \cite{W} originally addressed weighted estimates of martingale transforms on the line and associated with the standard dyadic lattice in $\R$. Her result (here Theorem \ref{sve prolazi}) is formulated and proven in Section \ref{tamo dolje niz mahalu}. 
A careful reading of 
the proof 
shows that the same 
holds for the ``planar" operators $M_\sigma$ defined above: 

\begin{theorem}[Wittwer \cite{W}]
\label{Wittwer}
There exists $C>0$ such that for any family of coefficients $\sigma$ and any $w\in A_2$,
\begin{equation*}
\label{satanas}
\nor{M_\sigma}_{B(L^2(w))}\leqslant C[w]_{A_2}\,.
\end{equation*}
\end{theorem}

\subsection{Main idea} 
\label{maide}
Briefly, we represent $T$ as an ``average'' of operators that on $\cB(L^2(w))$ admit linear estimates in terms of $[w]_{A_2}$.

Instead of a dyadic lattice let us for a moment consider a unit
{\it grid} ${\mathcal G}$ of squares. This is a family of
squares $I\times J$, where $I$ and $J$ are
dyadic intervals of unit length.

Our plan is to consecutively average $0-$type projections (i.e. operators of the form 
$$
f\mapsto\sum_{Q\in \cF}\sk{f}{h_Q^0}h_Q^0,
$$
where $\cF$ is some dyadic collection of squares) over (ever wider) families of $\cF$'s comprising
\begin{itemize}
\item
translated unit grids
\item
translated grids of arbitrary (fixed) size
\item
translated and rotated grids of arbitrary (fixed) size
\item
all lattices of a fixed calibre
\item
all lattices.
\end{itemize}
More precisely, introduce, for $f\in\cS$, 
$$
\cP_0 f := \sum_{Q\in {\mathcal G}} \sk{f}{h_Q^0}h_Q^0\,.
$$
Note that, at least in the unweighted case, $\cP_0$ is simply the orthogonal projection onto the subspaces generated by all Haar functions of type 0 having as supports dyadic squares of unit size.
Note also that we have one sole set of coefficients, namely $\sigma_Q^0=1$ and $\sigma_Q^+=\sigma_Q^-=0$ for all such $Q$.

This $\cP_0$ is our starting operator, the one we will conjugate by translations, dilations, rotations, and then average. In the process we will take care of:
\begin{itemize}
\item
approaching $T$, meaning that after the final average our resulting operator will {\it equal} the operator $T$, up to a nonzero multiplicative constant;
\item
justifying the name ``averaging'', i.e. preserving at any phase the same estimates that the original operator, $\cP_0$, enjoys on $L^2(w)$, $w\in A_2$, by Theorem \ref{Wittwer}.
\end{itemize}
The first part (the averaging) will be carried out in Section \ref{chantee}, while the second one (estimates) in Section \ref{capellanus}.

The idea presented here comes from the simple observation that $T$ is a (p.v.) integral operator whose kernel 
$k(z,w)=(z-w)^{-2}$ is characterized, up to a nonzero multiplicative constant, by the following properties:
\begin{itemize}
\item
translation invariance (convolution kernel): $k(z,w)=k(z+u,w+u)$, $\forall z,w,u\in\C$;
\item
``weighted dilation invariance'' (homogeneity of order -2): $k(\lambda z,\lambda w)=\lambda^{-2}k(z,w)$ for any $\lambda>0$;
\item
``weighted rotation invariance'': $k(e^{i\psi} z,e^{i\psi} w)=e^{-2i\psi}k(z,w)$ for any $\psi\in\R$.
\end{itemize}

\subsection{The averaging}
\label{chantee}

We start by carrying out the plan announced above.\\

\noindent\framebox{Averaging over translated unit grids.}
For $t\in\R^2$ define $\cP_t$ as the conjugation of $\cP_0$ by the translation $\tau_t$, explicitly,
\begin{equation}
\label{oto}
\cP_t=\tau_t\circ\cP_0\circ\tau_t^{-1}.
\end{equation}
It is straightforward that, for $f\in\cS$, 
$$
\cP_tf=\sum_{Q\in {\mathcal G}} \sk{f}{\tau_t h_Q^0}\tau_t h_Q^0\,.
$$
Observe that $\tau_t h_Q^0=h_{Q+t}^0$, therefore we may write
\begin{equation}
\label{fabijan}
\cP_tf=\sum_{Q\in {\mathcal G}_t} \sk{f}{h_Q^0}h_Q^0\,,
\end{equation}
where ${\mathcal G}_t:={\mathcal G}+t$ is the {\it grid} of
unit squares such that one of them contains the point $t$ as one of its vertices.

The family $\Omega:=\mn{{\mathcal G}_t}{t\in\R^2}$
of all unit grids naturally corresponds to the torus
$\R^2/\Z^2$, which is of course in one-to-one
correspondence with the square
$[0,1)^2$. Thus we are able to regard $\Omega$ as a
probability space $[0,1)^2$ where the probability measure equals 
the Lebesgue measure.

Consider the ``mathematical expectation" of the
``random variable" $\cP$\,: for $f\in\cS$ and $x\in\R^2$ define
$$
(\EP f)(x)=\int_\Omega(\cP_tf)(x)\, dt\, .
$$
It makes good sense to call the process of passing from all $\cP_t$ to $\EP$ the {\it
averaging}. The structure of the operator $\EP$ is revealed in the following proposition.

\begin{proposition}
\label{F}
The operator $\EP$ is a convolution operator with the kernel 
$-\beta\otimes\alpha$, where
$\alpha=h_0*h_0$ 
and
$\beta=\chi_0*\chi_0$.
Here $\chi_0$ and $h_0$ stand (respectively) for the characteristic and Haar function of the interval $(-1/2,1/2]$. 
That is, we have
$$
\EP f=f*F,
$$
where $f\in\cS(\R^2)$ and $F(x,y)=-\beta(x)\alpha(y)$.
\end{proposition}

\begin{proof}
Choose $t=(t_1,t_2)\in\R^2$ and $Q=I\times J\in {\mathcal
G}_t$. Then
$$
\sk{f}{h_Q^0}=\int_\R\!\int_\R f(s_1,s_2)h_Q^0(s_1,s_2)\,
ds_1\, ds_2 = \int_J\!\int_I f(s_1,s_2)\chi_I(s_1)h_J(s_2)\,
ds_1\, ds_2,.
$$
Thus for (fixed) $f\in
\cS(\R^2)$ and $x=(x_1,x_2)\in\R^2$ we have, by \eqref{fabijan},
$$
(\cP_t f)(x) =\sum_{Q\in {\mathcal G}_t}
\int_\R\!\int_\R f(s_1,s_2)\chi_I(s_1)h_J(s_2)\,
ds_1\, ds_2 \, \cdot h_Q^0(x).
$$
The expression under the summation sign in the last row is
nonzero for exactly one $Q\in {\mathcal G}_t$; namely, one such that
$h_Q^0(x)\not =0$. This means that
$x=(x_1,x_2)\in Q$ and hence $x_1\in I$ and $x_2\in J$.
Thus $h_Q^0(x)=h_J(x_2)$. We thus obtain
\begin{equation}
\label{a}
(\cP_t f)(x)=\int_\R\!\int_\R f(s_1,s_2)
\chi_{I}(s_1)h_{J}(s_2)h_J(x_2)   
\,ds_1\, ds_2\, .
\end{equation}
Because ${\mathcal G}_t$ does not change if we increase or
decrease any component of $t$ by 1, we may assume that
$I=[t_1-1,t_1)$ and $J=[t_2-1,t_2)$. Denoting
$I_0=[-1/2,1/2)$, this assumption implies
$$
I=t_1-\frac{1}{2}+I_0 
\hskip 30pt {\rm and} \hskip 30pt
J=t_2-\frac{1}{2}+I_0\, .
$$
Now let $\chi_0$ and $h_0$ be as in the formulation of the
proposition and let $k_0:=-h_0$. Then
$$
\chi_I(s)=\chi_0(t_1-1/2-s)
$$
$$
h_I(s)=k_0(t_1-1/2-s)
$$
for all $s\in\R^2$. The analogue pair of identities is valid also for $J$,
of course.

The point here is that we modified the
expressions on the left to look more like a part of a
convolution integral with $t_1$ and
$t_2$ as integration variables and $s$ as a center of
convolving.

Together with (\ref{a}), the last two equalities imply
\begin{equation}
\label{brisi}
(\cP_t f)(x)
=\int_\R\!\int_\R f(s_1,s_2)\,
\chi_0(t_1-1/2-s_1)k_0(t_2-1/2-s_2)k_0(t_2-1/2-x_2)
\,ds_1\,
ds_2.
\end{equation}
Recall that $x_1\in I=[t_1-1,t_1)$ and $x_2\in
J=[t_2-1,t_2)$. Hence $x_j<t_j\leq x_j+1$ for $j=1,2$.
Averaging in our case means integrating over all admissible
$t_j$. Therefore,
$$
(\EP f)(x)=\int_{x_2}^{x_2+1}\int_{x_1}^{x_1+1}
(\cP_{(t_1,t_2)}f)(x)\,dt_1\, dt_2\, .
$$
By using the most recent expression for $(\cP_tf)(x)$,
changing variables (from $t_j$ to $t_j-1/2$) 
and applying Fubini's theorem, we obtain
\begin{equation}
\label{b}
(\EP
f)(x_1,x_2)=
\int_\R\!\int_\R f(s_1,s_2)\
\int_{x_2-1/2}^{x_2+1/2}\int_{x_1-1/2}^{x_1+1/2}[A]\,
dt_1\, dt_2\ \
ds_1\,ds_2  \, ,
\end{equation}
where
$$
A=k_0(t_2-x_2)\,\chi_0(t_1-s_1)\,k_0(t_2-s_2).
$$
Let us take a closer look at the inner integral. We have
$$
\aligned
\int_{x_2-1/2}^{x_2+1/2}\int_{x_1-1/2}^{x_1+1/2}& [A]\,dt_1\, dt_2\\
&= \int_{x_1-1/2}^{x_1+1/2} \chi_0(t_1-s_1) \, dt_1
\cdot
\int_{x_2-1/2}^{x_2+1/2}k_0(t_2-x_2)\, k_0(t_2-s_2) \, dt_2\\
&=\int_{-1/2}^{1/2} \chi_0(y_1-u_1) \, du_1
\cdot
 \int_{-1/2}^{1/2} k_0(-u_2)\,k_0(y_2-u_2) \, du_2,
\endaligned
$$
where in the last integral we introduced a new pair of variables, namely
\begin{equation}
\label{zlatko}
\aligned
u_j & =x_j-t_j\\
y_j & =x_j-s_j
\endaligned
\end{equation}
for $j=1,2$. Clearly, this is the same as $-(\chi_0*\chi_0)(y_1)(k_0*k_0)(y_2)$. 
Observe that $k_0*k_0=h_0*h_0$ and recall that $y_j$'s were given in \eqref{zlatko}. Hence we proved 
$$
\int_{x_2-1/2}^{x_2+1/2}\int_{x_1-1/2}^{x_1+1/2} [A]\,dt_1\, dt_2=
F(x_1-s_1,x_2-s_2),
$$
where $F$ is as in the formulation of the proposition. By \eqref{b}, this finishes the proof.
\end{proof}

Graphs of functions $\alpha$ and $\beta$ are shown as
Figures \ref{fig1} and \ref{fig2}, respectively.


\setlength{\unitlength}{1mm}
\begin{figure}
\begin{center}
\begin{picture}(140,80)(-63.5,-35)
\put(0,0){\vector(1,0){60}}
\put(0,0){\vector(-1,0){60}}
\put(0,0){\vector(0,1){30}}
\put(0,0){\vector(0,-1){40}}
\multiput(-0.5,16)(3.72,0){5}{\line(1,0){1.5}}
\multiput(16,-0.5)(0,3.7){5}{\line(0,1){1.5}}
\multiput(16,-0.5)(0,3.7){1}{\line(0,1){1}}
%
%
%
%
\put(15,-5){$\frac12$}
\put(-3,15){$\frac12$}
\multiput(-16,-0.5)(0,3.7){1}{\line(0,1){1}}
\put(-20.5,-5){$-\frac12$}

\put(60,2){$x$}
\put(31,-4){$1$}
\put(-7,-33){$-1$}

\thicklines
\put(0,-32){\line(1,3){16}}
\put(16,16){\line(1,-1){16}}
\put(0,-32){\line(-1,3){16}}
\put(-16,16){
\line(-1,-1){16}}


\end{picture}
\end{center}
\caption{Graph of $\alpha$}
{\protect{\label{fig1}}}%
\end{figure}
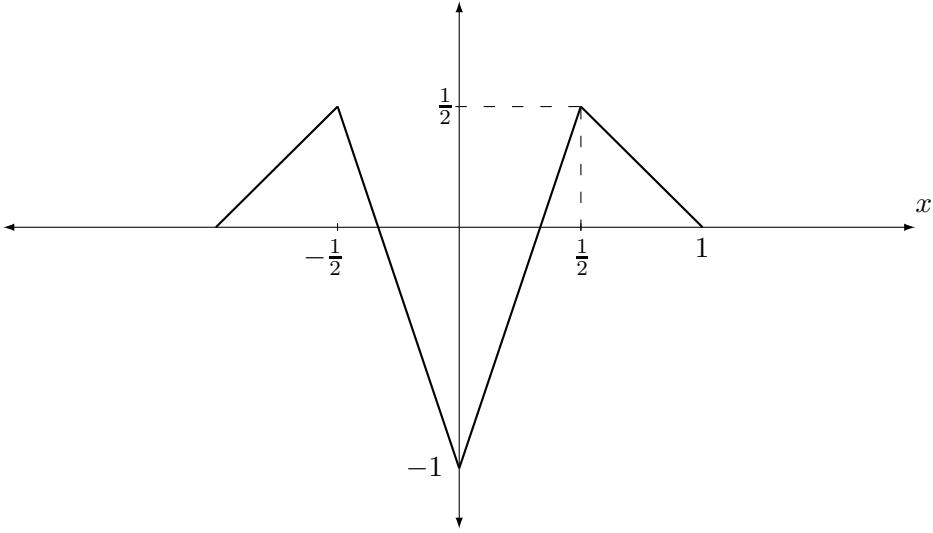



\setlength{\unitlength}{1mm}%
\begin{figure}
\begin{center}
\begin{picture}(140,45)(-63.5,-10)
\put(0,0){\vector(1,0){60}}
\put(0,0){\vector(-1,0){60}}
\put(0,0){\vector(0,1){40}}
\put(0,0){\vector(0,-1){15}}
\thicklines
\put(0,32){\line(1,-1){32}}
\put(0,32){\line(-1,-1){32}}


\put(60,2){$x$}
\put(31,-4){$1$}
\put(-33,-4){$1$}
\put(2,31){$1$}

\end{picture}
\end{center}
\caption{Graph of $\beta$}
{\protect{\label{fig2}}}%
\end{figure}
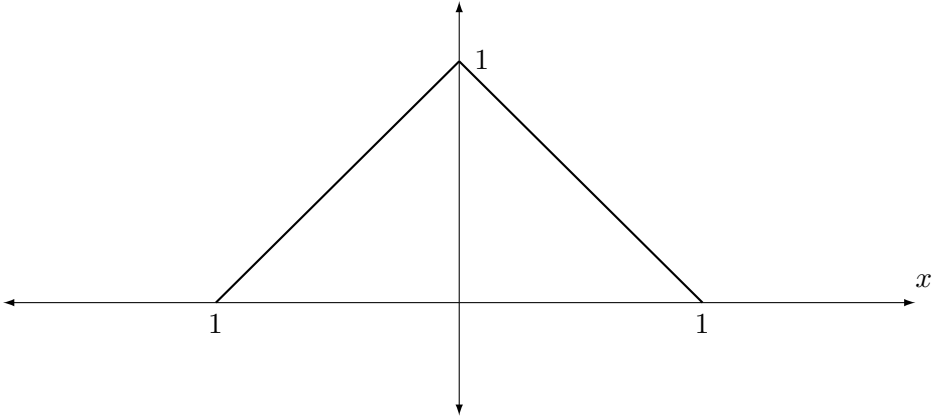


\bigskip\noindent\framebox{Averaging over translated grids of arbitrary (fixed) size.}
Instead of the unit grid we may consider a grid of squares
with sides of an arbitrary length $\rho >0$. Denote such a grid by ${\mathcal
G}_t^\rho$ if $t\in\R^2$ is a vertex of one of its members. Henceforth we
will call $\rho$ the {\it size} of the grid and $t$ its {\it reference point}.
We introduce another family of operators, defined by
$$
\cP_t^\rho f := \sum_{Q\in {\mathcal G}_t^\rho}
\sk{f}{h_Q^0}h_Q^0\,.
$$
Since ${\mathcal G}_t^\rho=\rho{\mathcal G}+t$, it follows that, for any $Q\in {\mathcal G}_t^\rho$, 
$$
h_Q^0=\frac1\rho(\tau_t\circ\delta_\rho)h_{\widetilde Q}^0,
$$
where $\widetilde Q$ is some (unique) square belonging to ${\mathcal G}$. From here we get a generalization of \eqref{oto}, namely
\begin{equation}
\label{bejturane jado}
\cP_t^\rho=\Lambda_{\rho,t}\circ\cP_0\circ\Lambda_{\rho,t}^{-1}
\end{equation}
where $\Lambda_{\rho,t}:=\tau_t\circ\delta_\rho$. So $\cP_t^\rho$ is obtained by conjugating our initial operator $\cP_0$ by a composition of a translation {\it and} dilation. 

In a similar way, by writing ${\mathcal G}_t^\rho=\rho{\mathcal G}_{t/\rho}$, we may derive the representation
\begin{equation}
\label{zvonko}
\cP_t^\rho=\delta_{\rho}\circ\cP_{t/\rho}\circ\delta_{\rho}^{-1}.
\end{equation}

\medskip
Let us {\bf fix} $\rho>0$ and average operators $\cP_t^\rho$ over all admissible $t$. This means studying the operator $\EP^\rho$, given on $\cS$ by
$$
\EP^\rho f:=\frac1{\rho^2}\int_{\Omega^\rho}\cP_t^\rho f\,dt,
$$
where $\Omega^\rho:=[0,\rho)\times[0,\rho)$.

We have the following generalization of Proposition \ref{F}.

\begin{exercise}
\label{ro}
Fix $\rho>0$. Then the operator $\EP^\rho$ can be on $\cS(\R^2)$ expressed as
$$
\EP^\rho f=f*F^\rho, 
$$
where 
$$
F^\rho(x,y):=\frac{1}{\rho^2}\,F\left ( \frac{x}{\rho},
\frac{y}{\rho}\right )\,
$$
and $F$ is as in Proposition \ref{F}.
\end{exercise}

\noindent\framebox{Averaging over translated and rotated grids of arbitrary (fixed) size.} 
At this step we bring rotations into play. Recall the notation \eqref{prokofjev}. 
For $\rho>0$ define 
\begin{equation}
\label{deveta godina}
G^\rho(\xi)=\frac1{2\pi}\int_0^{2\pi}(U_\psi F^\rho)(\xi)e^{-2i\psi}d\psi
\end{equation}
and $G=G^1$. Observe that 
\begin{equation}
\label{opa cupa}
\text{supp }G^\rho\subset K_2(0,\rho\sqrt 2).
\end{equation}
Indeed, if $|\xi|_2>\rho\sqrt 2$, then $|\cO_{-\psi}\xi|_2>\rho\sqrt 2$ for every $\psi\in\R$, hence $|\cO_{-\psi}\xi|_\infty>\rho$, therefore $F^\rho(\cO_{-\psi}\xi)=0$ for every $\psi\in\R$.

Exercise \ref{ro} helps to establish the following formula.
\begin{exercise}
\label{edinburgh}
For every $\rho>0$, $f\in\cS(\R^2)$ and $x\in\R^2$ we have
$$
(f*G^\rho)(x)=\frac1{2\pi}\int_0^{2\pi}
[(U_\psi\circ\EP^\rho\circ U_\psi^{-1})f](x)\, e^{-2i\psi}\,d\psi.
$$
\end{exercise}

\begin{remark}
From the definition of $\EP^\rho$ it follows that 
$$
[(U_\psi\circ\EP^\rho\circ U_\psi^{-1}) f](x)=
\frac1{\rho^2}\int_{\Omega^\rho}[(U_\psi\circ\cP_t^\rho\circ U_\psi^{-1}) f](x)\,dt.
$$
Now Exercise \ref{edinburgh} and \eqref{bejturane jado} imply
\begin{equation*}
f*G^\rho=\frac1{2\pi}\int_0^{2\pi}\frac1{\rho^2}\int_{\Omega^\rho}
\left(\Lambda_{\rho,t,\psi}\circ\cP_0\circ\Lambda_{\rho,t,\psi}^{-1}\right)f\,dt\,e^{-2i\psi}d\psi.
\end{equation*}
where $\Lambda_{\rho,t,\psi}:=U_\psi\circ\tau_t\circ\delta_\rho$. 
Therefore we could restate Exercise \ref{edinburgh} by saying that averaging the operators 
$$
\Lambda_{\rho,t,\psi}\circ\cP_0\circ\Lambda_{\rho,t,\psi}^{-1} : f\longmapsto \sum_{Q\in{\mathcal G}_t^\rho}\sk{f}{U_\psi h_Q^0}U_\psi h_Q^0
$$
where $\rho>0$ is fixed and $t,\psi$ run over all admissible values, returns a convolution operator whose (convolution) kernel is $G^\rho$. Note that the operators $\Lambda_{\rho,t,\psi}\circ\cP_0\circ\Lambda_{\rho,t,\psi}^{-1}$ are all martingale transforms, and that the supports of the ``rotated'' Haar functions are squares from the rotated grid $\cO_\psi{\mathcal G}_t^\rho$. 
\end{remark}

The following property of $G^\rho$ is simple yet important:
\begin{lemma}
\label{kamulatorsko kolo}
 $G^\rho$ has zero average on every circle centered at the origin. 
\end{lemma}
\begin{proof}
Here it is convenient to resort to the ``complex'' notation. 
For every $R>0$,
\begin{equation}
\aligned
\int_0^{2\pi}G^\rho(Re^{i\f})\,d\f& =\frac1{2\pi}\int_0^{2\pi}e^{-2i\psi}\int_0^{2\pi}F^\rho(Re^{i(\f-\psi)})\,d\f\, d\psi\\
&=\frac1{2\pi}\int_0^{2\pi}e^{-2i\psi}\, d\psi\int_0^{2\pi}F^\rho(Re^{i\lambda})\,d\lambda\\
&=0.
\endaligned
\end{equation}\vskip -20pt
\end{proof}

\noindent\framebox{Averaging over all lattices of a fixed calibre.}
Only now, owing to Lemma \ref{kamulatorsko kolo}, will it make real sense to introduce the ``kernel for averages over whole lattices''. Define
\begin{equation}
\label{frank martin}
 k^r(x):=\sum_{n=-\infty}^\infty G^{2^nr}(x),
\end{equation}
in the sense that
\begin{equation}
\label{lumbajla}
 k^r(x)=\lim_{M\rightarrow\infty}k_M^r(x),
\end{equation}
where
\begin{equation}
\label{mandarine-citrus}
k_M^r(x):=\sum_{n=-\infty}^M G^{2^nr}(x).
\end{equation} 

The fact that $k^r*$ is a sum of operators, obtained by
averaging over grids of size $r\cdot 2^n$, hints at $k^r*$
itself being an average, this time over {\sl unions} of
these grids, i.e. lattices of calibre $r$. While it is not clear what could be a probability space corresponding to all lattices of a fixed calibre, we define the above-said average as a limit of averages of truncated lattices. Then the statement makes sense and holds, as will be shown.

\medskip
For every $x\ne 0$ and $M\in\Z$, the sum $k_M^r(x)$ is finite (i.e. has only finitely many nonzero terms). Indeed, from 
\eqref{opa cupa} it emerges that $G^{2^nr}(x)\ne 0$ implies 
$$
n\geq\log_2\frac{|x|_2}r-\frac12=:C(x,r)>-\infty.
$$
Therefore, with $N(x,r):=[C(x,r)]$, where $[y]$ is the integer part of $y\in\R$,
$$
 k_M^r(x)=\sum_{n=N(x,r)}^M G^{2^nr}(x).
$$
Consequently,
\begin{equation}
\label{feriza}
 k^r(x)=\lim_{M\rightarrow\infty}\sum_{n=N(x,r)}^M G^{2^nr}(x).
\end{equation}

\begin{remark}
\label{daje}
Since $N(x,r)$ rises as $|x|$ rises, this shows that we also have 
$$
 k_M^r(y)=\sum_{n=N(x,r)}^M G^{2^nr}(y)
$$
for any $y\in\R^2$ such that $|y|\geq|x|$. Consequently, $ k_M^r$ is continuous on $\R^2\backslash\{0\}$ for any $M\in\Z$ and $r>0$, since on any set of the form $\{|y|>\e\}$ it is a finite sum of continuous functions $G^{2^nr}$.
\end{remark}

\begin{lemma}
\label{saban bajramovic}
The sum in \eqref{lumbajla} converges absolutely and uniformly on 
every complement of a neighbourhood of $0$
and satisfies the estimate $| k^r(x)|\leqsim |x|^{-2}$, uniformly in $r>0$. 
By Remark \ref{daje}, this implies that $k^r$ is continuous on $\R^2\backslash\{0\}$.
\end{lemma}
\begin{proof}
Take $x=(x_1,x_2)\in\R^2\backslash\{0\}$. 
Therefore, by \eqref{feriza},
$$
 k^r(x)
=\frac{1}{r^2}\sum_{n=N(x,r)}^{\infty}\frac{1}{4^n}\,G\Big(\frac{x}{2^nr}\Big). 
$$
Since $|F|\leq 1$, and thus $|G|\leq 1$, this implies
\[
|k^r(x)|
\leq\frac{1}{r^2}(1/4)^{N(x,r)}\,\frac{1}{1-1/4}
\leq\frac{4}{3r^2}(1/4)^{C(x,r)}
=\frac83\,\cdot\frac1{|x|_2^2}.
\qedhere
\]
\end{proof}

\begin{remark}
\label{six months}
Actually, the above proof shows that, for any nonzero $x\in\R^2$ and any subset $\cZ$ of $\Z$, 
$$
\sum_{n\in\cZ} |G^{2^nr}(x)|\leqsim |x|^{-2},
$$ 
with the implied constant independent of $x$, $r$ and $\cZ$. 
\end{remark}

\begin{lemma}
\label{zero average}
The average of $k^r$ over any circle centered at the origin is zero. That is, for every $R,r>0$ we have
$$
\int_0^{2\pi}k^r(Re^{i\psi})\,d\psi=0.
$$
\end{lemma}
\begin{proof}
We have, by \eqref{feriza},
$$
\int_0^{2\pi}k^r(Re^{i\psi})\,d\psi = \int_0^{2\pi}\lim_{M\rightarrow\infty}\sum_{n=N(Re^{i\psi},r)}^M G^{2^nr}(Re^{i\psi})\,d\psi\,.
$$
Because $N(Re^{i\psi},r)=N(R,r)$ for every $\psi\in\R$, and by Remark \ref{six months}, we may apply the dominated convergence theorem and conclude
$$
\int_0^{2\pi}k^r(Re^{i\psi})\,d\psi =\lim_{M\rightarrow\infty}\sum_{n=N(R,r)}^M  \int_0^{2\pi}G^{2^nr}(Re^{i\psi})\,d\psi.
$$
By Lemma \ref{kamulatorsko kolo}, this is equal to zero.
\end{proof}

\begin{proposition}
\label{jovo, momce mlado}
For every $r>0$, the function $k^r$ defines a tempered distribution, understood in the sense
$$
W_{k^r}(\phi)=\lim_{\e\rightarrow 0}\int_{\{|x|>\e\}}k^r(x)\phi(x)\,dx.
$$
\end{proposition}
\begin{proof}
We start with the splitting
$$
\int_{\{|x|>\e\}}k^r(x)\phi(x)\,dx=\int_{\{\e<|x|<1\}}+\int_{\{|x|\geq 1\}}.
$$
The second integral obviously converges, because $\phi\in{\mathcal S}$ and $k^r$ decays (quadratically), by Lemma \ref{saban bajramovic}. As for the first integral, by Lemma \ref{zero average} we may write
\begin{equation}
\label{crescent}
\int_{\{\e<|x|<1\}}k^r(x)\phi(x)\,dx=\int_{\{\e<|x|<1\}}k^r(x)|x|\,\frac{\phi(x)-\phi(0)}{|x|}\,dx.
\end{equation}
By Lemma \ref{saban bajramovic}, the function $k^r(x)|x|$ is integrable on $\{0<|x|<1\}$, while $\frac{\phi(x)-\phi(0)}{|x|}$
is bounded. 
Indeed, since $\phi(y)-\phi(0)=\sk{\nabla\phi(ty)}{y}$ for some $t\in (0,1)$, it follows that
\begin{equation}
\label{beru l' te djevojke}
\frac{|\phi(y)-\phi(0)|}{|y|}\leq\nor{\pd_{y_1}\phi}_\infty+\nor{\pd_{y_2}\phi}_\infty.
\end{equation}

Thus the limit $\e\rightarrow 0$ of \eqref{crescent} exists.
We proved, for any $r>0$ and $\phi\in{\mathcal S}$, that $W_{k^r}(\phi)$ is well defined and that
\begin{equation}
\label{gora mirisala}
W_{k^r}(\phi)=\int_{\{|x|<1\}}\widetilde\phi(x)k^r(x)|x|\,dx+\int_{\{|x|\geq 1\}}\phi(x)k^r(x)\,dx,
\end{equation}
where
$$
\widetilde\phi(x):=
\left\{
\begin{array}{ccc}
\displaystyle{\frac{\phi(x)-\phi(0)}{|x|}} & ; & x\ne 0;\\
0 & ; & x=0.
\end{array}
\right.
$$

Now it is easy to show that $W_{k^r}$ defines a tempered distribution:
by combining \eqref{gora mirisala}, Lemma \ref{saban bajramovic} and \eqref{beru l' te djevojke}  we get
$$
\aligned
|W_{k^r}(\phi)|
& \leqsim
\nor{\widetilde\phi}_\infty\int_{\{|x|<1\}}\,\frac{dx}{|x|}+
\sup_{y\in\R^2}|\phi(y)||y|
\int_{\{|x|\geq 1\}}\,\frac{dx}{|x|^3}\\
&\leqsim
\nor{\pd_{y_1}\phi}_\infty+\nor{\pd_{y_2}\phi}_\infty
+\sup_{y\in\R^2}|\phi(y)||y|.
\endaligned
$$
This means, e.g. \cite[Proposition 2.3.4.(b)]{G1}, that indeed $W_{k^r}\in\cS'(\R^2)$. 
\end{proof}

\begin{remark}
\label{evo na jad, jado, deveta godina}
As before, we see that 
$$
W_{k_M^r}(\phi)=\lim_{\e\rightarrow 0}\int_{\{|x|>\e\}}k_M^r(x)\phi(x)\,dx
$$
also defines a tempered distribution and that 
\begin{equation}
\label{daje me avava}
W_{k_M^r}(\phi)=\int_{\{|x|<1\}}\widetilde\phi(x)k_M^r(x)|x|\,dx+\int_{\{|x|\geq 1\}}\phi(x)k_M^r(x)\,dx.
\end{equation}
Moreover, $W_{k_M^r}\rightarrow W_{k^r}$ in the sense of tempered distributions as $M\rightarrow\infty$, i.e.
$W_{k_M^r}(\phi)\rightarrow W_{k^r}(\phi)$ for every $\phi\in{\mathcal S}$, e.g.  \cite[p. 110]{G1}. Indeed,
$$
\aligned
W_{k^r}(\phi)- W_{k_M^r}(\phi)
&=\int_{\{|x|<1\}}\widetilde\phi(x)[k^r(x)-k_M^r(x)]|x|\,dx+\int_{\{|x|\geq 1\}}\phi(x)[k^r(x)-k_M^r(x)]\,dx\\
&=\int_{\{|x|<1\}}\widetilde\phi(x)\sum_{n=M+1}^\infty G^{2^nr}(x)|x|\,dx+\int_{\{|x|\geq 1\}}\phi(x)\sum_{n=M+1}^\infty G^{2^nr}(x)\,dx.
\endaligned
$$
Because of Lemma \ref{saban bajramovic} and the fact that $\phi\in{\mathcal S}$, we may use the dominated convergence theorem and show that the above expression tends to zero as $M\rightarrow\infty$, for any $x\in\R^2$.
\end{remark}

The operators we are interested in are the convolutions with the distributions $W_{\kappa_M^r}$ and $W_{\kappa^r}$:
$$
\cK_M^r f:= f*W_{k_M^r}.
$$
$$
\cK^r f:= f*W_{k^r}.
$$
It is well known \cite[Theorem 2.3.20]{G1} that a convolution of a Schwartz function $\f$ and a tempred distribution $u$ is actually a ($C^\infty$) function and \cite[eq. (2.3.21)]{G1} that it is given by 
$$
x\mapsto u(\f(x-\cdot)).
$$ 
Therefore for $ f\in\cS$ we have, by Proposition \ref{jovo, momce mlado} and Remark \ref{evo na jad, jado, deveta godina}, 
$$
(\cK^r f)(x)=
\lim_{\e\rightarrow 0}\int_{\{|y|>\e\}}k^r(y) f(x-y)\,dy
$$
and 
\begin{equation}
\label{pfanner}
(\cK_M^r f)(x)=
\lim_{\e\rightarrow 0}\int_{\{|y|>\e\}}k_M^r(y) f(x-y)\,dy.
\end{equation}

\noindent\framebox{Averaging over all lattices.}
The final step is to average over dilations, in other words, over all calibres $r$. 
The family of all calibres (of dyadic lattices) is represented by any interval of the form $[a,2a)$.
For our purpose, the most appropriate 
measure turns out to be $dr/r$. This
makes all such intervals have the same measure ($\log 2$). 
Therefore when averaging over all (translated, rotated, dilated) lattices we may take, for instance, $[1,2)$ and define, for $f\in\cS$ and $x\in\R^2$, 
$$
(Sf)(x):=\int_1^2 (\cK^r f)(x)\,\frac{dr}r.
$$
This integral is well defined. 

\medskip
The operator $S$ is the final outcome of our averaging process.
We needed to sum over $n$ and integrate over $dt$, $d\psi$ 
and $dr$. The problem was to do this rigorously and in the correct order.

\medskip
As explained above, we need to show that with $S$ we ``hit the target'' - $T$.
 
\begin{proposition}
\label{prostranstvima gdje svi}
We have $S=cT$, where $c\in\R$ is nonzero.
\end{proposition}

\begin{proof}
We have
$$
(\cK^r f)(x)=\int_{\{|y|<1\}}\frac{f(x-y)-f(x)}{|y|}\,k^r(y)|y|\,dy+\int_{\{|y|\geq 1\}} f(x-y)k^r(y)\,dy.
$$
Now from Lemma \ref{saban bajramovic} it emerges that 
$$
Sf=\text{p.v. }f*k,
$$
where
$$
k(x)=\int_1^2 k^r(x)\,\frac{dr}r\,.
$$
We need to prove that $k$ is a (nonzero) constant multiple of the (p.v.) convolution kernel of $T$. We have, by
\eqref{frank martin}, Lemma \ref{saban bajramovic} and \eqref{deveta godina} that
$$
\aligned
k(x)&=\int_1^2 \sum_{n=-\infty}^\infty G^{2^nr}(x)\,\frac{dr}r\\
&= \sum_{n=-\infty}^\infty \int_1^2G^{2^nr}(x)\,\frac{dr}r\\
&= \sum_{n=-\infty}^\infty \int_{2^n}^{2^{n+1}}G^{s}(x)\,\frac{ds}s\\
& =\int_0^\infty G^s(x)\,\frac{ds}s\\
& =\frac1{2\pi}\int_0^\infty\int_0^{2\pi}F^s(\cO_{-\psi}x)e^{-2i\psi}d\psi\,\frac{ds}s.
\endaligned
$$
Therefore, in polar coordinates,
$$
\aligned
k(re^{i\f})&=\frac1{2\pi}\int_0^\infty\int_0^{2\pi}s^{-2}F(e^{-i\psi}re^{i\f}/s)e^{-2i\psi}d\psi\,\frac{ds}s\\
&=\frac{e^{-2i\f}}{2\pi}\int_0^\infty\int_0^{2\pi}F(re^{it}/s)e^{2it}dt\,\frac{ds}{s^3}\\
&=\frac{e^{-2i\f}}{\pi r^2}\cdot\frac1{2}\int_0^\infty\int_0^{2\pi}F(u e^{it})e^{2it}dt\,u\,du\,.
\endaligned
$$
Since the (p.v.) convolution kernel of $T$ is equal to 
$$
re^{i\f}\mapsto -\frac{e^{-2i\f}}{\pi r^2}\,,
$$
we indeed obtain $S=cT$, where, by the evenness of $F$, 
$$
c=-\frac1{2}\int_0^\infty\int_0^{2\pi}F(u e^{it})\cos2t\,dt\,u\,du.
$$

\medskip
It remains to be verified that $c\ne 0$. 
In cartesian coordinates, we can express $c$ as
$$
c=-\frac1{2}\int_\R\int_\R F(x,y)
\frac{x^2-y^2}{x^2+y^2}\,dx\,dy\,.
$$
In Proposition \ref{F} we saw that 
$
F(x,y)=-\beta(x)\alpha(y),
$
which leads to 
$$
c=-\frac12\int_\R\int_\R \alpha(x)\beta(y)
\frac{x^2-y^2}{x^2+y^2}\,dx\,dy\,.
$$
Since $\alpha$ and $\beta$ are even functions, supported on
the interval $[-1,1]$, we obtain
$$
c=-2\int_0^1\!\!\int_0^1 \alpha(x)\beta(y)
\frac{x^2-y^2}{x^2+y^2}\,dx\,dy\,.
$$
Figure \ref{fig2} shows that $\beta(y)=1-y$ on 
$[0,1]$. Combine this with the identity
$$
\frac{x^2-y^2}{x^2+y^2}=1-\frac{2y^2}{x^2+y^2}
$$
and the observation (following from Figure \ref{fig1}) that
$\int_0^1\alpha(x)\,dx=0$ to get
$$
c=4\int_0^1(1-y)y^2\int_0^1
\frac{\alpha(x)}{x^2+y^2}\,dx\,dy\,.
$$
For $y>0$, computation returns
$$
\aligned
C(y):&=\int_0^1\frac{\alpha(x)}{x^2+y^2}\,dx\\
&=\frac{1}{y}\left(\arctan\frac{1}{y}-2\arctan\frac{1}{2y}\right)+
2\log(4y^2+1)-\frac{1}{2}\log(y^2+1)
-3\log y-4\log 2\,.
\endaligned
$$
We now need only evaluate the integral
$$ 
\int_0^1(1-y)y^2C(y)\,dy\,.
$$
We can directly calculate this integral to find that it
equals
$$
\frac{1}{12}\left(\arctan 2-4\arctan\frac12
+\frac{15}{8}\log5-4\log2 \right)\,,
$$
which is approximately $-0^\cdot 042$\,.
\end{proof}

\subsection{Estimates for $S$ on $L^2(w)$}
\label{capellanus}
\begin{remark}
Recall that, by \eqref{pfanner} and \eqref{mandarine-citrus},
$$
\cK_M^r f=\text{ p.v. }f*\left(\sum_{n=-\infty}^M G^{2^nr}\right).
$$ 
Let us show that we also have
\begin{equation}
\label{costella}
\cK_M^r f=\sum_{n=-\infty}^M \left(f*G^{2^nr}\right).
\end{equation}
Indeed, this follows from 
$$
(\cK_M^r f)(x)=\int_{\{|y|<1\}}\frac{f(x-y)-f(x)}{|y|}\,k_M^r(y)|y|\,dy+\int_{\{|y|\geq 1\}} f(x-y)k_M^r(y)\,dy,
$$
Lemma \ref{saban bajramovic} and the dominated convergence theorem. 
\end{remark}

By \eqref{costella} and Exercise \ref{edinburgh}, 
$$
\aligned
(\cK_M^1f)(x)
&=\sum_{n=-\infty}^M\frac1{2\pi}\int_0^{2\pi}[(U_\psi\circ\EP^{2^n}\circ U_\psi^{-1})f](x)\, e^{-2i\psi}d\psi\\
&=\sum_{n=-\infty}^M\frac1{2\pi}\int_0^{2\pi}
\frac1{4^n}\int_{[0,2^n)^2}
\cP_t^{2^n}(U_\psi^{-1} f)(\cO_{-\psi}x)\,dt\,e^{-2i\psi}d\psi\,.
\endaligned
$$
As implied by \eqref{brisi} and \eqref{zvonko}, for every $f,x,n$, the function 
$$
(t,\psi)\mapsto
\cP_t^{2^n}(U_\psi^{-1} f)(\cO_{-\psi}x)\,e^{-2i\psi}
$$ 
is bounded on $[0,2^n)^2\times [0,2\pi)$, therefore it belongs to $L^1$ of that space and hence we can apply Fubini's theorem. With
\begin{equation}
\label{lubenica}
(\cR_t^{\rho}f)(x)=
\frac1{2\pi}\int_0^{2\pi}[(U_\psi\circ\cP_t^{\rho}\circ U_\psi^{-1})f](x)\, e^{-2i\psi}d\psi,
\end{equation}
we get
$$
(\cK_M^1f)(x)=\sum_{n=-\infty}^M\frac1{4^n}\int_{[0,2^n)^2}(\cR_t^{2^n}f)(x)\,dt.
$$
By using the fact that $\cP_\cdot^{\rho}$ is periodic and that $n\leq M$, we can continue as
\begin{equation}
\label{lato}
(\cK_M^1f)(x)=\sum_{n=-\infty}^M\frac1{4^M}\int_{[0,2^M)^2}(\cR_t^{2^n}f)(x)\,dt.
\end{equation}

Let us define, for 
$N\in\Z_M:=\mn{k\in\Z}{k\leq M}$,

\begin{equation*}
(\cK_{M,N}^1f)(x)=\sum_{n=N}^M\frac1{4^M}\int_{[0,2^M)^2}(\cR_t^{2^n}f)(x)\,dt.
\end{equation*}

The following result is straightforward yet useful.
\begin{lemma} 
\label{jazz for cows}
Let $A$ be an invertible affine transformation of $\R^2$, i.e. 
$Ax=\Lambda x+c$ for $x\in\R^2$, where $\Lambda$ is linear and invertible on $\R^2$ and $c\in\R^2$. 
Define the associated operator $\cC=\cC_A$ on $\cS(\R^2)$ by $\cC g=g\circ A$. Then
$$
\nor{\cC g}_{L^2(Sw)}=\frac1{\sqrt{\det\Lambda}}\,\nor{g}_{L^2(w)}
$$
for any $w\in A_2$.
\end{lemma}

\begin{proposition}
\label{fab}
For every $f\in\cS$, $M\in\Z$ and $w\in A_2$, the sequence $(\cK_{M,N}^1f)_{N\in\Z_M}$ is a Cauchy sequence in $L^2(w)$. 
\end{proposition}

\begin{proof}
We have, for $N_1<N_2\in\Z_M$,
$$
\nor{\cK_{M,N_1}^1f-\cK_{M,N_2}^1f}_{L^2(w)}=
\left(\int_{\R^2}\Big|\frac1{4^M}\int_{[0,2^M)^2}\sum_{n=N_1}^{N_2-1}(\cR_t^{2^n}f)(x)\,dt\Big|^2w(x)\,dx\right)^{1/2}.
$$
By the Minkowski's integral inequality, 
$$
\nor{\cK_{M,N_1}^1f-\cK_{M,N_2}^1f}_{L^2(w)}\leq
\frac1{4^M}\int_{[0,2^M)^2}
\left(\int_{\R^2}\Big|\sum_{n=N_1}^{N_2-1}(\cR_t^{2^n}f)(x)\Big|^2w(x)\,dx\right)^{1/2}
\,dt.
$$
From \eqref{lubenica} and by using the Minkowski's integral inequality again we get 
$$
\aligned
\left(\int_{\R^2}\Big|\right.&\left.\sum_{n=N_1}^{N_2-1}(\cR_t^{2^n}f)(x)\Big|^2w(x)\,dx\right)^{1/2}\\
& \leq
\frac1{2\pi}\int_0^{2\pi}
\left(\int_{\R^2}\bigg|
\Big[(U_\psi\circ\sum_{n=N_1}^{N_2-1}\cP_t^{2^n}\circ U_\psi^{-1})f\Big](x)\bigg|^2w(x)\,dx\right)^{1/2}
d\psi,
\endaligned
$$
where
$$
\aligned
\bigg(\hdots
\bigg)^{1/2}&=\left(\int_{\R^2}\bigg|
\Big[\sum_{n=N_1}^{N_2-1}\cP_t^{2^n}(U_\psi^{-1} f)\Big](y)\bigg|^2w(\cO_{\psi} y)\,dy\right)^{1/2}\\
&=\Nor{\sum_{n=N_1}^{N_2-1}\cP_t^{2^n}(U_{-\psi} f)}_{L^2(U_{-\psi} w)}\\
&=\Nor{\bigg(\tau_t\circ \sum_{n=N_1}^{N_2-1}\cP_0^{2^n}\circ\tau_{-t}\bigg)(U_{-\psi} f)}_{L^2(U_{-\psi} w)}\\
&=\Nor{\bigg(\sum_{n=N_1}^{N_2-1}\cP_0^{2^n}\bigg)(\tau_{-t}U_\psi f)}_{L^2(\tau_{-t}U_\psi w)}.
\endaligned
$$
The last inequality is due to Lemma \ref{jazz for cows}. To summarize, we so far proved that
\begin{equation}
\aligned
\label{avaj avaj mo cavo}
\nor{\cK_{M,N_1}^1f-\cK_{M,N_2}^1f}_{L^2(w)}\\
&\hskip -100pt \leq
\frac1{4^M}\int_{[0,2^M)^2}
\frac1{2\pi}\int_0^{2\pi}
\Nor{\Big(\sum_{n=N_1}^{N_2-1}\cP_0^{2^n}\Big)(\tau_{-t}U_\psi f)}_{L^2(\tau_{-t}U_\psi w)}
d\psi\,dt.
\endaligned
\end{equation}
Verifying the Cauchy condition means proving that
$$
\lim_{N_2\rightarrow-\infty}\sup_{N_1\in\Z_{N_2}}\nor{\cK_{M,N_1}^1f-\cK_{M,N_2}^1f}_{L^2(w)}
=0.
$$
Since for any integrable sequence $(f_k)_k$ on any space we clearly have $\sup\int|f_k|\leq\int\sup|f_k|$, 
we get from \eqref{avaj avaj mo cavo} that
\begin{equation}
\label{maljarkica}
\aligned
\sup_{N_1\in\Z_{N_2}}&\nor{\cK_{M,N_1}^1f-\cK_{M,N_2}^1f}_{L^2(w)}\\
&\leq
\frac1{4^M}\int_{[0,2^M)^2}
\frac1{2\pi}\int_0^{2\pi}
\sup_{N_1\in\Z_{N_2}}\Nor{\Big(\sum_{n=N_1}^{N_2-1}\cP_0^{2^n}\Big)(\tau_{-t}U_\psi f)}_{L^2(\tau_{-t}U_\psi w)}
d\psi\,dt.
\endaligned
\end{equation}
Certainly, we would now like to use the dominated convergence theorem and put $\lim_{N_2\rightarrow-\infty}$ inside the integral. For that purpose we need to know that functions
$$
\gamma_{N_2}(\psi,t):=\sup_{N_1\in\Z_{N_2}}\Nor{\Big(\sum_{n=N_1}^{N_2-1}\cP_0^{2^n}\Big)(\tau_{-t}U_\psi f)}_{L^2(\tau_{-t}U_\psi w)}
$$ 
have a $L^1(d\psi\, dt)$ majorant which is independent on $N_2$. 
In order to show that we invoke a result by Treil and Volberg \cite[Theorem 5.2]{TV} stating, roughly, that a (one-dimensional) Haar system is a Riesz basis in $L^2(w)$ when $w\in A_2$. Recall that a complete system of vectors $e_n$ in a Hilbert space is called a {\it Riesz basis} if 
$$
\Nor{\sum_n\lambda_n e_n}\sim\left(\sum_n|\lambda_n|^2\nor{e_n}^2\right)^{1/2}
$$
for all finitely supported sequences $\lambda_n$ of coefficients.

\begin{lemma}
Let $w\in A_2$, $g\in{\mathcal S}$ and 
$$
g_{N_1,N_2}=\Big(\sum_{n=N_1}^{N_2-1}\cP_0^{2^n}\Big)g.
$$
Then
$$
\lim_{N_2\rightarrow-\infty}\sup_{N_1\in\Z_{N_2}}\nor{g_{N_1,N_2}}_{L^2(w)}=0.
$$
\end{lemma}

\begin{proof}
Since $\sum_{n=N_1}^{N_2-1}\cP_0^{2^n}$ is a classical martingale transform (in $\R^2$), we can apply (a two-dimensional analogue of) the Wittwer's result \cite{W}, see Theorems \ref{sve prolazi} and \ref{Wittwer}, and conclude that
\begin{equation}
\label{bajro}
\nor{g_{N_1,N_2}}_{L^2(w)}\leqsim[w]_{A_2}\nor{g}_{L^2(w)},
\end{equation}
where the implied constants are independent of $N_1,N_2$.

On the other hand, we can write
$$
g_{N_1,N_2}=\sum_{n=N_1}^{N_2-1}\sum_{Q\in{\mathcal G}_0^{2^n}} \sk{g}{h_Q^0}h_Q^0.
$$
Assume for a moment that $g$ is of compact support. Then the sum above is finite and we may use the above-noted fact that Haar system $\{h_Q^j\}_{Q,j}$ is a Riesz basis. We get
\begin{equation}
\label{viki me roma}
\nor{g_{N_1,N_2}}_{L^2(w)}\geqsim
\left(\sum_{n=N_1}^{N_2-1}\sum_{Q\in{\mathcal G}_0^{2^n}}|\sk{g}{h_Q^0}|^2\nor{h_Q^0}_{L^2(w)}^2\right)^{1/2}.
\end{equation}
Together \eqref{bajro} and \eqref{viki me roma} yield convergence of the series 
$$
\sum_{n=-\infty}^{\infty}\sum_{Q\in{\mathcal G}_0^{2^n}}|\sk{g}{h_Q^0}|^2\nor{h_Q^0}_{L^2(w)}^2,
$$
which thus satisfies the Cauchy condition, in particular
\begin{equation}
\label{besdile}
\lim_{N_2\rightarrow-\infty}\sup_{N_1\in\Z_{N_2}}\sum_{n=N_1}^{N_2-1}\sum_{Q\in{\mathcal G}_0^{2^n}}|\sk{g}{h_Q^0}|^2\nor{h_Q^0}_{L^2(w)}^2=0.
\end{equation}
Finally, Haar system's being a Riesz basis also means that 
$$
\nor{g_{N_1,N_2}}_{L^2(w)}\leqsim
\left(\sum_{n=N_1}^{N_2-1}\sum_{Q\in{\mathcal G}_0^{2^n}}|\sk{g}{h_Q^0}|^2\nor{h_Q^0}_{L^2(w)}^2\right)^{1/2},
$$
which in view of \eqref{besdile} finishes the proof of the lemma in the case of compactly supported $g$.

When $g\in{\mathcal S}$, use that $C_c^\infty$ is a dense subspace of $L^2(w)$, cf. \cite[p. 93 bottom]{M}, and combine with \eqref{bajro}: start with
$g_{N_1,N_2}=(g-f)_{N_1,N_2}+f_{N_1,N_2}$. Then
$$
\lim_{N_2\rightarrow-\infty}\sup_{N_1\in\Z_{N_2}}\nor{g_{N_1,N_2}}_{L^2(w)}\leq \e+0
$$
 for every $\e>0$.
\end{proof}

We continue with the proof of Proposition \ref{fab}. From \eqref{bajro} we get
$$
\aligned
\gamma_{N_2}(\psi,t)
&\leqsim[\tau_{-t}U_\psi w]_{A_2}\nor{\tau_{-t}U_\psi f}_{L^2(\tau_{-t}U_\psi w)}\\
&\leqsim[w]_{A_2}\nor{f}_{L^2(w)}.
\endaligned
$$
Therefore we may use the dominated convergence theorem and conclude from 
\eqref{maljarkica} that
$$
\aligned
\lim_{N_2\rightarrow-\infty}\sup_{N_1\in\Z_{N_2}}&\nor{\cK_{M,N_1}^1f-\cK_{M,N_2}^1f}_{L^2(w)}\\
&\leq
\frac1{4^M}\int_{[0,2^M)^2}
\frac1{2\pi}\int_0^{2\pi}
\lim_{N_2\rightarrow-\infty}
\gamma_{N_2}(\psi,t)d\psi\,dt.
\endaligned
$$
By the previous lemma, this expression is equal to zero.
\end{proof}

It is straightforward to show that if $(\f_N)_N$ is Cauchy sequence in $L^2$ and converges to $\f$ pointwise, then it also converges to $\f$ in $L^2$. 
[True: $\f_N$ Cauchy in $L^2$, therefore $\f_N\rightarrow \psi$ in $L^2$, therefore $(\f_{N_k})_k\rightarrow \psi$ almost everywhere. At the same time we know $(\f_{N_k})_k\rightarrow \f$. Hence $\psi=\f$ almost everywhere.]
This gives the following.

\begin{corollary}
For every $M\in\Z$ and $f\in\cS$ we have
$\lim_{N\rightarrow-\infty}\cK_{M,N}^1f=\cK_{M}^1f$ in $L^2(w)$.
\end{corollary}

The preceeding proof also shows that for any $N_1,N_2\in\Z$, $N_1\geq N_2$,
$$
\nor{\cK_{N_1,N_2}^1f}_{L^2(w)}\leqsim [w]_{A_2}\nor{f}_{L^2(w)}\,,
$$
with the implied constant independent of $M,N,f,w$. Indeed, with any $M>N_1$ we may write 
$$
\cK_{N_1,N_2}^1=\cK_{M,N_2}^1-\cK_{M,N_1+1}^1.
$$
Now \eqref{avaj avaj mo cavo} and Lemma \ref{jazz for cows} imply
$$
\nor{\cK_{N_1,N_2}^1f}_{L^2(w)}
\leq
\Nor{\Big(\sum_{n=N_1}^{N_2}\cP_0^{2^n}\Big)f}_{L^2(w)}\,,
$$
which is dominated by $C[w]_{A_2}\nor{f}_{L^2(w)}$, according to Theorem \ref{Wittwer}.

Hence the previous corollary immediately yields

\begin{corollary}
For every $M\in\Z$ and $f\in\cS$ we have
$$
\nor{\cK_{M}^1f}_{L^2(w)}\leqsim [w]_{A_2}\nor{f}_{L^2(w)}\,.
$$ 
\end{corollary}

Because the ``Cauchy differences'' of the sequence $(\cK_{M,N}^1f)_{N\in\Z_M}$ look exactly as the ``Cauchy differences'' of the sequence $(\cK_{M}^1f)_{M\in\Z_M}$, effectively the same proof works for $\cK^1$:

\begin{corollary}
For every $f\in\cS$ we have
$\lim_{M\rightarrow-\infty}\cK_{M}^1f=\cK^1f$ in $L^2(w)$ and therefore 
$$
\nor{\cK^1f}_{L^2(w)}\leqsim [w]_{A_2}\nor{f}_{L^2(w)}\,.
$$ 
\end{corollary}

And the same can be done for any $r>0$:

\begin{corollary}
For every $r>0$, $f\in\cS$ and $w\in A_2$ we have
$$
\nor{\cK^r f}_{L^2(w)}\leqsim [w]_{A_2}\nor{f}_{L^2(w)}\,.
$$ 
where the implied constants do not depend on $r,f,w$.
\end{corollary}

Finally, this yields the weighted estimate of $S$:
\begin{corollary}
For any $w\in A_2$ we have
\begin{equation*}
\nor{S}_{\cB(L^2(w))}\leqsim [w]_{A_2}.
\end{equation*}
\end{corollary}
\begin{proof}
For every $f\in\cS$, another application of Minkowski's integral inequality gives
$$
\nor{Sf}_{L^2(w)}\leq
\int_1^2\nor{\cK^r f}_{L^2(w)}
\,\frac{dr}r\,.
$$
The previous corollary implies $\nor{Sf}_{L^2(w)}\leqsim [w]_{A_2}\nor{f}_{L^2(w)}$. 
Since $\cS$ is dense in $L^2(w)$ \cite{M}, this allows an extension of $S$ onto the whole of $L^2(w)$ with the same bound.
\end{proof}

Now Proposition \ref{prostranstvima gdje svi} implies 
Theorem \ref{zujovic-crni} and thus Theorem \ref{novljan}, as explained at the beginning of this section.

\subsection{Sharpness}

We still need to show that this estimate is sharp, in the same sense the
Wittwer's estimate (Theorem \ref{Wittwer}) was.
For that purpose we shall need
the following auxiliary result, which can be proved by direct simple
calculations.

\begin{exercise}
\label{nevemkaj}
Let $\alpha\in(-2,2)$ and define $w_\alpha: \R^2\rightarrow [0,\infty)$ by
$
w_\alpha(x)=|x|^\alpha\,.
$ 
Then
$$
[w_\alpha]_{A_2}\sim\frac1{2-|\alpha|}\,.
$$
\end{exercise}

Now we are ready to prove the sharpness.
Calculations of this kind have already been made for other singular integral
operators (e.g. \cite{Buc}) but not for $T$,
so we include this calculation for the sake of completeness.

\begin{proposition}
\label{protiprimer}
Let $\phi:\R\rightarrow (0,\infty)$ grow at infinity
more slowly than a linear function, i.e. 
$$
\lim_{x\rightarrow\infty}\frac{\phi(x)}x=0\,. 
$$ 
Then there is a weight $w\in
A_2$ 
such that
  \begin{equation}
   \label{vecja}
   \nor{T}_{\cB(L^2(w))} > \phi([w]_{A_2})\,.
  \end{equation}
\end{proposition}

\Do 
For $\alpha\in(-2,2)$
define $w(z)=w_\alpha(z)=|z|^\alpha$ as in Exercise \ref{nevemkaj}. 
Denote
$$
E=\mn{re^{i\varphi}}{0<r<1,\, 0<\varphi<\pi/2},
$$ 
$X=-E$,
and $f(z)=|z|^{-\alpha}\chi_{E}(z)$.

We begin by
\begin{equation}
\label{tehnika}
\aligned
\nor{f}^2_{L^2(w)}
&=\iint_{\R^2}|f(x,y)|^2\,w(x,y)\,dx\,dy\\
&=\iint_E(x^2+y^2)^{-\alpha}(x^2+y^2)^{\alpha/2}\,dx\,dy\\
&=\int_0^{\pi/2}\int_0^1r^{-2\alpha+\alpha}\,r\,dr\,d\varphi\\
&=\frac{\pi}{2}\cdot\frac{1}{2-\alpha}\,.
\endaligned
\end{equation}
In particular, $f\in L^2(w)$ and therefore also $Tf\in L^2(w)$. Let us estimate its norm from below:
\begin{equation}
\label{Helmwige}
\nor{Tf}^2_{L^2(w)} =\int_{\R^2}|(Tf)(\xi)|^2\,w(\xi)\,d\xi
 \geq \int_{X}|(Tf)(\xi)|^2\,w(\xi)\,d\xi\,.
\end{equation}
When $z\in X$ we have
$$
(Tf)(z)=-\frac1\pi\int_E\frac{f(u)}{(z-u)^2}\,du\,, 
$$
i.e. the integral converges in the ordinary sense (not only as a principal value integral), because $z\in X$ is on a positive distance from supp\,$f\subset E$. Therefore \eqref{Helmwige} implies
\begin{equation*}
\nor{Tf}^2_{L^2(w)}\geq\iint_X\left|\frac1\pi\iint_E\frac{(s^2+t^2)^{-\alpha/2}}
{[(x-s)+i(y-t)]^2}\,ds\,dt\right|^2\,
(x^2+y^2)^{\alpha/2}\,dx\,dy\,.
\end{equation*}
We use the identity
$$
\frac{1}{[a+ib]^2} =
\frac{a^2-b^2}{[a^2+b^2]^2}
-i\,\frac{2ab}{[a^2+b^2]^2}
$$
to estimate the square of modulus of the inner integral (i.e.
the one over $E$) from below by the square of its imaginary
part, that is, by
$$
\left(\iint_E\frac{2(x-s)(y-t)}{[(x-s)^2+(y-t)^2]^2}(s^2+t^2)^{-\alpha/2}
\,ds\,dt\right)^2\,.
$$
Since $(x,y)\in X$ and $(s,t)\in E$, we have
$$
(x-s)(y-t)\geq xy\,.
$$
The bound for the denominator comes from the triangle
inequality:
$$
[(x-s)^2+(y-t)^2]^2=|(x,y)-(s,t)|^4\leq
\left( |(x,y)| + |(s,t)|\right )^4\,.
$$
Therefore,
$$
\aligned
\left|\iint_E\frac{(s^2+t^2)^{-\alpha/2}}{[(x-s)+i(y-t)]^2}\,ds\,dt\right|^2
& \geq 4x^2y^2\left(\iint_E\frac{(s^2+t^2)^{-\alpha/2}} {(|(x,y)| + |(s,t)|)^4}
\,ds\,dt\right)^2\\
&=\pi^2x^2y^2\left(\int_0^1\frac{r^{-\alpha}}{(|(x,y)| + r)^4}\,r\,dr\right)^2\,.
\endaligned
$$
By taking $u=r/|(x,y)|$ we can continue with
$$
\aligned
\pi^2x^2y^2\bigg(&|(x,y)|^{-\alpha-2}\int_0^{1/|(x,y)|}\frac{u^{1-\alpha}}
{( 1 + u)^4}
\,du\bigg)^2\,\geq\\
&\geq \frac{\pi^2x^2y^2}{(x^2+y^2)^{\alpha+2}}
\left( \int_0^{1}\frac{u^{1-\alpha}}
{( 1 + 1)^4}\,du\right)^2\,
=\frac{\pi^2}{256(2-\alpha)^2}\cdot
\frac{x^2y^2}{(x^2+y^2)^{\alpha+2}}\,.
\endaligned
$$
We proved that
\begin{equation}
\label{leva}
\nor{Tf}^2_{L^2(w)}\geq \frac{1}{256\,(2-\alpha)^2}
\iint_X x^2y^2\,(x^2+y^2)^{\alpha/2-\alpha-2}\,dx\,dy\,=\frac{C_l}{(2-\alpha)^3}\,,
\end{equation}
where
$$
C_l=\frac{1}{256}
\int_0^{\pi/2}\cos^2\varphi\,
\sin^2\varphi\,d\varphi
=\frac{1}{256}\cdot\frac12B\left(\frac32,\frac32\right)=\frac\pi{4096}\,.
$$

Thus we summarize \eqref{leva} and \eqref{tehnika} as
$$
\nor{T}_{\cB(L^2(w))}\geq\frac{\nor{Tf}_{L^2(w)}}{\nor{f}_{L^2(w)}}\,\geqsim\, \frac1{2-\alpha}\,.
$$

In order to finish the proof we need to find $\alpha$ that satisfies the inequality
$$
\frac{1}{2-\alpha} > C\, \phi([w_\alpha]_{A_2})
$$
for a given constant $C$. It suffices to show that
$$
\lim_{\alpha\rightarrow 2^-}(2-\alpha)\, \phi([w_\alpha]_{A_2})=0\,,
$$
which however follows from Exercise \ref{nevemkaj} and the assumption on $\phi$.
\qed

\subsection{Bellman function and Wittwer's theorem}
\label{tamo dolje niz mahalu}
Let $\cD=\cD[0,1]$ be the collection of all dyadic intervals in $[0,1)$. 
As remarked in Section \ref{flash gordon}, the set $\{1\}\cup\mn{h_I}{I\in\cD}$ is a complete orthonormal system in $L^2([0,1])$ \cite[Theorem 5.4.6]{G1}. 
Let $\sigma=(\sigma_I)_{I\in\cD}$ be a sequence in 
$\overline\Delta$. 
Define on $\cS(\R)$ the operator
$$
T_\sigma f=\sum_{I\in\cD}\sigma_I\sk{f}{h_I}h_I\,.
$$
Define also a {\it dyadic $A_2$ weight}: this is a positive function $w\in L^1([0,1])$, such that
$$
[w]_{A_2}^{dyadic}:=\sup_{I\in\cD}\av{w}_I\av{w^{-1}}_I
$$
is finite.

The following theorem was proven by Wittwer \cite{W} who based her approach on the paper by Nazarov, Treil and Volberg \cite{NTV}. We will give another proof, following Volberg \cite[Theorem 2.10]{V2}. A key element will be a concrete Bellman function.

First we need the following simple lemma introducing ``Haar functions'' which are more suitable to the {\it weighted} space $L^2(w)$ (see \cite[Section 2.1]{NTV0}). 
Denote 
$$
d\lambda=w\,dm\,. 
$$

\begin{exercise}
\label{sesta}
For any interval $I$ and any positive $w\in L_{loc}^1(\R)$ there exists a unique function (called $h_I^w$) satisfying the properties:
\begin{enumerate}[(i)]
\item
$h_I^w$ is a linear combination of functions $\chi_{I_-}$ and $\chi_{I_+}$;

\item
\label{asd}
$$
\int_Ih_I^w\,d\lambda=0\,;
$$

\item
\label{wer}
$$
\int_I|h_I^w|^2\,d\lambda=1\,;
$$

\item
$$
h_I^w\Big|_{I_+}>0\,.
$$
\end{enumerate}
Functions $\mn{h_I^w}{I\in\cD}$ form an orthonormal system in $L^2(\lambda)$.
\end{exercise}

Since $h_I^w$ and $\chi_I$ are (linearly independent) linear combinations of functions $\chi_{I_-}$, $\chi_{I_+}$, and so is $h_I$, there exist unique real constants $\alpha_I, \beta_I$ such that
\begin{equation}
\label{Shostakovich - suita u fis-molu}
h_I=\alpha_Ih_I^w+\beta_I\,\frac{\chi_I}{\sqrt{|I|}}\,.
\end{equation}
Observe that $\alpha_I$ cannot be zero, since otherwise $h_I$ would be constant on $I$. 
Certainly, we can extract $\alpha_I, \beta_I$ by using Exercise \ref{sesta}:
\begin{itemize}
\item
compute the $L^2(w)$ norm of \eqref{Shostakovich - suita u fis-molu}:
$$
\av{w}_I=\nor{h_I}_{L^2(w)}^2=\alpha_I^2+\beta_I^2\av{w}_I\,;
$$
\item
divide \eqref{Shostakovich - suita u fis-molu} by $\sqrt{|I|}$ and integrate with respect to $w\,dm$ to get
$$
\frac12\,\Delta_I w=\beta_I\av{w}_I\,.
$$
\end{itemize}
Here $\Delta_I w:=\av{w}_{I_+}-\av{w}_{I_-}\,$. 

Let us summarize the above findings. 

\begin{lemma}
\label{DSCH}
The constants $\alpha_I, \beta_I$ from \eqref{Shostakovich - suita u fis-molu} obey estimates
$$
0<|\alpha_I|\leq\sqrt{\av{w}_I}
\hskip 40pt
\text{and}
\hskip 40pt
|\beta_I|\leq\frac{|\Delta_I w|}{\av{w}_I}\,.
$$
\end{lemma}

\begin{theorem}
\label{sve prolazi}
$$
\sup_{\sigma}\nor{T_\sigma}_{\cB(L^2(w))}\,\leqsim\, [w]_{A_2}^{dyadic}.
$$
\end{theorem}

\begin{proof}
We want to prove $\nor{T_\sigma f}_{L^2(w)}\leqsim\, [w]_{A_2}^{dyadic}\nor{f}_{L^2(w)}$ for $f\in L^2(w)$ and $w\in A_2$. Replacing $f$ by $fw^{-1}$ enables us to rephrase this inequality as 
$$
\nor{T_\sigma(fw^{-1})}_{L^2(w)}\leqsim\, [w]_{A_2}^{dyadic}\nor{f}_{L^2(w^{-1})}
$$ 
for $f\in L^2(w^{-1})$ and $w\in A_2$, which is (take $w^{-1}$ in place of $w$) the same as 
$$
\nor{T_\sigma(fw)}_{L^2(w^{-1})}\leqsim\, [w]_{A_2}^{dyadic}\nor{f}_{L^2(w)} 
\hskip 40pt
\text{ for }
f\in L^2(w), w\in A_2\,.
$$
Since the duality in $L^2(w^{-1})$ means that 
$$
\nor{\f}_{L^2(w^{-1})}=\sup_{\nor{g}_{L^2(w^{-1})}}|\sk{\f}{g}_{L^2(w^{-1})}|\,, 
$$
we see that the theorem is equivalent to proving
$$
|\sk{T_\sigma(fw)}{gw^{-1}}| \leqsim\, [w]_{A_2}^{dyadic}\nor{f}_{L^2(w)} \nor{g}_{L^2(w^{-1})}\,.
$$
We emphasize that in the last inequality we take the usual (unweighted) scalar product, i.e. with respect to the Lebesgue measure.

Since functions $\mn{h_I}{I\in\cD}$ are orthonormal in (the unweighted) $L^2$, we get
$$
\sk{T_\sigma(fw)}{gw^{-1}}=\sum_I\sigma_I\sk{fw}{h_I}\,\overline{\sk{gw^{-1}}{h_I}}\,.
$$
By using the assumption on $\sigma_I$, we see that our task will be completed once we prove
$$
\sum_{I\in\cD}|\sk{fw}{h_I}|\,|\sk{gw^{-1}}{h_I}|\,\leqsim\,[w]_{A_2}^{dyadic}\,\nor{f}_{L^2(w)} \nor{g}_{L^2(w^{-1})}\,.
$$
We use \eqref{Shostakovich - suita u fis-molu} and Lemma \ref{DSCH} to estimate
$$
\sum_{I\in\cD}|\sk{fw}{h_I}|\,|\sk{gw^{-1}}{h_I}|
\leq I+II+III+IV,
$$
where

$$
\aligned
I &= \sum_{I\in\cD}|\sk{fw^{}}{h_I^w}|\sqrt{\av{w}_I}\,|\sk{gw^{-1}}{h_I^{w^{-1}}}|\sqrt{\av{w^{-1}}_I}\\
II &= \sum_{I\in\cD}|\av{fw^{}}_I| \frac{|\Delta_I w|}{\av{w}_I} \sqrt{|I|} \cdot |\sk{gw^{-1}}{h_I^{w^{-1}}}|\sqrt{\av{w^{-1}}_I}\\
III &= \sum_{I\in\cD}|\sk{fw^{}}{h_I^{w}}|\sqrt{\av{w}_I}\cdot |\av{gw^{-1}}_I|\frac{|\Delta_I w^{-1}|}{\av{w^{-1}}_I}\sqrt{|I|} \\
IV &= \sum_{I\in\cD}
|\av{fw^{}}_I| \frac{|\Delta_I w|}{\av{w}_I} \sqrt{|I|} \cdot
|\av{gw^{-1}}_I| \frac{|\Delta_I w^{-1}|}{\av{w^{-1}}_I} \sqrt{|I|}\,.
\endaligned
$$
\end{proof}

\noindent\framebox{Estimate of $I$.} 
We obviously have 
$$
\aligned
I 
& \leq \sqrt{[w]_{A_2}^{dyadic}} \sqrt{\sum_{I\in\cD}|\sk{fw}{h_I^w}|^2}\,\sqrt{\sum_{I\in\cD}|\sk{gw^{-1}}{h_I^{w^{-1}}}|^2}\\
& = \sqrt{[w]_{A_2}^{dyadic}} \sqrt{\sum_{I\in\cD}|\sk{f}{h_I^w}_{L^2(w)}|^2}\,\sqrt{\sum_{I\in\cD}|\sk{g}{h_I^{w^{-1}}}_{L^2(w^{-1})}|^2}\\
& \leq \sqrt{[w]_{A_2}^{dyadic}} \nor{f}_{L^2(w)} \nor{g}_{L^2(w^{-1})}.
\endaligned
$$
For the second inequality we applied Exercise \ref{sesta}. \\

\noindent\framebox{Estimate of $II, III, IV$.} 
Since $III$ is symmetric to $II$, we only consider $II$ and $IV$.

Fix $\alpha\in(0,1/2)$ and introduce
$$
\mu_I:=\av{w}_I^\alpha\av{w^{-1}}_I^\alpha\left(\frac{|\Delta_I w|^2}{\av{w}_I^2}+\frac{|\Delta_I w^{-1}|^2}{\av{w^{-1}}_I^2}\right)|I|\,.
$$
Observe that 
\begin{equation}
\label{maja}
\max\left\{
\frac{|\Delta_I w|}{\av{w}_I}\sqrt{|I|}\ ,\,
 \frac{|\Delta_I w^{-1}|}{\av{w^{-1}}_I}\sqrt{|I|}
 \right\}
\leq
\av{w}_I^{-\alpha/2}\av{w^{-1}}_I^{-\alpha/2}\sqrt{\mu_I}
 \,.
\end{equation}
Now suppose that $p\in(1,2)$. Then H\"older's inequality gives
\begin{equation}
\label{blagdan}
|\av{fw^{}}_I|\leq\av{|f|^pw^{}}_I^{1/p}\av{w^{}}_I^{1-1/p}.
\end{equation}
Introduce measures $d\lambda=w\,dm$ and $d\lambda^*=w^{-1}dm$. Then \eqref{blagdan} can be rewritten as
\begin{equation}
\label{sunce moru}
|\av{fw^{}}_I| 
 \leq \av{w^{}}_I
\left(
\frac1{\lambda(I)}\int_I|f|^p\,d\lambda
\right)^{1/p}
 \leq \av{w^{}}_I 
\inf_{x\in I}\left(M_{\lambda}^d|f|^p(x)
\right)^{1/p}\,,
\end{equation}
where for any positive measure $\mu$ on $[0,1]$ we define its {\it dyadic (weighted) maximal function} by
$$
M_\mu^d\f(x):=\sup_{I\in\cD\atop I\ni x}\frac1{\mu(I)}\int_I|\f|\,d\mu\,.
$$
We summarize \eqref{maja} and \eqref{sunce moru} into
$$
|\av{fw}_I|\frac{|\Delta_I w|}{\av{w}_I}\sqrt{|I|}\leq
\av{w}_I^{1-\alpha/2}\av{w^{-1}}_I^{-\alpha/2}\inf_{I}\left(M_{\lambda}^d|f|^p
\right)^{1/p}\sqrt{\mu_I}\,.
$$
Consequently
$$
\aligned
II &\leq 
\sum_{I\in\cD}
\av{w}_I^{1-\alpha/2}\av{w^{-1}}_I^{1-\alpha/2}
\frac{\inf_{I}\left(M_{\lambda}^d|f|^p
\right)^{1/p}}{\av{w^{-1}}_I^{1/2}}\sqrt{\mu_I}\cdot |\sk{gw^{-1}}{h_I^{w^{-1}}}|
\\
&\leq 
\left([w]_{A_2}^{dyadic}\right)^{1-\alpha/2}
\sum_{I\in\cD}
\frac{\inf_{I}\left(M_{\lambda}^d|f|^p
\right)^{1/p}}{\av{w^{-1}}_I^{1/2}}\sqrt{\mu_I}\cdot |\sk{gw^{-1}}{h_I^{w^{-1}}}|
\\
&\leq 
\left([w]_{A_2}^{dyadic}\right)^{1-\alpha/2}
\sqrt{\sum_{I\in\cD}
\frac{\inf_{I}\left(M_{\lambda}^d|f|^p
\right)^{2/p}}{\av{w^{-1}}_I}\mu_I}
\sqrt{\sum_{I\in\cD} |\sk{gw^{-1}}{h_I^{w^{-1}}}|^2}\\
&\leq 
\left([w]_{A_2}^{dyadic}\right)^{1-\alpha/2}
\nor{g}_{L^2(w^{-1})}
\sqrt{\sum_{I\in\cD}
\frac{\inf_{I}\left(M_{\lambda}^d|f|^p
\right)^{2/p}}{\av{w^{-1}}_I}\mu_I}
\endaligned
$$

and similarly
$$
\aligned
IV &\leq 
\left([w]_{A_2}^{dyadic}\right)^{1-\alpha}
\sum_{I\in\cD}
\inf_{I}\left(M_{\lambda}^d|f|^p
\right)^{1/p}
\inf_{I}\left(M_{\lambda^*}^d|g|^p
\right)^{1/p}\mu_I\\
&\leq 
\left([w]_{A_2}^{dyadic}\right)^{1-\alpha}
\sum_{I\in\cD}
\inf_{I}\left(M_{\lambda}^d|f|^p
\right)^{1/p}
\left(M_{\lambda^*}^d|g|^p
\right)^{1/p}\mu_I
\,.
\endaligned
$$

\begin{lemma}
\label{spavaj}
The sequence $\{\mu_I\}_{I\in\cD}$ is a Carleson sequence with the Carleson constant bounded by $\leqsim\,
\left([w]_{A_2}^{dyadic}\right)^\alpha=:C_\alpha(w)$. This means that, for every $I\in\cD$, 
$$
\sum_{J\in\cD\atop J\subset I}\mu_J\,\leqsim_\alpha\,C_\alpha(w)|I|.
$$
\end{lemma}

This lemma will be proven at the end of the section. We couple it with another result, which follows directly from the dyadic version of the Carleson embedding theorem, see \cite{NT}, \cite{NTV} (for the weighted version and a nicely written proof) and \cite{NTV1} for a proof establishing the optimal constant.

\begin{lemma}
Suppose $(\alpha_J)_{J\in\cD}$ defines a Carleson measure with Carleson constant bounded by $B>0$. If $F,G$ are positive measurable functions on $(0,1)$, then
$$
\aligned
\sum_{J\in\cD}(\inf_J F)\,\alpha_J\,&\leqsim\,B\int_{(0,1)}F\\
\sum_{J\in\cD}\frac{\inf_J G}{\av{w}_J}\,\alpha_J\,&\leqsim\,B\int_{(0,1)}\frac Gw.
\endaligned
$$
\end{lemma}
We apply the above lemmas with 
$F=\left(M_{\lambda}^d|f|^p
\right)^{1/p}
\left(M_{\lambda^*}^d|g|^p
\right)^{1/p}$
and 
$G=\left(M_{\lambda}^d|f|^p
\right)^{2/p}$.
We get
$$
\aligned
II& \leq 
\left([w]_{A_2}^{dyadic}\right)^{1-\alpha/2}
\left([w]_{A_2}^{dyadic}\right)^{\alpha/2}
\nor{g}_{L^2(w^{-1})}
\sqrt{\int_{(0,1)}
\left(M_{\lambda}^d|f|^p
\right)^{2/p}\,d\lambda}\\
& =
[w]_{A_2}^{dyadic}
\nor{g}_{L^2(w^{-1})}
\nor{M_{\lambda}^d|f|^p}_{L^{2/p}(w)}^{1/p}\\
& \leqsim
[w]_{A_2}^{dyadic}
\nor{g}_{L^2(w^{-1})}
\nor{|f|^p}_{L^{2/p}(w)}^{1/p}\\
& =[w]_{A_2}^{dyadic}
\nor{g}_{L^2(w^{-1})}
\nor{f}_{L^{2}(w)}.
\endaligned
$$
Similarly, 
$$
\aligned
IV& \leq 
\left([w]_{A_2}^{dyadic}\right)^{1-\alpha}
\left([w]_{A_2}^{dyadic}\right)^{\alpha}
\int_{(0,1)}
\left(M_{\lambda}^d|f|^p
\right)^{1/p}
\left(M_{\lambda^*}^d|g|^p
\right)^{1/p}
\,dm\\
& \leq
[w]_{A_2}^{dyadic}
\sqrt{\int_{(0,1)}
\left(M_{\lambda}^d|f|^p
\right)^{2/p}\,d\lambda}
\sqrt{\int_{(0,1)}
\left(M_{\lambda^*}^d|g|^p
\right)^{2/p}\,d\lambda^*}
\\
& \leqsim [w]_{A_2}^{dyadic}
\nor{f}_{L^{2}(w)}.
\nor{g}_{L^2(w^{-1})}
\endaligned
$$

At the end we used that $p<2$, i.e. $2/p>1$ and that the weighted maximal function $M_{\lambda}^d$ is bounded on $L^r(\lambda)$ for any $r>1$ and with a constant that does not depend on $w$. See \cite[Exercises 2.1.2, 2.1.12, 1.3.3]{G1}. (As acknowledged in \cite{RTV}, this fact was learned from \cite{CUMP}.) Note that $d\lambda=w\,dm$ is a doubling measure when $w$ belongs to the (usual) $A_2$: by using the {\it ad hoc} notation $2I=(x-2\e,x+2\e)$ for $I=(x-\e,x+\e)$, we see that
$$
\aligned
\lambda(2I)
&=2|I|\av{w}_{2I}
\leq2|I|\frac{[w]_{A_2}}{\av{w^{-1}}_{2I}}
=4|I|^2\frac{[w]_{A_2}}{\int_{2I}{w^{-1}}}
\leq 4|I|^2\frac{[w]_{A_2}}{\int_{I}{w^{-1}}}
=4|I|\frac{[w]_{A_2}}{\av{w^{-1}}_{I}}\\
& \leq4|I|[w]_{A_2}\av{w}_{I}
=4[w]_{A_2}\lambda(I)\,.
\endaligned
$$

We still need to prove Lemma \ref{spavaj}. We do this by means of a {\bf Bellman function}.

\begin{proof}[{\bf Proof of Lemma \ref{spavaj}.}]
Let 
$
\Omega=(0,\infty)\times(0,\infty)\,.
$
Choose $\alpha\in(0,1/2)$ and define the function $b=b_\alpha:\Omega\rightarrow (0,\infty)$ by $b(x,y)=(xy)^\alpha$.
We have, at $(x,y)\in\Omega$, its Hessian matrix $d^2b(x,y)$:
$$
-d^2b(x,y)=\alpha x^\alpha y^\alpha
\left[
\begin{array}{cc}
(1-\alpha)x^{-2} & -\alpha x^{-1}y^{-1}\\
 -\alpha x^{-1}y^{-1} & (1-\alpha)y^{-2} 
 \end{array}
\right]\,.
$$
Obviously
$$
-d^2b(x,y)\geq \alpha(1-2\alpha) (xy)^\alpha \left(\frac{(dx)^2}{x^2}+\frac{(dy)^2}{y^2}\right)\,,
$$
meaning that for any $u,v\in\R$ we have
\begin{equation}
\label{kasno}
\Sk{-d^2b(x,y)\left[\begin{array}{c}
u\\ v
\end{array}\right]
}{\left[\begin{array}{c}
u\\ v
\end{array}\right]
}
\geq \alpha(1-2\alpha) (xy)^\alpha \left(\frac{u^2}{x^2}+\frac{v^2}{y^2}\right)\,.
\end{equation}
Owing to the choice of $\alpha$, the last expression is positive for any $(x,y)\in\Omega$ and any $u,v\in\R$. 
Moreover, if $Q>1$ and $xy\leq Q$, then obviously $b(x,y)\leq Q^\alpha$.

\medskip
Now fix an interval $I$. Denote $\sigma=w^{-1}$. 
Let 
$$
\aligned
a&=(\av{w}_I,\av{\sigma}_I)\\
a_j&=(\av{w}_{I_j},\av{\sigma}_{I_j})\,,
\endaligned
$$
where $j\in\{-,+\}$ and $I_-,I_+$ are the left and the right half of $I$, respectively.

Consider 
$$
\begin{array}{rccl}
c_j: & [0,1] & \longrightarrow & \Omega\\
& t & \longmapsto & (1-t)a+ta_j\,.
\end{array}
$$
and
$$
\begin{array}{rccl}
q_j=b\circ c_j: & [0,1] & \longrightarrow & (0,\infty)\\
& t & \longmapsto & b((1-t)a+ta_j)\,.
\end{array}
$$
Since $q_j$ is of order $C^2$, we have
\begin{equation}
\label{jope diskus}
q_j(0)-q_j(1)=-q_j'(0)-\int_0^1(1-t)q_j''(t)\,dt\,.
\end{equation}
The chain rule gives
\begin{equation}
\label{kisa}
\aligned
q_j'(t) & =\sk{\nabla b(c_j(t))}{a_j-a}\\
q_j''(t)& =\Sk{d^2b(c_j(t))(a_j-a)}{a_j-a}\,.
\endaligned
\end{equation}
Hence from \eqref{kisa} and \eqref{kasno} it follows that $-q_j''(t)\geq 0$, therefore \eqref{jope diskus} gives
\begin{equation}
\label{sarulja}
q_j(0)-q_j(1)\geq-q_j'(0)-\frac12\int_0^{1/2}q_j''(t)\,dt\,.
\end{equation}
Now estimate $-q_j''(t)$ again, but in a more delicate way. Namely, we use again \eqref{kisa} and \eqref{kasno}, but now observe that {\it for $t$ away from 1}, say, if $0\leq t\leq 1/2$ as in the integral above, the first component of $c_j(t)$ cannot be smaller than its own half, i.e. $(1/2)\av{w}_I$:
indeed, either draw a picture or estimate
$$
(1-t)\av{w}_I+t\av{w}_{I_j}\geq(1-t)\av{w}_I\geq\frac12\av{w}_I\,.
$$
On the other hand, since 
\begin{equation}
\label{hepek}
\av{w}_I=\frac{\av{w}_{I_-}+\av{w}_{I_+}}2\,,
\end{equation}
it is always true that the first component of $c_j(t)$ is not larger than $2\av{w}_I$. To summarize,  
for $0\leq t\leq 1/2$, the first component of $c_j(t)$ is comparable with $\av{w}_I$, while the second one is comparable with $\av{\sigma}_I$.

Therefore, when $0\leq t\leq 1/2$ we get from \eqref{kisa} and \eqref{kasno}
$$
\aligned
-q_j''(t)\,& \geqsim\,\alpha(1-2\alpha)\left(\av{w}_{I}\av{\sigma}_{I}\right)^\alpha
\left(
\frac{(\av{w}_{I}-\av{w}_{I_j})^2}{\av{w}_{I}^2}
+\frac{(\av{\sigma}_{I}-\av{\sigma}_{I_j})^2}{\av{\sigma}_{I}^2}
\right)\\
& \geqsim\,\alpha(1-2\alpha)
b(a)
\left(
\frac{|\Delta_Iw|^2}{\av{w}_{I}^2}
+\frac{|\Delta_I\sigma|^2}{\av{\sigma}_{I}^2}
\right)\,.
\endaligned
$$
For the last inequality we used again \eqref{hepek}. Now we get from here and \eqref{sarulja} that 
\begin{equation}
q_j(0)-q_j(1)\geq-q_j'(0)-
C\alpha(1-2\alpha)
b(a)
\left(
\frac{|\Delta_Iw|^2}{\av{w}_{I}^2}
+\frac{|\Delta_I\sigma|^2}{\av{\sigma}_{I}^2}
\right)\,.
\end{equation}
We want to add these inequalities for $j\in\{-,+\}$. Since $(a_+-a)+(a_--a)=0$, we see from \eqref{kisa} that what remains is 
$$
q_+(0)-q_+(1)+q_-(0)-q_-(1)\,\geqsim\,
\alpha(1-2\alpha)
b(a)
\left(
\frac{|\Delta_Iw|^2}{\av{w}_{I}^2}
+\frac{|\Delta_I\sigma|^2}{\av{\sigma}_{I}^2}
\right)\,,
$$
that is,
$$
b(a_{I})-\frac{b(a_{I_+})+b(a_{I_-})}2
\ \geqsim\ 
\alpha(1-2\alpha)
\,b(a_{I})
\left(
\frac{|\Delta_Iw|^2}{\av{w}_{I}^2}
+\frac{|\Delta_I\sigma|^2}{\av{\sigma}_{I}^2}
\right)\,.
$$
Multiply this by $|I|$ and write temporarily $\gamma_J=|J|b(a_{J})$ to get
$$
\gamma_I\geq\gamma_{I_+}+\gamma_{I_-}+C\alpha(1-2\alpha)\mu_I\,.
$$
Finally iterate this inequality. 
After the first iteration we get 
$$
\gamma_I\geq\underbrace{\gamma_{I_{++}}+\gamma_{I_{+-}}+\gamma_{I_{-+}}+\gamma_{I_{-}}}_{\geq 0}+C\alpha(1-2\alpha)[\mu_I+\mu_{I_+}+\mu_{I_-}]\,.
$$
Similarly, after $n$ iterations we obtain
$$
\gamma_I\,\geqsim\, \alpha(1-2\alpha)\sum_{{J\in\cD\atop J\subseteq I}\atop |J|\geq2^{-n}|I|}\mu_J\,.
$$
This is valid for all $n\in\N$, therefore we conclude that
$$
\gamma_I\,\geqsim\, \alpha(1-2\alpha)\sum_{J\in\cD\atop J\subseteq I}\mu_J\,.
$$
Since $\gamma_I=|I|b(a_I)\leq |I|C_\alpha(w)$, where $C_\alpha(w)$ is as in the formulation of the lemma, this finishes the proof.
\end{proof}

\subsection{Unweighted estimates for $T^n$ by means of averaging martingale transforms}

Let us briefly discuss an approach to {\it unweighted} estimates of $T$ or $T^n$ by modifying the averaging technique from Section \ref{chantee}. 

When using Proposition \ref{prostranstvima gdje svi} to estimate $T$ on a nonweighted space (cf. the problem of Iwaniec), one needs to consider a type of estimates of martingale transforms that is different from Wittwer's. For on the unweighted space the problem is very specific - one needs as small {\it numerical} constant as possible in front of the norms of the martingale transforms; the Iwaniec conjecture comprises an {\it equality}, not an inequality, namely $\nor{T}_{\cB(L^p)}=p^*-1$. 
Fortunately, there is a sharp estimate of martingale transforms. This is a famous result of Burkholder \cite{Bu1, Bu2, Bu}, see also \cite[Theorem 8.18]{MS}:
\begin{equation}
\label{balun}
\sup_\sigma\nor{T_{\sigma}}_{B(L^p)}=p^*-1\,. 
\end{equation}
Actually, for martingale transforms on the {\it plane} we need a more general version, regarding Hilbert-space-valued differentially subordinate pairs of martingales. Again it exists and was proven by  Burkholder in \cite{Bu3}. It is very useful due to its generality and sharpness. We present it here for the convenience of the reader.
\begin{lemma}
\label{Bu}
Let $(${\got W}$, {\mathcal F}, P)$ be a probability space, $\mn{{\mathcal
F}_n}{n\in\N}$ a filtration in ${\mathcal F}$ and $H$ a separable Hilbert space. Furthermore, let
$(X_n, {\mathcal F}_n, P)$ and $(Y_n, {\mathcal F}_n, P)$ be $H$-valued
martingales satisfying
\begin{equation}
\label{difsub}
\|Y_0(\omega)\|_H \leqslant \|X_0(\omega)\|_H\ \text{and }
\|Y_n(\omega)-Y_{n-1}(\omega)\|_H \leqslant
\|X_n(\omega)-X_{n-1}(\omega)\|_H
\end{equation}
for all $n\in\N$ and almost every $\omega\in$ {\got W}. Then for
any $p\in (1,\infty)$
$$
\|Y_n\|_p
\leqslant (p^*-1) \|X_n\|_p
\,.
$$
The constant $p^*-1$ is sharp.
\end{lemma}

The property \eqref{difsub} is called {\it differential subordination}.

\begin{theorem}
\label{}
For any $Q\in \cL$ and $*\in\{0,+,-\}$ let $\sigma_Q^*$ be an arbitrary unimodular complex number. Define the operator 
$$
T_\sigma f:=
\sum_{Q\in \cL}
\left[\sigma_Q^0\sk{f}{h_Q^0}h_Q^0
+\sigma_Q^+\sk{f}{h_Q^+}h_Q^+
+\sigma_Q^-\sk{f}{h_Q^-}h_Q^-\right].
$$
Then $\nor{T_\sigma}_{p}\leqslant p^*-1$. This estimate is sharp. 
\end{theorem}

As we saw in the weighted case, when dealing with averaging processes it is desirable to make the objects of averaging have the same norm estimate. In the unweighted case \cite{DPV} this was achieved by a construction which permitted a {\sl uniform} application of the sharp Burkholder's result on the norms of general differentially subordinated martingales. Doing so all the martingale transforms
we considered were, regardless of their symbols, on $L^p$ bounded by $p^*-1$.

It turns out, however, that the averaging method adjusted for powers $T^n$ also calls for 
Haar systems defined on rectangles. 
When using it to obtain as precise estimates as possible of $T$ on $L^p$, we can further extend and sharpen our method by considering different sets of coefficients, and also different Haar systems - ones built on rectangles, parallelograms, triangles and yet more exotic shapes (such as ``L-shaped'' tiles). See \cite{DPV, DPV1}. 
The problem we encounter is that a certain amount of precision is lost during the averaging process. The best constant 
obtained so far in this way is approximately $2^\cdot007$ \cite[p. 4]{DPV1}. That was at the time (2004) when the best known constant was 2. 

Another problem might be the fact that the Burkholder's estimate is sharp when considering {\it all} martingale transforms. Yet we typically average just {\it one} operator, and its translations, dilations and rotations should all obey pretty much the same estimates (as outlined in Section \ref{maide}). 

However, by merging the averaging method and the Burkholder estimate we managed to prove the following result \cite{DPV}: 

\begin{theorem}
\label{dreizehnte_jahrhunderts}
There is $N_0\in\N$ such that for all $n\geqslant N_0$ and $1<p<\infty$, 
\begin{equation}
\label{galobani}
\nor{T^n}_p\leqslant 2.716\,n(p^*-1)\,.
\end{equation}
\end{theorem}
One may compare this estimate to the estimate of $\kappa_n(p)$ in Proposition \ref{junuz} and with Conjecture \ref{gryanyem bratci}. 

It seems to be quite difficult to derive the numerical value of $\nor{T^n}_{L^1\rightarrow L^{1,\infty}}$ in \cite{CRdF}, \cite{H}. Estimate \eqref{galobani} hints that these norms may be bounded by $en$.

\section{Unweighted estimates for $T^n$}
\label{nas na babu promenyal}

The main objective of this section is to prove the results announced in Section \ref{Mitropoulos}.
We will of course separately treat estimates from below and above.

\subsection{Lower estimates}
Take $n\in\N$, $p>2$, $\alpha<1/p$ and define $f=f_{n,\alpha}$ by
$$
f_{n,\alpha}(z)=z^n|z|^{-2\alpha}g_n(|z|)\,.
$$
Here $g_n:[0,\infty)\rightarrow [0,1]$ is a sufficiently smooth, rapidly decaying function satisfying $g_n(x)=1$ for $x\in [0,1]$. 
Its primary r\^ole is to extend $z^n|z|^{-2\alpha}$ from the unit disc to the whole plane in a way such that the outcome lies in $W^{n,p}$. 
For example, one can take
$$
g_n(x)=
\left\{
\begin{array}{lcr}
1 & ; & 0\leqslant x\leqslant 1\\
\exp(-(x-1)^{n}) & ; & x\geqslant 1
\end{array}
\right.
\,.
$$
This is a somewhat natural extension of the example considered by Lehto \cite{L} when giving the lower $L^p$ estimates on the Ahlfors-Beurling operator $T$; see also \cite{BM-S}. Note however that $f$ is not a radial stretch function, since it involves a power of $z$. We use it to prove lower estimates for the $L^p$ norms of powers of $T$. 

\begin{exercise}
\label{Stravinsky: L'Histoire du soldat}
Function $f$ belongs to the Sobolev space $W^{n,p}(\C)$. 
\end{exercise}

This ensures that $T^n(\bar\pd^n f)=\pd^n f$.
Denote by $q$ the conjugate exponent of $p$, i.e. $q=p/(p-1)$. Recall 
that 
the 
Pochhammer symbol $(a)_n=a\,(a+1)\cdot\hdots\cdot(a+n-1)$ 
was introduced in \eqref{psy}. 
Define 
\begin{equation}
\label{blanchard - in time of need}
\kappa_n(p)=\frac{(1/q)_n}{(1/p)_n}\,.
\end{equation}

\begin{proposition}[\cite{D1}]
\label{hm}
We have $\nor{T^n}_p\geqslant\kappa_n(p)$, more precisely,
\begin{equation*}
\label{cacao}
\lim_{\alpha\rightarrow 1/p}
\frac{\nor{\partial^n f_{n,\alpha}}_{L^p(\C)}}{\nor{\bar\partial^n f_{n,\alpha}}_{L^p(\C)}}=\kappa_n(p)
\,.
\end{equation*}
\end{proposition}

\begin{proof}
Write 
$$
\frac{\nor{\partial^n f_{n,\alpha}}_{L^p(\C)}^p}{\nor{\bar\partial^n f_{n,\alpha}}_{L^p(\C)}^p}=
\frac{\displaystyle{\int_\Delta|\pd^nf|^p+\int_{\Delta^c}|\pd^nf|^p}}{\displaystyle{\int_\Delta|\bar\pd^nf|^p+\int_{\Delta^c}|\bar\pd^nf|^p}}
\,.
$$
The integrals over $\Delta$ can be explicitly calculated. Recall that $g_n(|z|)=1$ for $z\in\Delta$. Consequently,
for $z\in\Delta\backslash\{0\}$, 
$$
\partial^nf_{n,\alpha}(z)
=(n-\alpha)(n-1-\alpha)\hdots(1-\alpha)|z|^{-2\alpha}
$$
and
$$
\aligned
\bar\partial^nf_{n,\alpha}(z)
& =(-\alpha)(-\alpha-1)\hdots(-\alpha-n+1)\,\Big(\frac{z}{\bar z}\Big)^n|z|^{-2\alpha} 
\,.
\endaligned
$$
Hence
$$
\aligned
\int_\Delta|\partial^nf_{n,\alpha}(z)|^pdA(z)
& =\prod_{k=1}^{n}(k-\alpha)^p\cdot \int_\Delta |z|^{-2\alpha p}\,dA(z)\\ 
& =\prod_{k=1}^{n}(k-\alpha)^p\cdot2\pi\int_{0}^1r^{1-2\alpha p}\,dr\\ 
& =\frac{\pi}{1-\alpha p}\prod_{k=1}^{n}(k-\alpha)^p
\endaligned
$$
and
$$
\int_\Delta|\bar\partial^nf_{n,\alpha}(z)|^pdA(z)
=\frac{\pi}{1-\alpha p}\prod_{k=0}^{n-1}(k+\alpha)^p\,.
$$
Therefore 
\begin{equation}
\label{hildegard-von-bingen}
\frac{\nor{\partial^n f_{n,\alpha}}_{L^p(\C)}^p}{\nor{\bar\partial^n f_{n,\alpha}}_{L^p(\C)}^p}=
\frac{\displaystyle{\pi\prod_{k=1}^{n}(k-\alpha)^p+(1-\alpha    p)\int_{\Delta^c}|\pd^nf_{n,\alpha}|^p}}{\displaystyle{\pi\prod_{k=0}^{n-1}(k+\alpha)^p+(1-\alpha p)\int_{\Delta^c}|\bar\pd^nf_{n,\alpha}|^p}}
\,.
\end{equation}

As for the integrals over $\Delta^c$, it suffices to notice that they stay bounded as $\alpha\rightarrow 1/p$. Indeed, for $D\in\{\partial,\bar\pd\}$ and $z\in \Delta^c$ we have
$$
|D^nf_{n,\alpha}(z)|
\leqslant 
P(n,\alpha)|z|^M 
e^{-(|z|-1)^{n}}\,.
$$
Here $P$ is a polynomial in two variables, whereas $M>0$ is independent of $\alpha$. Thus when we raise this to the $p-$th power we can estimate the integrands from above uniformly for $\alpha$ close to $1/p$. The majorant is integrable. As a result we can apply the dominated convergence theorem and send $\alpha\rightarrow 1/p$. The integrals over $\Delta^c$ will stay away from infinity. Hence, owing to the factor $1-\alpha p$ in front of them, they will disappear from \eqref{hildegard-von-bingen}. So the result will be 
$$
\prod_{k=0}^{n-1}\left(\frac{k+1-1/p}{k+1/p}\right)^p\,.
$$
By taking the $p-$th root we get $\kappa_n(p)$.
\end{proof}

The asymptotic behaviour of $\kappa_n(p)$ is described by the following proposition.

\begin{proposition}[\cite{DPV, D1}]
\label{junuz}
We have 
$\kappa_n(p)
\sim
n^{1-2/p}(p-1)$. More precisely, for every $n\in\N$ and $p\geqslant 2$,
\begin{equation}
\label{j.l.hooker}
0^\cdot964\leqslant\frac{\kappa_n(p)}{n^{1-2/p}(p-1)}\leqslant 1\,,
\end{equation}
the upper estimate being sharp.
\end{proposition}

\begin{proof}
Denote $\delta=1-2/p$. Then 
\begin{equation*}
\label{ksmet}
\kappa_n(p)=\prod_{k=0}^{n-1}\frac{2k+1+\delta}{2k+1-\delta}\,,
\end{equation*}
therefore \eqref{j.l.hooker} is equivalent to 
\begin{equation}
\label{and one more}
0^\cdot964\leqslant n^{-\delta}\prod_{k=1}^{n-1}\frac{2k+1+\delta}{2k+1-\delta}\leqslant 1\,, 
\hskip 40pt \forall \delta\in[0,1), \forall n\in\N\backslash\{1\}\,.
\end{equation}

First let us prove the lower estimate. 
So our aim is to find the best (i.e. largest) constant $C$ in the inequality 
$$
\log C\leqslant\sum_{k=1}^{n-1}\log\frac{2k+1+\delta}{2k+1-\delta}-\delta\log n\,.
$$
Fix $k\in\N$. The function
$$
g(\delta)=\log\frac{2k+1+\delta}{2k+1-\delta}
$$
is convex on $(0,2k+1)$: indeed,
$$
g'(\delta)=\frac{2(2k+1)}{(2k+1)^2-\delta^2}\,
$$
obviously increases with $\delta$. Therefore $g(\delta)\geqslant g(0)+\delta g'(0)$, 
that is,
$$
g(\delta)\geqslant\frac{2\delta}{2k+1}\,.
$$
Hence $C\geqslant C_1$, where $C_1$ is the best (i.e. largest) constant in the inequality
\begin{equation}
\label{citadela}
\log C_1\leqslant\delta\bigg(\sum_{k=1}^{n-1}\frac{2}{2k+1}-\log n\bigg)\,,\hskip 15pt 0\leqslant\delta<1,\ n\in\N\backslash\{1\} \,.
\end{equation}
Denote the expression in the parentheses above by $\psi(n)$. 
\begin{exercise}
\label{deneg}
Function $\psi$ is decreasing and negative on $\N\backslash\{1\}$.
\end{exercise}

Therefore, since $\psi$ is negative on $\N\backslash\{1\}$, we can in our search for the best $C_1$ in \eqref{citadela} get rid of 
$\delta$ by replacing it with 
the supremum of all of its admissible values 
(i.e. 1). And since $\psi$ is also decreasing, we conclude
\begin{equation}
\label{celtica}
\log C_1=\inf_{n\in\N\backslash\{1\}}\psi(n)=\lim_{n\rightarrow\infty}\psi(n)\,.
\end{equation}

Now introduce 
$$
a_n=\sum_{k=1}^n\frac1k
$$
and 
$b_n=a_n-\log n\,.$
Then 
$$
\aligned
\psi(n)
& =2\bigg(\frac13+\frac15+\frac17+\hdots+\frac1{2n-1}\bigg)-\log n\\
& =2\bigg(a_{2n-1}-1-\frac{a_{n-1}}2\bigg)-\log n\\
& =2b_{2n-1}-b_{n-1}+\log\bigg(4+\frac{1}{n(n-1)}\bigg)-2\,.
\endaligned
$$
It is a
classical fact that $\lim_{n\rightarrow\infty}b_n=\gamma$, where 
$
\gamma\approx 0^\cdot5772
\hdots
$ 
is the Euler--Mascheroni constant, therefore $\lim_{n\rightarrow\infty}\psi(n)=\gamma+2\log 2-2$. 
By \eqref{celtica}, $C_1=4e^{\gamma-2}\approx 0^\cdot964$.

\medskip
Now let us turn to the upper estimate in \eqref{and one more}. 
Because
$$
n^{\delta}=\prod_{k=1}^{n-1}\Big(\frac{k+1}{k}\Big)^{\delta}\,,
$$
it is enough to show that
\begin{equation*}
\label{jugzasut}
\frac{2k+1+\delta}{2k+1-\delta}\leqslant \Big(\frac{k+1}{k}\Big)^{\delta}\,,\hskip 20pt \forall\, k\in\N\,,\ \delta\in[0,1)\,. 
\end{equation*}
When $\delta=0$ the inequality clearly holds, so we may assume that $\delta\in(0,1)$.
Write 
$2k+1=\delta/w$ for some $w\in(0,\delta)$. 
Then the equality to prove becomes
$$
\frac{1+w}{1-w}\leqslant \Big(\frac{\delta+w}{\delta-w}\Big)^\delta
\,,\hskip 20pt \forall\, 0<w<\delta<1\,. 
$$
Verifying this is left as an exercise.
\end{proof}

\begin{exercise}
\label{napoletano}
If $0<w<\delta<1$ then
$$
\frac{1-w}{1+w}\Big(\frac{\delta+w}{\delta-w}\Big)^\delta>1\,.
$$
\end{exercise}

\begin{remark}
Recently D. Kalaj (personal communication, November 2016) calculated the optimal lower constant in \eqref{j.l.hooker}. It turns out to be 
$$
\min_{x\in[0,1/2]}\frac{\Gamma(3/2-x)}{\Gamma(3/2+x)}
$$
which is approximately $0^\cdot985796$.
\end{remark}

\begin{remark}
Proposition \ref{junuz} can be formulated in a slightly more symmetric way, namely
$$
\frac{(a)_n}{(b)_n}\sim
\frac{an^a}{bn^b}
\,,
$$
where $a=1/q$ and $b=1/p$. 
Note that ``without fractions" this is far from true, i.e. clearly $(a)_n\not\sim an^a$.
\end{remark}

As an immediate consequence of Propositions \ref{hm} and \ref{junuz}, we have the following estimate: 

\begin{corollary}
\label{stark}
We have $\nor{T^n}_p\,\geqsim\, n^{1-2/p^*}(p^*-1)$.
\end{corollary}

\subsection{Upper estimates}
Here we will show (Theorem \ref{goin' to louisiana}) that the lower asymptotic estimate from Corollary \ref{stark} has its upper counterpart of the same order, which is valid not only for $T^n$ but for all ${\bf H}_\C^k$, i.e. even when $k$ is odd. 
Our plan is to single out constants in the weak (1,1) and strong (2,2) inequalities for ${\bf H}_\C^k$ and then interpolate.

\medskip
Weak-type (1,1) estimate for singular integrals is one of the cornerstones of the Calder\'on-Zygmund theory.  Many papers were devoted to establishing such results under various hypotheses. See e.g. \cite{Du,G1,MS} for the corresponding references. 
We are interested in 
a) precise and ``simultaneous'' $L^p$ estimates of a concrete sequence of operators and 
b) the conditions that allow such estimates in our case. 
For example, given the oscillation factor $\zeta^{-k}$ in our kernels (see \eqref{chains and things} below), usual gradient (or Lipschitz) conditions would add another factor of $k$ and thus prevent us from getting sharp $L^p$ results in terms of $k$; see e.g. \cite[Proposition 7.4]{MS} or \cite[Theorem 4.3.3]{G1}. A theorem adequate for our purpose, Theorem \ref{I'm going upstairs} below, was proven independently by Christ, Rubio de Francia \cite{CRdF} and Hofmann \cite{H}. (It should be acknowledged that both of these references were brought to the attention of the authors of \cite{DPV} by Michael Christ.) 
The formulation in \cite{H} is explicit about the behaviour of the estimates. We present it here for the reader's convenience. Note that it is valid for kernels far more general (``rough'') than ours, since no ``smoothness'' condition is assumed.

\begin{theorem}[\cite{H}]
\label{I'm going upstairs}
Suppose $\Omega\in L^q(S^1)$ for some $q>1$ and $\int_{S^1}\Omega=0$. For any $\e>0$ define the operator $T_\e$ associated with $\Omega$ and $\e$ by  
\begin{equation}
\label{vrijeme}
T_\e f(z)=\int_{\{|\zeta|>\e\}} f(z-\zeta){\Omega(\zeta/|\zeta|)\over|\zeta|^2}\,dm(\zeta).
\end{equation}
Then, for any $\alpha,\e>0$ and $f\in {\mathcal S}$, 
$$
m\{|T_\e f|>\alpha\}\,\leqsim\,\frac{\nor{\Omega}_q}{\alpha}\,\nor{f}_1.
$$
\end{theorem}

\begin{proof}[{\bf Proof of Theorem \ref{goin' to louisiana}.}]
Fix $k\in\Z\backslash\{0\}$. First consider the case when $1<p\leq2$. 

By \eqref{dihat za ovratnik},
each ${\bf H}_\C^k$ is an isometry on $L^2$. 

Now let us address the weak (1,1) inequality. 
Fix  also $\alpha,\e>0$ and $f\in L^1$.
Let ${\bf H}_{\C,\e}^k$ be the ``$\e$-truncated'' version of ${\bf H}_\C^k$, in the sense of \eqref{vrijeme}.
Denote $M_\e=\{|{\bf H}_{\C,\e}^k f|>\alpha\}$ and $M=\{|{\bf H}_{\C}^k f|>\alpha\}$.

A well-known result on the almost everywhere convergence of homogeneous singular integrals (see, for example, \cite[Corollary 7.11]{MS}) implies that ${\bf H}_{\C}^k f=\lim_{\e\rightarrow 0}{\bf H}_{\C,\e}^k f$ almost everywhere. This leads to
$\chi_M\leqslant \liminf_{\e\rightarrow 0}\chi_{M_\e}$. Consequently, by Fatou's lemma,
$$
m(M)=\int_\C\chi_M\,dm
\leq\int_\C\liminf_{\e\rightarrow 0}\chi_{M_\e}\,dm
\leq\liminf_{\e\rightarrow 0}\int_\C\chi_{M_\e}\,dm
\leqsim\,\frac{\nor{\Omega_k}_{2}}{\alpha}\,\nor{f}_1
.
$$
In the last inequality we applied Proposition \ref{I'm going upstairs}, in particular the fact that the estimates there are independent of $\e$. 
So we proved
\begin{equation}
\label{shake it}
m\left\{|{\bf H}_{\C}^k f|>\alpha\right\}\leqsim \,\frac{|k|}{\alpha}\,\nor{f}_1.
\end{equation}

Now everything is set for interpolation. 
We actually do it as in \cite[Exercise 1.3.2]{G1}.
Choose $p\in (1,2)$ and $r\in (1,p)$. 
Let $N_{1,k}$ be the weak (1,1) constant for ${\bf H}_\C^k$. 
By the Marcinkiewicz interpolation theorem, applied to $1<r<2$,
\begin{equation}
\label{marc}
\nor{{\bf H}_\C^k}_{r}
\leqslant 
2r^{1/r}\bigg({1\over r-1}+{1\over 2-r}\bigg)^{1/r} 
N_{1,k}^{\frac{2-r}r}\,.
\end{equation}
The Riesz-Thorin interpolation theorem we apply to $r<p<2$ and obtain
\begin{equation}
\label{rt}
\nor{{\bf H}_\C^k}_p\leqslant\nor{{\bf H}_\C^k}_{r}
^{\frac r{2-r}\cdot\frac{2-p}p}
\end{equation}
Together \eqref{marc} and \eqref{rt} give
$$
\aligned
\nor{{\bf H}_\C^k}_p
&\leqsim
\bigg[{1\over (r-1)(2-r)}\bigg]^{{1\over p}} 
N_{1,k}^{\frac2p-1}\,.
\endaligned
$$
By choosing $r=(p+1)/2$ and using that $(p-1)^{1/p}\sim p-1$ for $p\in(1,2)$, we finally arrive at 
\begin{equation}
\label{schwach}
\nor{{\bf H}_\C^k}_p\leqsim {N_{1,k}^{\frac2p-1} \over p-1}\,. 
\end{equation}
Recall that, according to \eqref{shake it}, we have $N_{1,k}\leqsim |k|$. 
This proves the theorem for $p\in(1,2)$. 

When $p>2$ use the behaviour of ${\bf H}_\C^k$ in the Fourier domain \cite[Section 4.2]{AIM} and duality: with ${\mathbf m}_k(\zeta)=(|\zeta|/\zeta)^{k}$ we have
$$
\sk{{\bf H}_\C^kf}{g}=\sk{\widehat{{\bf H}_\C^k f}}{\hat g}=\sk{{\mathbf m}_k\hat f}{\hat g}=\sk{\hat f}{\overline {\mathbf m}_k\hat g}=\sk{\hat f}{{\mathbf m}_{-k}\hat g}=\sk{f}{{\bf H}_\C^{-k}g}\,.
$$
This completes the proof of Theorem \ref{goin' to louisiana}.
\end{proof}

\begin{proof}[{\bf Proof of Theorem \ref{izvorska}.}]
In the special case when $k$ is even we get the upper estimate for (integer) powers of the Ahlfors-Beurling operator:
$\nor{T^n}_p\,\leqsim\, n^{1-2/p^*}(p^*-1)$, for any $n\in\Z\backslash\{0\}$ and $p>1$. 
Together with Corollary \ref{stark} this constitutes Theorem \ref{izvorska}.
\end{proof}

\begin{proof}[{\bf Proof of Theorem \ref{v-o-d-a}.}]
When $k=2n$, by combining \eqref{shake it} with Corollary \ref{stark} and \eqref{schwach} we also get Theorem \ref{v-o-d-a}.
\end{proof}

\subsection{Candidates for $\nor{T^n}_p$ and some necessary conditions}
Based on the above findings one may cautiously make the following supposition:

\begin{conjecture}[\cite{DPV,D1}]
\label{gryanyem bratci}
For $p\geqslant 2$ and all $n\in\N$,
\begin{equation}
\label{ahd}
\nor{T^n}_p=\kappa_n(p)=\frac{(1/q)_n}{(1/p)_n}\,.
\end{equation}
\end{conjecture}
As mentioned earlier (Conjecture \ref{peter hammers}), this question is open even for $n=1$, in which case it is well-known since 1982 as the {\it Iwaniec conjecture}.

Since $\nor{T^{m+n}}_p\leqslant\nor{T^{m}}_p\nor{T^{n}}_p$ the same must be true for $\kappa_n(p)$, if \eqref{ahd} is to hold. 
Indeed, such is the case, which slightly reinforces our belief in \eqref{ahd}:

\begin{exercise}
\label{ojmostaru}
For $m,n\in\N$ and $p\geqslant 2$ we have
$$
\kappa_{m+n}(p)\leqslant\kappa_{m}(p)\kappa_{n}(p)\,.
$$
with the equality only when $p=2$.
\end{exercise}

We have seen in Proposition \ref{junuz} that the function $(n,p)\mapsto\kappa_n(p)$ tightly fits to the function $(n,p)\mapsto n^{1-2/p}(p-1)$, which makes the latter worth investigating further.

\begin{exercise}
\label{gGFGEEEEE-DCDECGECGECAG}
The analogue of Exercise \ref{ojmostaru} also holds for the function 
$$
\gamma_p(s)=s^{1-2/p}(p-1).
$$  
Actually, we have $\gamma_p(s+t)\leqslant\gamma_p(s)\gamma_p(t)$ for $p\geqslant 2$ and any $s,t\geqslant 1$. 
\end{exercise}

Let us consider further necessary conditions, arising from varying not $n$ as above, but $p$.

\begin{exercise}
\label{margin call}
Suppose $(\Omega,\mu)$ is a measurable space and $\Lambda$ a bounded operator on $L^r(\Omega,\mu)$ for any $r\geq2$. Assume also that $\nor{\Lambda}_{2\rightarrow2}\leq1$. 
Then the function 
$$
p\mapsto
\nor{\Lambda}_{p\rightarrow p}^{p/(p-2)}
$$
is increasing on $(2,\infty)$.
\end{exercise}

\begin{exercise}
\label{1984}
For $u>0$ define
$$
\f(u):=\left\{
\begin{array}{ccl}
{\displaystyle\frac{u+1}{u-1}\log u} & ; & u\ne1\\
2 & ; & u=1.
\end{array}
\right.
$$
Verify that:
\begin{itemize}
\item
$\f(1/u)=\f(u)$;
\item
$\f\in C^1(0,\infty)$;
\item
$\f\big|_{[1,\infty)}$ is strictly increasing.
\end{itemize}
Compare with Exercise \ref{u meni je sve sto imam}. 

It follows immediately that the function
$$
p\mapsto
(p-1)^{p/(p-2)}
$$
is increasing on $(2,\infty)$.
\end{exercise}

In view of Exercise \ref{margin call}, if \eqref{ahd} is to hold, then the following must also hold:

\begin{exercise}
Suppose $n\in\N$, $p>2$, $1/p+1/q=1$. Then the function
$$
p\mapsto
\left[
\frac{(1/q)_n}{(1/p)_n}\right]^{p/(p-2)}
$$
is increasing on $(2,\infty)$.
\end{exercise}

Finally observe that by Exercise \ref{gGFGEEEEE-DCDECGECGECAG} and the last statement of Exercise \ref{1984},  one may argue that $\gamma_p(n)$ could also be a candidate for $\nor{T^n}_p$ when $p\geq2$.  
Note, however, that 
it is not known whether $\nor{T^n}_p\geq\gamma_p(n)$, in contrast with Proposition \ref{hm}.

\section{Spectral theory for $T$}
\label{george szell}

The estimates of $T$ on $L^p$ and $L^p(w)$, and of $T^n$ on $L^p$, were all motivated by various aspects of the quasiconformal theory, see \cite{I,AIS,IM} or Sections \ref{Ravel piano G}, \ref{liberation sarajevo} and \ref{saso mange}, respectively. As indicated in Theorem \ref{jednamladost}, the averaging method developed in \cite{DV1} and \cite{DPV} can be equally successfully applied for estimating $T^n$ on $L^p(w)$ for arbitrary $n\in\Z$, $p>1$, $w\in A_p$.
Thus to complete the picture 
one might want
to pose a question in the ``reverse" direction, i.e., to ask about
possible complex-analytic implications of an estimate such as \eqref{prituri_se}.

\medskip
Since powers of an operator feature prominently in the spectral radius formula, we are led to inquire whether Theorem \ref{jednamladost} offers any meaningful information about the spectrum of $T$ on $L^p(w)$. 
Let us first recall a few basic notions.

Let $X$ be a Banach space, $\Lambda\in\cB(X)$, and $I$ the identity on $X$. We denote the {\it spectrum}, {\it approximate point spectrum} and {\it continuous spectrum}, respectively, by
$$
\aligned
\sigma(\Lambda) & =\{\lambda\in\C\ ; \  \Lambda-\lambda I \text{ not invertible in } X\}\\
\sigma_{ap}(\Lambda)& =\{\lambda\in\C\ ; \ \exists\ (f_n)_{n\in\N}\subset X\, : \nor{f_n}_{X}=1\text{ and } (\Lambda-\lambda I)f_n\rightarrow 0 \text{ in }X\}\\
\sigma_c(\Lambda)&=\{\lambda\in\sigma(\Lambda)\ ; \  \Lambda-\lambda I \text{ has trivial kernel and dense range in }X\}.
\endaligned
$$
It is well known that $\sigma_c(\Lambda)\subset\sigma_{ap}(\Lambda)\subset\sigma(\Lambda)$ 
\cite[Theorem 6.17]{AA}.
Furthermore, let $\rho(T)=\C\backslash\sigma(T)$ be the resolvent set and $r(T)=\max\mn{|\lambda|}{\lambda\in\sigma(T)}$ the spectral radius of $T$. Recall that $r(T)=\lim_{n\rightarrow\infty}\nor{T^n}^{1/n}$.

Fix $p>1$ and $w\in A_p$.
We have the following result.

\begin{corollary}
\label{nini_diplomirala}
For any  $p>1$ and $w\in A_p$, we have $\sigma(T)
=\partial\Delta$ on $L^p(w)$. 
\end{corollary}

In the unweighted case
this result
already appeared
in \cite[Proposition 2]
{AIS} and \cite[Theorem 14.1.1]{AIM}.
In order to prove it
the authors first showed (by a method different from the one in this paper) that
(a) $\sigma(T)\subset\pd\Delta$,
and then argued that
(b)
\label{soggybottom}
any $\lambda\in\pd\Delta$ is an eigenvalue for $T$.
Yet
the claim (b)
cannot
hold, which can for $p=2$ be seen either by
taking the Fourier transform or by recalling that
any normal bounded operator on a separable Hilbert space has an at most countable point spectrum.
A closer inspection of the proofs in \cite{AIS} and \cite{AIM} reveals that their candidate for an eigenvector of $T-\zeta I$
is actually
a zero function if $|\zeta|=1$.
It 
is true,
however, that 
$\sigma(T)$
consists of {\sl approximate} eigenvalues, not only in $L^p$ but in all $L^p(w)$. This is because of a simple fact that for general $\Lambda$ as above, every boundary point of $\sigma(\Lambda)$ is an approximate eigenvalue of $\Lambda$ 
\cite[Chapter VII, Proposition 6.7]{C}.

It should be emphasized 
that the proofs from \cite{AIS} and \cite{AIM} have the advantage of bringing up an explicit formula for the inverse of $T-\zeta I$; see \cite[identity (14.9)]{AIM}.
Our proof does not entail that.

These remarks motivate us to classify $\sigma(T)$ in the unweighted case, following the usual way of decomposing spectra of bounded operators on Banach spaces.

\begin{theorem}
\label{skalinada}
On any $L^p(\R^2)$, $p>1$, we have $\sigma(T)=\sigma_c(T)$.
\end{theorem}
Recall that $\mathcal S$ denotes the Schwartz class on $\C$. We are still working in $L^p(w)$ with $p>1$ and $w\in A_p$ fixed.

\begin{proof}[{\bf Proof of Corollary \ref{nini_diplomirala}.}]
We get straight from Theorem \ref{jednamladost}
that $\max \{r(T),r(T^{-1})\} \leqslant 1$.
Hence $\sigma(T)\cup\sigma(T^{-1})\subset \overline\Delta$.
By the spectral mapping theorem,
$\sigma(T^{-1})=[\sigma(T)]^{-1}$. Thus we conclude that $\sigma(T)\subset \pd\Delta$.

Now let us verify the inclusion in the opposite direction. 
The author is indebted to Michael Cowling for showing him how to do that.

Take $\lambda\in\pd\Delta$ and write $\lambda=e^{-2i\vartheta}$. Our intention is to show that $\lambda$ is an approximate eigenvalue for $T$ on $L^p(w)$.
For 
$n\in\Z$ let $S_n$ be the operator of multiplication by the function $z\mapsto e^{2n\pi i\Re(e^{-i\vartheta} z)}$, i.e. 
$$
(S_ng)(z)=e^{2n\pi i\Re(e^{-i\vartheta} z)}g(z)
$$ 
for any function $g:\C\rightarrow\C$ and any $z\in\C$.
Take
$ f \in\mathcal S$ such that $\nor{ f }_{L^p(w)}=1$ and $\supp\widehat f \subset\Delta$.
We claim that
$
\lim_{n\rightarrow\infty}(T-\lambda I)(S_n f )=0 
\text{ in } L^p(w)\,.
$
This is equivalent to
\begin{equation}
\label{eq:certainmistique}
\lim_{n\rightarrow\infty}(T_n-\lambda I) f =0\hskip 10pt \text{ in } L^p(w)\,,
\end{equation}
where $T_n=S_n^{-1}TS_n$. 

It is known that \cite[p. 94]{M} that $w\in A_p$ implies $\int_{\R^2}w(x)/(1+|x|)^k\,dx<\infty$ for some sufficiently large $k>0$. Consequently, \eqref{eq:certainmistique} would follow if we knew that 
$$
|(T_n-\lambda I) f|^p(x)(1+|x|)^{k}\rightarrow 0
\hskip 30pt \text{ in } L^\infty(\C) \text{ as } n\rightarrow\infty,
$$
or 
\begin{equation}
\label{kdo so oni}
|(T_n-\lambda I) f|(x)(1+|x|)^{k/p}\rightarrow 0
\hskip 30pt \text{ in } L^\infty(\C) \text{ as } n\rightarrow\infty.
\end{equation}

By using \eqref{infisa} and a well-known property of the Fourier transform, e.g. \cite[Proposition 2.2.11.(7)]{G1}, one calculates
\begin{equation}
\label{incy wincy}
\widehat{T_n f }(\xi)=\Phi(\xi/n+e^{i\theta})\widehat f (\xi)\,,
\end{equation}
where
$$
\Phi(\zeta)=\frac{\bar\zeta}\zeta
\hskip 40pt 
\text{ for }\zeta\ne0.
$$
The assumption on $f$ implies that $\xi/n+e^{i\theta}$ stays away from zero for any $n\in\Z\backslash\{0\}$ and $\xi\in\supp\widehat f$. Therefore \eqref{incy wincy} shows that $T_nf\in{\mathcal S}$.
Now \eqref{kdo so oni} can be rephrased as a special case of 
$T_n f\rightarrow \lambda f$ in ${\mathcal S}$. 
This is in turn equivalent to $\widehat{T_n f }\rightarrow \lambda \widehat f $ in ${\mathcal S}$. 
Note that $\lambda=\Phi(e^{i\theta})$. 
So we are proving, for $g=\hat f\in C_c^\infty(\Delta)$ and $\zeta=e^{i\theta}\in\pd\Delta$,
$$
\lim_{n\rightarrow\infty}\Phi(\cdot/n+\zeta)g=\Phi(\zeta)g
\hskip 40pt 
\text{ in }\mathcal S.
$$
Owing to the assumption on $\supp\widehat f $ it is now enough to verify that $\Phi(\cdot/n+\zeta)$ tends to $\Phi(\zeta)$ in $C^\infty(\overline\Delta)$, which emerges after a computation (for ``everything happens'' in a small neighbourhood of $\zeta$, i.e. in a ball centered at $\zeta$ and of radius $1/n$).
\end{proof}

\begin{remark}
With $U_\zeta \f(z)=\f(\zeta z)$ for $\zeta\in\pd\Delta$ and $z\in\C$ we have
\begin{equation}
\label{nijemo}
T-\zeta^2I=\zeta^2U_{\zeta} (T- I)U_{\zeta}^{-1}\,.
\end{equation}
When $w$ is rotation invariant
this identity quickly
implies that $\sigma(T)$ is also rotation-invariant, hence $\sigma(T)\subset\pd\Delta$ immediately gives $\sigma(T)=\pd\Delta$.
Kari Astala pointed out (personal communication, 2009) that a similar argument is valid in many other function spaces, e.g., Besov, H\"older, Triebel-Lizorkin.
\end{remark}

Now we turn to proving Theorem \ref{skalinada}.
First we need a few simple auxiliary lemmas.

\begin{exercise}
\label{u meni je sve sto imam}
Suppose $\f\in C^\infty(\R)$ is such that $\f(0)=0$. Define
$$
\psi(x)=\left\{
\begin{array}{ccl}
\displaystyle{\frac{\f(x)}{x}} & ; & x\ne 0\\
\f'(0) & ; & x= 0\,.
\end{array}
\right.
$$
Then $\psi\in C^\infty(\R)$ and for all $n\in\N$,
\begin{equation}
\label{kao da je dio sna}
\psi^{(n)}(0)=\frac{\f^{(n+1)}(0)}{n+1}\,.
\end{equation}
\end{exercise}

Before stating the next result we need to invoke a few standard notions. 

If $G\in\cS'(\R^2)$ then by $\xi_2G$ we denote the product of the function (projection) $(x_1,x_2)\mapsto x_2$ and the tempered distribution $G$. For necessary definitions see, for example, \cite[Definition 2.3.15]{G1} or \cite[Chapter III]{Hor}.
Basic facts about convolution of a function and a tempered distribution can be found, for example, in \cite[Section 2.3]{G1} and \cite[Chapter IV]{Hor}. 
 Furthermore, 
 $\delta$ is the Dirac delta distribution (evaluation at zero), while by $1$ we mean the distribution $\phi\mapsto\int_\R\phi$. It is straightforward 
  that $1=\wh\delta$ and $\delta=\wh 1$.

\begin{exercise}
\label{scarlatti}
Suppose $u\in\cS'(\R^2)$ and that $u(\phi)=0$ for any $\phi\in\cS(\R)\otimes\cS(\R)$. Then $u=0$.
\end{exercise}

\begin{exercise}
\label{Drummer's salute}
Suppose that $G\in\cS'(\R^2)$ and that $\xi_2G=0$. Then there exists a tempered distribution $w\in\cS'(\R)$ such that $G=w\otimes\delta$. Consequently, $\wh G=\wh w\otimes 1$.
\end{exercise}

\begin{exercise}
\label{khusro}
Suppose $g$ is a H\"older continuous function on $\R^2\equiv\C$ such that, in the sense of tempered distributions, $g=W\otimes 1$ for some $W\in\cS'(\R)$. Then $g(z)=g(\Re z)$ for all $z\in\C$.
\end{exercise}

\begin{proof}[{\bf Proof of Theorem \ref{skalinada}.}]
It suffices to consider the case $p\geqslant 2$.
Indeed, since the same proofs as for $T$ also work for $T^*$ (cf. Exercise \ref{lullaby}) and since the spaces $L^p$ and $L^q$ are mutually dual, the case $1<p<2$ follows from 
Exercise \ref{berg}. 
Furthermore, by \eqref{nijemo}
it is enough to see that $1\in\sigma_c(T)$.

Thus assume $p\geqslant 2$ and write $T_1=T-I$. First we will show that $\Ker T_1=\{0\}$ on $L^p$. If $p=2$ this follows directly from \eqref{infisa}.
Now suppose $p>2$.
For every $h\in L^p$ we 
define 
$Ph$ as in \eqref{posvjecenje proljeca}. 
By Exercises \ref{ruptura partialis} and \ref{preobrazhenskij}, if $p>2$, then $g=Ph$ is well defined on $L^p$ and satisfies
\begin{equation}
\label{tomato}
h=\pd_{\bar z}g\hskip 15pt \text{ and }\hskip 15pt Th=\pd_{z}g
\end{equation}
in the distributional sense, with test functions from $C_c^1$. 
Here $\pd_{\bar z}=(\pd_{x}+i\pd_{y})/2$ and $\pd_{z}=(\pd_{x}-i\pd_{y})/2$, as usual.
If we additionally assume that $h\in \Ker T_1$, 
then $\pd_yg=0$ in the same distributional sense.
We want to show that $\pd_yg$ exists and is equal to zero in the 
usual
sense, as well. 
Note the similarity with the Weyl lemma.

Since $g$ is H\"older continuous with the exponent $1-2/p$,
it defines a tempered distribution \cite[Example 2.3.5.6]{G1}, denoted by the same letter.
The space $C_c^\infty$ is dense in $\mathcal S$, equipped with the usual topology induced by the Schwartz seminorms \cite[Lemma 7.1.8]{Hor};
therefore $\pd_yg=0$ is also valid for test functions from ${\mathcal S}$.
By taking the Fourier transform we get $\xi_2\,\hat g=0$, see \cite[Proposition 2.3.22.(8)]{G1}, where $(\xi_1,\xi_2)$ are coordinates in the Fourier domain and $\xi_2$ is also a projection onto the second coordinate.
From 
Exercise \ref{Drummer's salute} we conclude
that $\hat g=w\otimes \delta$, where $w\in{\mathcal S}'(\R)$ and $\delta$ is the Dirac distribution, i.e., the evaluation at $0$. Repeated application of the Fourier transform yields $g=\widetilde w\otimes 1$ for some $\widetilde w\in{\mathcal S}'(\R)$. Exercise \ref{khusro} implies that $g$, now again understood as a function, only depends on the real part, i.e., 
$g(z)=g(\Real z)$ for all $z\in\C$, as desired.
In particular, $Qh\equiv 0$, where
$(Qh)(z)=g(z)-g(z+i)\,.$
By means of \eqref{tomato} we 
obtain
$[\pd_{\bar z}(Qh)](z) =h(z)-h(z+i)$
in the distributional sense. 
Thus we proved that $h(z)=h(z+i)$ p.p. $z\in\C$, which for $h\in L^p$ is only possible if $h\equiv 0$ p.p. $\C$. This confirms that $\Ker T_1$ is trivial.

\medskip

We are left with proving that $\Img T_1$ is dense in $L^p$.
Take $f\in L^p$ and $\e>0$.
Our goal is to find $g\in\mathcal S$ such that $\nor{f-T_1 g}_p<\e$.
We can assume that $f\in\mathcal S$.
Take any $\f\in\mathcal S$ such that $\nor\f _q=1$.
By using \eqref{infisa} and the Plancherel identity compute, for a generic $g$ which is to be determined later,
$
\sk{f-T_1 g}{\f}
=\sk{\hat f-H\hat{g}}{\hat \f}
\,,
$
where
$H(\xi)=e^{-2i\arg\xi}-1$
for
$\xi\in\C\backslash\{0\}$.
Since $1<q\leqslant 2$, the
Hausdorff-Young
and the Cauchy-Schwarz inequalities imply that
$
|\sk{f-T_1 g}{\f}|
\leqslant\nor{\hat f-H\hat{g}}_q\,.
$
Therefore, given $F\in{\mathcal S}$, our problem reduces to finding $G\in{\mathcal S}$ such that
$
\nor{F-HG}_q<\e\,.
$
We may assume that $F\in C_c^\infty(\C)$.
Choose $R>0$ so that $\supp F\subset \mn{z\in\C}{|z|\leqslant R}=:K_R$.
For small $\delta>0$ define
$
D_\delta:=\mn{(x,y)\in K_R}{|y|\geqslant\delta}
$
and $E_\delta:=K_R\backslash D_\delta$.
There exists an Urysohn  function $u_\delta\in C_c^\infty(\C)$ such that
\begin{enumerate}[(i)]
\item
$u_\delta\in[0,1]$ everywhere on $\C$,
\item
$u_\delta\equiv 1$ on $D_\delta$,
\item
$u_\delta\equiv 0$ on $E_{\delta/2}$.
\end{enumerate}
By letting $G_\delta(\xi):=0$ for $\xi\in\R$ and $G_\delta:=F\cdot u_\delta/H$ otherwise, we see that
$G_\delta\equiv 0$ on $E_{\delta/2}$ and so $G_\delta\in C_c^\infty(\C)$. 
Now clearly $\nor{F-HG_\delta}_q\rightarrow 0$ as $\delta\rightarrow 0$.
\end{proof}

\section{Open questions}

Let us conclude this presentation by summarizing a few problems encountered in our study of the Ahlfors-Beurling operator. Most likely, in terms of difficulty these questions vary greatly, the first three being extremely difficult.

\begin{question}
Is it true that $\nor{T}_p\leq p-1$ for $p\geq 2$? This is the well-known Iwaniec conjecture, see Section \ref{Ravel piano G}. 
It is a special case of the following question:
\end{question}

\begin{question}
\label{indian tonic water}
Is it true that $\nor{T^n}_p\leq n^{1-2/p}(p-1)$ for $p\geq 2$ and $n\in\N$?
In view of Proposition \ref{j.l.hooker}, an affirmative answer to this question would follow if we could confirm the following one:
\end{question}

\begin{question}
\label{charley patton}
Is it true that 
\begin{equation}
\label{louvre}
\nor{T^n}_p\leq \frac{(1/q)_n}{(1/p)_n}
\end{equation}
for $p\geq 2$, where $1/p+1/q=1$ and $(a)_n$ is defined in \eqref{psy}?  See Conjecture \ref{gryanyem bratci}.
\end{question}

\begin{question}
Is it true that $H_p(2k+1)\,\geqsim\, |k|^{1-2/p}(p-1)$ for all $k\in\Z$ and $p\geq 2$? Recall that $H_p(m)$ was defined, following \cite{IM}, in Section \ref{saso mange}, which is also where the problem's background was presented.
\end{question}

\begin{question}
Is it true that the optimal weighted estimates of powers $T^n$ on $L^p(w)$ are linear in $n$? That is, does for 
every $p>1$ exist $C(p)>0$ such that for every $n\in\Z\backslash\{0\}$ and $w\in A_p$ we have
$$
\nor{T^n}_{B(L^p(w))}\leqslant C(p)\, |n|\,[w]_p^{p^*\!/p}\,?
$$
See Section \ref{american pie} for a related discussion.
\end{question}

\begin{question}
Is it true that the spectrum of $T$ on $L^p(w)$ is continuous for all $p>1$ and $w\in A_p$? 
This is known to be true in the unweighted case ($w\equiv 1$), see Theorem \ref{skalinada}.
\end{question}

\subsection{Hypergeometric functions and de Branges' theorem}
Here we present a point of view of the Iwaniec conjecture which was first published in \cite{D1}.

\bigskip
So far all the attempts to prove \eqref{roll over, roll over}
relied heavily on the Burkholder's sharp $L^p$ estimate for martingale transforms \eqref{balun}. Although this theorem proved to be extremely important for obtaining very good estimates for $\nor{T}_p$, see \cite{BJ} and the references listed there, up to now there has not been a definite proof of the Iwaniec conjecture, as far as we are aware. 
 One of the obstacles in applying Burkholder's theorem to that purpose is that one is indirectly prompted to work with 
  the real and the imaginary part of $T$
 (see \cite{DV2}, for example).
 Nazarov and Volberg \cite{NV} proved that $\text{Re\,}T$ and $\text{Im\,}T$ have $L^p$ norms bounded from above 
 by $p^*-1$. 
Recently Geiss, Montgomery-Smith and Saksman \cite{GM-SS} and then Boros, Sz\'ekelyhidi Jr. and Volberg \cite{BSV} gave the same estimate from below, thus proving $\nor{\text{Re\,}T}_p=\nor{\text{Im\,}T}_p=p^*-1$. This additionally hints that one should 
try not to treat 
the real and the imaginary part of $T$ separately when determining $\nor{T}_p$.
 
Here we bring up a perspective of the $p-1$ problem different from the ones pursued so far in that it is entirely ``complex" and does not involve Burkholder's theorem. Instead, we construct some kind of a generating function whose Taylor coefficients involve $T^n$ and then discuss possibilities of applying 
other sharp estimates, such as 
de Branges' theorem about {\it schlicht} functions. The latter has, to our knowledge, so far not been connected to the $p-1$ problem. \\

\framebox{\ensuremath{\boldsymbol{Caveat\ lector}}:} 
This section should be viewed merely as a short discussion. The approach presented here has so far brought no definitive results regarding the operator $T$ and it may likely be that it actually cannot. Still, we decided to include it, in case it might raise some of the readers' attention and motivate  
thoughts in this direction that would reach beyond the level of simple remarks that can be found here.\\

Remember that $\kappa_n(p)$ was defined in \eqref{blanchard - in time of need}. 
The fact that it is a quotient of Pochhammer symbols makes us recall the standard hypergeometric function 
$$
F(a,b,c;z)=1+\sum_{n=1}^\infty\frac{(a)_n(b)_n}{(c)_n\,n!}\,z^n\,.
$$
Thus \eqref{ahd} is the same as to say that the $L^p$ norms of $T^n$ are the coefficients of the hypergeometric function $F(1,q^{-1},p^{-1};z)$. This suggests considering 
a sort of a
power series 
involving quotients of $T^n$ and $\kappa_n(p)$, and then trying to estimate its coefficients.
In such a manner we might think about approaching 
\eqref{louvre}, 
which would in view of Proposition \ref{hm} immediately establish the $L^p$ norms of $T^n$, i.e. \eqref{ahd}. 
To summarize, the idea is to take all powers at the same time and to think of them as 
coefficients of some generating function which are then to be estimated by means of some sharp theorem. 

All said invokes the following celebrated theorem, which was already mentioned in Section \ref{leeves}.
\begin{theorem}[de Branges \cite{dB}]
\label{bieb}
Suppose the function
$$
g(z)=\sum_{n=1}^\infty b_nz^n
$$
is holomorphic and injective in $\Delta$. Then $|b_n|\leqslant n|b_1|$ and this estimate is sharp.
\end{theorem}

The sharpness part is equally important for us, because we can hardly hope to obtain another sharp result, i.e. \eqref{louvre}, by applying an unsharp one. The sharpness in the theorem above is obtained (only) by considering the Koebe function or its rotations, explicitly, 
$$
g_{\beta}(z)=\frac{\beta^2 z}{(\beta- z)^2}\,,
$$
where $\beta\in\pd\Delta$.
Functions $g_\beta$ somewhat resemble the integral kernel of $T$, which is another -- though very circumstantial -- hint
that there might be some connection.

Thus we wonder whether Theorem \ref{bieb} (or some other sharp estimate of Taylor coefficients) can be applied in this context. 
Take a real sequence $a=(a_n)_{n\in\N}$, functions $f,g$ from the Schwartz class ${\mathcal S}$ such that $\nor{f}_p=\nor{g}_q=1$ and define function $\Psi=\Psi_{a,p,f,g}$ as
\begin{equation*}
\label{sczerbiecz}
\Psi(z)=z+\sum_{n=2}^\infty e^{ia_n}\frac{\sk{T^{n-1}f}{g}}{\kappa_{n-1}(p)}\,nz^n.
\end{equation*}
We wonder whether it is injective for some sequence $a$.
Alternatively, by targeting Question \ref{indian tonic water} instead of Question \ref{charley patton} we may define
$$
\widetilde\Psi(z)=(p-1)z+\sum_{n=2}^\infty e^{ia_n}\sk{T^{n-1}f}{g}\,n^{2/p}z^n\,.
$$
Cauchy-Hadamard formula and Theorem \ref{izvorska} imply that every such $\Psi$ (or $\widetilde\Psi$) is holomorphic in $\Delta$.

Theorem \ref{bieb} immediately yields:
\begin{proposition}
\label{trskagora}
Let $f,g\in\mathcal S$ satisfy $\nor{f}_p=\nor{g}_q=1$. If there exists a real sequence $a=(a_n)_{n\in\N}$ such that the function $\Psi=\Psi_{a,f,g}$ is injective on $\Delta$, then $|\sk{T^nf}{g}|\leqslant \kappa_n(p)$ for all $n\in\N$. 
\end{proposition}
Surely, the problem with verifying \eqref{ahd} through this approach is 
that it is difficult
to check whether the function $\Psi$ is injective on $\Delta$.

\begin{question}
By de Branges' theorem, we know that if the function
$$
z\longmapsto z+\sum_{n=2}^\infty b_n\,z^n
$$
is injective on $\Delta$, then $|b_n|\leq n$ for all $n\in\N$. We also know that the converse is not necessarily true, even if all $b_n$ are strictly positive (Proposition \ref{segovia - hotel de condes de castilla}). 

Suppose $0<|b_n|\leq n$ for all $n\in\N$. Does there exist a real sequence $a=(a_n)_{n\in\N}$ such that the function
$$
z\longmapsto z+\sum_{n=2}^\infty e^{ia_n}b_n\,z^n
$$
is injective on $\Delta$?
\end{question}

If the answer is affirmative, then we have an alternative (equivalent) formulation of Conjecture \ref{gryanyem bratci}, in the sense that for any $p>2$ and test functions $f,g$ of unit norm in $L^p$ and $L^q$, respectively, we have a real sequence $a$ such that $\Psi_{a,f,g}$ is injective.

\bigskip
Certainly this scheme can be repeated in a much more general setting. 
Assume we are given operators $S_n$ on a Banach space $X$ and positive numbers $c_n$ such that $\nor{S_n}\leqslant Mc_n$ for some absolute $M>0$ and all $n\in\N$. Suppose we want to improve this estimate to $\nor{S_n}\leqslant c_n$. 
Surely we may consider the series
$$
z+\sum_{n=2}^\infty\frac{\sk{S_nx}{\f}}{c_n}\,nz^n\,,
$$
where $x\in X$ and $\f\in X^*$ are such that $\nor{x}=\nor{\f}=1$. 
Yet it is but impossible to expect that such an approach would work in such generality.
One must take into account the special structure of the operator(s) involved (in our case, $T$).

Our situation, however, is somewhat special, for: 
\begin{itemize}
\item
we dealing with powers of a single operator, $T$, which should somehow naturally correspond to power series; 
\item
the conjectured sequence of norms -- $(1/q)_n/(1/p)_n$ -- reminds one of power (hypergeometric) functions, as noted earlier; 
\item
 we also have the property of $T$ being a Fourier multiplier \eqref{infisa}. 
\end{itemize}
Then one can apply the Plancherel theorem for an alternative description of $\Psi$ involving hypergeometric functions, namely
$$
\Psi_{a,p,f,g}(z)=z\left(
1+\sk{G_z\cdot\hat f}{\hat g}_{L^2(\C)}
\right)\,,
$$
where
$$
G_z(\zeta)=G_{z,a,p}(\zeta)=\sum_{n=1}^\infty e^{ia_n}\frac{(1/p)_n}{(1/q)_n}\,(n+1)\left(\frac{\bar\zeta}\zeta\,z\right)^n\,.
$$

\begin{remark}
One may remark that neither the Bellman-function-heat-flow method nor the stochastic methods, that both yield quite good asymptotic estimates for $T$, reproduce the simple fact that $\nor{T}_2=1$. So one may test the observation discussed now first by taking $p=2$ in the definition of $\Psi$ (or $\widetilde\Psi$) and ask whether for given $f,g\in \cS$ with $L^2$-norm equal to one there exists a real sequence $(a_n)_n$ such that
$$
\Psi(z)=z+\sum_{n=2}^\infty e^{ia_n}\sk{T^{n-1}f}{g}\,nz^n\,
$$
is injective. Certainly,  in this case the corresponding $G_z$ simplifies to
$$
G_z(\zeta)=G_{z,a,p}(\zeta)=\sum_{n=1}^\infty e^{ia_n}(n+1)\left(\frac{\bar\zeta}\zeta\,z\right)^n\,.
$$
Hence by defining $h:=\hat f\,\overline{\hat g}$ we have
$$
\aligned
\Psi_{a,p,f,g}(z) & =z\left(
1+\int_\C G_z(\zeta)h(\zeta)\,dA(\zeta)
\right)\\
& = z\left(
1+\sum_{n=1}^\infty e^{ia_n}(n+1)\int_\C \left(\frac{\bar\zeta}\zeta\,z\right)^n h(\zeta)\,dA(\zeta)
\right)\,.
\endaligned
$$
\end{remark}

However, injectivity still turns out to be an elusive property in general.

\subsubsection*{Lack of injectivity}

Indeed, assume $f,g\in\mathcal S$ satisfy 
\begin{equation*}
\label{labung}
\nor{f}_p=\nor{g}_q=1,\ \sk{Tf}{g}\ne 0, \text{  but  }\sk{T^nf}{g}= 0 \text{  for all } n\in\N\backslash\{1\}. 
\end{equation*}
Then
$\Psi(z)=z+2az^2$, where 
\begin{equation*}
\label{quell}
a=\frac{\sk{Tf}{g}}{p-1}\,.
\end{equation*}
It is trivial to verify that polynomials of the form $z+2az^2$ are injective on $\Delta$ (if and) only if $|a|\leqslant 1/4$, which is to say that if $|a|\in(1/4,1]$ then the conclusion of Theorem \ref{bieb} is still true, but the assumption is not.

\begin{example}
Let us construct a concrete $\Psi$ which is not injective.

For $z=(x,y)\equiv x+iy$ define $f(z)=g(z)=\pd_xe^{-\pi|z|^2}=-2\pi xe^{-\pi|z|^2}$. Then $\hat f(z)=2\pi ixe^{-\pi|z|^2}$. Compute 
$$
\aligned
\sk{T^nf}{g}&=\sk{\widehat{T^nf}}{\hat g}=\int_\C e^{-2ni\arg z}\cdot 2\pi ixe^{-\pi|z|^2}\cdot (-2\pi i)\,xe^{-\pi|z|^2}\,dA(z)\\
& =4\pi^2\int_0^\infty r^3e^{-2\pi r^2}\,dr \int_0^{2\pi}e^{-2ni\f}\cos^2\f \,d\f\\
& = \frac12\int_0^{2\pi}e^{-2ni\f}\cos^2\f \,d\f\,.
\endaligned
$$
By writing $\cos\f$ as the average of $e^{i\f}$ and $e^{-i\f}$ we immediately see that the only $n\in\N$ for which the last integral is nonzero is $n=1$. In that case $\sk{Tf}{g}=\frac\pi4$. We still have to normalize functions $f$ and $g$ in $L^p$ and $L^q$, respectively. We calculate
\begin{equation*}
\label{nedjelja}
\aligned
\nor{f}_p^p 
&=\frac{2^p\pi^{\frac{p-1}2}\Gamma\big(\frac{p+1}2\big)}{p^{1+\frac p2}}
\,.
\endaligned
\end{equation*}
It suffices to choose $p=q=2$. Then $\nor{f}_2^2=\nor{g}_2^2=\frac\pi2$.
Consequently, for $f_1=f/\nor{f}_2$ and $g_1=g/\nor{g}_2$, 
$$
a=\sk{Tf_1}{g_1}=\frac12>\frac14.
$$
Hence the corresponding function $\Psi$ is not injective. 
\end{example}

The situation described above, namely when $g$ belongs to the intersection of almost all annihilators $(T^nf)^\bot$, seems rather odd. It might still be that for sufficiently many pairs $f,g$ the function $\Psi_{f,g}$ is injective. 

\subsubsection{More on de Branges' theorem and injectivity}

Here we present an example showing that the function $z+\sum_2^\infty b_n z^n$ may have coefficients satisfying $0<b_n\leq n$, yet not be injective.

\begin{proposition}
\label{segovia - hotel de condes de castilla}
Let $K$ be the Koebe function from Exercise \ref{Koebe}. If $\alpha\in[0,1]$ then 
the convex combination $(1-\alpha)z+\alpha K$ is injective if and only if $\alpha\in\{0\}\cup[2/3,1]$.
\end{proposition}

Before proving the proposition we have to confirm a few simple auxiliary statements.

\begin{lemma}
\label{sam cooke}
If $\Re a,\Re b>1/2$, then 
$$
\bigg|\frac{ab}{1-\bar a-\bar b}\bigg|>\frac12\,.
$$
\end{lemma}
\begin{proof}
Set
$$
\eta:=\frac{1-\bar a-\bar b}{ab}=\frac1a\cdot\frac1b-\frac{\bar a}a\cdot\frac1b-\frac{\bar b}b\cdot\frac1a\,.
$$
By writing $u=1/a$ and $v=1/b$, we see that 
$$
\eta=uv-\frac{u}{\bar u}\cdot v-\frac{v}{\bar v}\cdot u\,.
$$
Multiply this expression by $(\bar u/u)(\bar v/v)$. We obtain
$$
\eta\frac{\bar u}{u}\frac{\bar v}{v}=\bar u\bar v-\bar v-\bar u=\overline{uv-u-v}=\overline{(u-1)(v-1)-1}\,.
$$
Since
\begin{equation}
\label{povela je jelka}
\left|\frac1z-1\right|^2=1-\frac2{|z|^2}\left(\Re z-\frac12\right)
\end{equation}
for $z\in\C\backslash\{0\}$, we have $u,v\in1+\Delta$ and
hence
\[
|\eta|=\bigg|\eta\frac{\bar u}{u}\frac{\bar v}{v}\bigg|=|\underbrace{(u-1)}_{\in\Delta}\underbrace{(v-1)}_{\in\Delta}-1|\leq|u-1||v-1|+1<2\,.
\]
This finishes the proof of the lemma.
\end{proof}

Set $P:=\mn{\zeta\in\C}{\Re\zeta>1/2}$ and, for $\gamma,\lambda\in\C$,
$$
\aligned
\cM_\gamma &=\mn{(a,b)\in P\times P}{a\ne b\text{ {\rm and} } ab(1-a-b)=\gamma}\\
\cN_\lambda &=\mn{(a,b)\in P\times P}{a\ne b\text{ {\rm and} } ab=\lambda (1-\bar a-\bar b)}\,.
\endaligned
$$
The two sets (or equations) are related via the identity
\begin{equation}
\label{kolasinac}
ab(1-a-b)=\frac{ab}{1-\bar a-\bar b}\cdot|1-a-b|^2\,.
\end{equation}

\begin{lemma}
\label{nobody knows the sorrow}
For all $\lambda>1/2$ we have $\cN_\lambda\ne\emptyset$. If $(a,b)\in\cN_\lambda$ then $|1-a-b|>1$.
\end{lemma}

\begin{proof}
Fix $\lambda>1/2$. We will describe all the elements in $\cN_\lambda$ and for each of them verify the last inequality.

We are solving the equation
\begin{equation}
\label{nalepka oceanomania}
ab=\lambda(1-\bar a-\bar b)\,;
\hskip 30pt \Re a,\Re b>1/2\ \text{ and }\ a\ne b\,.
\end{equation}
Write
$$
\aligned
a&=x_1+y_1i\\
b&=x_2+y_2i
\endaligned
$$
for some $x_1,x_2>1/2$ and $y_1,y_2\in\R$ to be determined, and plug into \eqref{nalepka oceanomania}. By considering separately the real and imaginary part we get the system
\begin{equation}
\label{kratke hlace}
\aligned
y_1y_2&=x_1x_2+\lambda(x_1+x_2-1)\\
(x_2-\lambda)y_1+(x_1-\lambda)y_2&=0\,.
\endaligned
\end{equation}
One obvious family of solutions is obtained by taking $x_1=x_2=\lambda$, whereupon the first equation means that $(y_1,y_2)$ lie on the curve $y_1y_2=3\lambda^2-\lambda$ and off the diagonal, i.e.
$$
\aligned
y_1&=c\sqrt{3\lambda^2-\lambda}\\
y_2&=c^{-1}\sqrt{3\lambda^2-\lambda}
\endaligned
$$
for some $c\in\R\backslash\{0,1\}$. Then 
$$
\aligned
|a+b-1|^2
&=(2\lambda-1)^2+(c^2+2+c^{-2})(3\lambda^2-\lambda)\\
&>(2\lambda-1)^2+4(3\lambda^2-\lambda)\\
&=(4\lambda-1)^2\\
&>1\,.
\endaligned
$$

If, alternatively, $x_1-\lambda$ and $x_2-\lambda$ are both nonzero, they must be of different sign, since the first equation in \eqref{kratke hlace} implies that $y_1,y_2$ must be of the same (nonzero) sign. By symmetry of \eqref{nalepka oceanomania} we may thus choose $1/2<x_1<\lambda<x_2$ and write 
$$
\delta:=\sqrt{\frac{\lambda-x_1}{x_2-\lambda}}
\hskip 30pt\text{and}\hskip 30pt 
\eta:=\sqrt{x_1x_2+\lambda(x_1+x_2-1)}\,.
$$
Now \eqref{kratke hlace} reduces to 
$$
\aligned
y_1y_2&=\eta^2\\
y_1/y_2&=\delta^2\,,
\endaligned
$$
which has obvious solutions, namely $(y_1,y_2)=\pm\eta(\delta,1/\delta)$. 
In this case 
$$
\aligned
|a+b-1|^2
&=(x_1+x_2-1)^2+\eta^2(\delta^2+2+\delta^{-2})\\
&>0+(x_1x_2+\lambda(x_1+x_2-1))\cdot4\\
&>4x_1x_2\\
&>1\,.
\endaligned
$$
This finishes the proof.
\end{proof}

\begin{corollary}
\label{norge}
If $\Re a,\Re b>1/2$ and $ab(1-a-b)>0$, then $ab(1-a-b)>1/2$.
\end{corollary}

\begin{proof}
From \eqref{kolasinac} and from our assumption we get that
$$
\lambda:=\frac{ab}{1-\bar a-\bar b}>0\,.
$$
By Lemma \ref{sam cooke} it follows $\lambda>1/2$. Since $(a,b)\in\cN_\lambda$ by the definition of $\lambda$, Lemma \ref{nobody knows the sorrow} and \eqref{kolasinac} finish the proof.
\end{proof}

\begin{lemma}
\label{sukrija}
For every $\gamma>1/2$ we have $\cM_\gamma\ne \emptyset$.
\end{lemma}

\begin{proof}
We saw that some pairs from $\cN_\lambda$ have the same real parts. As we know by \eqref{kolasinac}, the sets $\cN_\lambda$ and $\cM_\gamma$ are related, therefore let us try finding $(a,b)\in\cM_\gamma$ with 
$$
\aligned
a&=x+y_1i\\
b&=x+y_2i
\endaligned
$$
where $x>1/2$ and $y_1,y_2\in\R$, $y_1\ne y_2$, are the quantities we are looking for. 
By considering in $ab(1-a-b)=\gamma$ separately the real and imaginary part we get the system
\begin{equation}
\label{sami khedira}
\aligned
(2x-1)(y_1y_2-x^2)+xy^2&=\gamma\\
y\big[3x^2-x-y_1y_2\big]&=0\,,
\endaligned
\end{equation}
where $y=y_1+y_2$. If $y=0$ then the first equation gives 
$-(2x-1)(x^2+y_1^2)=\gamma$, which is impossible, since $2x-1>0$. Therefore $y\ne 0$, thus $3x^2-x-y_1y_2=0$. Now the first equation in \eqref{sami khedira} can be expressed as
\begin{equation}
\label{loew}
x\big[(2x-1)^2+y^2\big]=\gamma\,.
\end{equation}
Since we require that $y_1\ne y_2$, the inequality between arithmetic and geometric mean gives 
$(y_1+y_2)/2>\sqrt{y_1y_2}$, i.e. $y^2>4(3x^2-x)$, i.e.
\begin{equation}
\label{bravo}
y^2=4(3x^2-x)+\e
\end{equation}
for some $\e>0$. Then \eqref{loew} reads
\begin{equation}
\label{falafel}
x\big[(4x-1)^2+\e\big]=\gamma\,.
\end{equation}
Observe that $\inf\mn{x(4x-1)^2}{x>1/2}=1/2<\gamma$.

So we may finally construct a solution to our problem. First choose $x>1/2$ such that $x(4x-1)^2<\gamma$ 
and then $\e>0$ so that \eqref{falafel} is valid. In view of \eqref{bravo} this determines $y$ (up to a sign). Again by \eqref{bravo} we see that the intersection of the curves $y_1+y_2=y$ and $y_1y_2=3x^2-x$ is nonvoid and lies off the diagonal in the plane $(y_1,y_2)$.
\end{proof}

\begin{proof}[{\bf Proof of Proposition \ref{segovia - hotel de condes de castilla}.}]
Fix $\alpha\in(0,1)$ and write $g(z)=(1-\alpha)z+\alpha K(z)$. Clearly $g$ is holomorphic on $\Delta$ and $g(0)=0$, $g'(0)=1$. Thus if $g$ were injective, it would be {\it schlicht}. We aim to show this is not the case precisely for $\alpha\in(0,2/3)$. We have, for $z,w\in\Delta$,
$$
\aligned
g(z)-g(w)
&=(1-\alpha)(z-w)+\alpha \bigg(\frac z{(1-z)^2}-\frac w{(1-w)^2}\bigg)\\
&=(z-w)\bigg[(1-\alpha)+\alpha\frac{1-zw}{(1-z)^2(1-w)^2}\bigg]\,.
\endaligned
$$ 
We are wondering whether $g(z)-g(w)=0$ for some $z,w\in\Delta$ such that $z\ne w$. Thus we want to find different $z,w\in\Delta$ for which
\begin{equation}
\label{kocholino}
(1-z)^2(1-w)^2+\beta(1-zw)=0\,,
\end{equation}
where 
$$
\beta:=\frac\alpha{1-\alpha}\in(0,\infty)\,.
$$
Write $u=1-z$ and $v=1-w$. Then $u,v\in 1+\Delta$ and \eqref{kocholino} becomes 
$$
(uv)^2+\beta(u+v-uv)=0\,.
$$
Since $u,v\ne 0$ we may divide by $uv$ and get
\begin{equation}
\label{Bra:Ger 1:7}
uv+\beta\bigg(\frac1u+\frac1v-1\bigg)=0\,.
\end{equation}
Now write $a=1/u$, $b=1/v$. Then \eqref{povela je jelka} gives $a,b\in 1/(1+\Delta)=\mn{\zeta\in\C}{\Re\zeta>1/2}=P$ and the equation \eqref{Bra:Ger 1:7} becomes $1/(ab)+\beta(a+b-1)=0$, or
\begin{equation*}
\label{oceanomania}
ab(1-a-b)=\frac1\beta\,.
\end{equation*}
In other words, $g$ is not injective if and only if $\cM_{1/\beta}\ne\emptyset$.
However by Corollary \ref{norge} and Lemma \ref{sukrija} this happens if and only if $\beta<2$, i.e. $\alpha<2/3$.
\end{proof}

\subsubsection{Few other conjectures on Taylor coefficients of holomorphic functions in $\Delta$}
It might be interesting to ask about other possible sets of assumptions on $\Psi$ (other than injectivity) which would imply sharp upper estimates of its Taylor coefficients. We list a few of them. \\

\noindent
{\bf Rogosinski conjecture} \cite[\S 1.3, p. 29]{Gong}.
{\it If 
$$
g(z)=\sum_{n=1}^\infty b_nz^n
$$
is a holomorphic function in $\Delta$ whose image is contained in the image of some schlicht function, 
then $|b_n|\leq n$ for all $n\in\N$.
}\\

\noindent
Clearly the Rogosinski Conjecture implies the Bieberbach Conjecture.\\

\noindent
{\bf Krzy\.z conjecture} \cite{Krz}.
{\it If
$$
f(z)=\sum_{n=0}^\infty a_nz^n
$$
is holomorphic in $\Delta$ and maps $\Delta\rightarrow\Delta\backslash\{0\}$, then for all $n\in\N$ we have
$|a_n|\leq 2/e$. The extremals are precisely the functions of the form $\alpha F(\beta z^n)$ 
where $\alpha,\beta\in\pd\Delta$, $n\in\N$ and
$$
F(z)=\exp\left(\frac{z-1}{z+1}
\right)\,.
$$
}

The conjecture was presented in 1967 in Krak\'ow \cite{Z}. 
See also a very recent paper by Mart\'in, Sawyer, Uriarte-Tuero and Vukoti\'c \cite{MSUTV}.\\

\noindent
{\bf Zalcman conjecture}. 
{\it Any {\it schlicht} function $f(z) = z + \sum_2^\infty a_n z^n$ on $\Delta$ satisfies 
$$
|a_n^2 - a_{2n-1}| \leq (n-1)^2
$$ 
for all $n\in\N\backslash\{1\}$. The equality holds if and only if $f$ is the Koebe function. }\\

The Zalcman conjecture implies the Bieberbach conjecture.


\end{document}